\documentclass[11pt]{article}

\usepackage[T1]{fontenc}
\usepackage{graphicx}
\usepackage{pgfplots}
\usepackage{natbib}
\usepackage{amsfonts}
\usepackage{amsmath}
\usepackage{amssymb}
\usepackage{amsthm}
\usepackage{tgtermes}
\usepackage{newtxtext}
\usepackage{newtxmath}
\usepackage{bm}
\usepackage{dsfont}
\usepackage{url}
\usepackage{fancyhdr}
\usepackage{indentfirst}
\usepackage{enumerate}
\usepackage{titlesec}
\usepackage[misc]{ifsym}
\usepackage{cases}
\usepackage{algorithm}
\usepackage{algpseudocode}
\usepackage{tikz-qtree}
\usepackage{tikz}
\usepackage{tabularx}
\usepackage{caption}
\usepackage{xcolor}
\usepackage[onehalfspacing]{setspace}
\usepackage[colorlinks=true,citecolor=blue,linkcolor=blue,urlcolor=blue]{hyperref}

\bibpunct[, ]{(}{)}{,}{a}{}{,}

\def\bibsep{\smallskipamount}

\theoremstyle{definition}
\newtheorem{example}{Example}

\newtheorem{definition}{Definition}

\newtheorem{assumption}{Assumption}

\theoremstyle{plain}
\newtheorem{theorem}{Theorem}
\newtheorem{lemma}{Lemma}
\newtheorem{proposition}{Proposition}
\newtheorem{corollary}{Corollary}

\theoremstyle{remark}
\newtheorem{remark}{Remark}

\def\N{\mathbb{N}}
\def\P{\mathbb{P}}
\def\p{\mathbb{P}}
\def\E{\mathbb{E}}

\def\R{\mathbb{R}}

\def\d{\mathrm{d}}

\def\bxi{\bm\xi}
\def\bepsilon{\bm\epsilon}
\def\bbeta{\bm\beta}
\def\blambda{\bm\lambda}
\def\balpha{\bm\alpha}
\def\btheta{\bm\theta}

\def\bomega{\bm\omega}
\def\bgamma{\bm\gamma}

\def\bX{\mathbf X}
\def\bY{\mathbf Y}

\def\ba{\mathbf a}
\def\bc{\mathbf c}
\def\bx{\mathbf x}
\def\by{\mathbf y}
\def\bz{\mathbf z}
\def\bt{\mathbf t}
\def\br{\mathbf r}
\def\bs{\mathbf s}
\def\bp{\mathbf p}
\def\bq{\mathbf q}
\def\bv{\mathbf v}
\def\bw{\mathbf w}
\def\bg{\mathbf g}
\def\bu{\mathbf u}

\newcommand{\id}{\mathbf{1}}

\DeclareMathOperator*{\argmin}{arg\,min}
\DeclareMathOperator*{\argmax}{arg\,max}

\renewcommand{\epsilon}{\varepsilon}
\renewcommand{\ge}{\geqslant}
\renewcommand{\le}{\leqslant}
\renewcommand{\geq}{\geqslant}
\renewcommand{\leq}{\leqslant}
\renewcommand{\cdots}{\dots}

\pgfplotsset{compat=1.18}

\title{Harnessing Heterogeneous Data for Conditional Optimization via Optimal Transport}
\author{
Jonathan Yu-Meng Li\thanks{Telfer School of Management, University of Ottawa, Ottawa, Ontario K1N 6N5, Canada. \Letter\ {\scriptsize\url{jonathan.li@telfer.uottawa.ca}}}
\and
Qinyu Wu\thanks{Center for Algorithms, Data, and Market Design at Yale (CADMY), Yale University, New Haven, CT 06511, USA. \Letter\ {\scriptsize\url{qinyu.wu@yale.edu}}}
}
\date{\today}

\begin{document}
\maketitle

\begin{abstract}
Conditional optimization tailors decisions to contextual or event information, but its practical use is often limited by the difficulty of learning the relevant conditional distribution from finite samples of a target joint distribution. This challenge is especially acute when target joint data are scarce or unavailable, or when few observations fall in the conditioning region of interest. Related joint data may be available from multiple sources, such as different stores, markets, populations, or operating environments, but these sources may be biased relative to the target distribution and cannot be pooled naively. We develop a distributionally robust framework based on optimal transport (OT) for harnessing such heterogeneous data in conditional optimization. The framework constructs ambiguity sets over joint distributions using OT distances to empirical source distributions and optimizes worst-case conditional performance over plausible target laws. We propose three OT ambiguity sets that capture different ways of using heterogeneous sources: enforcing simultaneous source consistency, aggregating source discrepancies through weights, and centering the ambiguity set at an OT barycenter. We derive tractable reformulations, establish feasibility conditions, discuss parameter choices, and characterize the relationships among the formulations, revealing trade-offs between robustness, information aggregation, and computational complexity. We demonstrate the value of the framework through a conditional assortment problem using demand and product-feature data from multiple stores.
\end{abstract}


%%%%%%%%%%%%%%%%%%%%%%%%%%%%%%%%%%%%%%%%%%%%%%%%%%%%%%%%%%%%%%%%%%%%%%

% Text of your paper here
\section{Introduction}\label{sec:Intro}

Tailoring decisions to specific conditions--such as minimizing financial risk under tail events, selecting treatments based on patient profiles, or optimizing prices for targeted customer segments--is increasingly central in modern analytics. These problems require customizing decisions based on contextual or event-based information and are naturally formulated as conditional stochastic optimization (CSO) problems of the form:
\begin{align}\label{intro_f}
\min_{\bbeta\in\mathcal D} \;\; \E_Q[\ell(\bbeta,\bxi)\mid\bxi\in \mathcal N],
\end{align}
where $\bbeta$ denotes a decision variable, and the objective evaluates the expected loss conditional on the event $\bxi \in \mathcal N$. 
% {
% \color{black} A typical scenario arises when $\bxi$ can be decomposed as $\bxi = (\bX,\bY)$, where $\bX$ represents the available side information and $\bY$ corresponds to the associated risk. Under this decomposition, the conditional event $\mathcal N$ naturally translates into a subset within the space of side information. Consequently, the corresponding CSO problem can be reformulated as
% \begin{align*}
% \min_{\bbeta\in\mathcal D}  \E_Q[\ell(\bbeta,\bY) \mid \bX \in \mathcal N].
% \end{align*}
% This formulation leverages side information explicitly, allowing for decisions to be tailored more precisely according to observed contextual conditions.
% }
While this formulation is expressive and widely applicable, its practical value depends on learning the relevant conditional law from data generated by the target environment and on optimizing the resulting conditional expectation in a tractable manner. Most data-driven CSO formulations therefore implicitly rely on a single representative dataset from the target joint distribution. This requirement can be restrictive. In many applications, target joint data are scarce or unavailable; even when a target sample exists, few observations may fall in the conditioning region of interest. A firm launching a new product may have no historical demand records, a financial institution evaluating a novel asset class may have only a short return history, and a retailer entering a new market may observe customer behavior that differs materially from its existing markets.

Scarce target data do not mean that no useful information exists. Related joint datasets are often available from other products, markets, stores, populations, or operating environments. The difficulty is that such data are heterogeneous: each source may carry useful signal but may also be biased relative to the target law. Pooling the sources as if they were drawn from the target can produce misleading conditional decisions, whereas ignoring them can leave the decision maker with too little information. This trade-off motivates the question we study: when target joint data are limited, can heterogeneous sources, collected under different conditions or for different purposes, be harnessed to support reliable conditional decisions while preserving tractability and robustness?

Distributionally robust CSO provides a natural starting point for this question. Recent work by \citet{NZBDY20, E22, N24, X25} guards against uncertainty about the conditional distribution by optimizing performance under the worst-case conditional expectation, evaluated over an ambiguity set of plausible joint distributions based on empirical samples $\{\widehat{\bxi}_i\}_{i=1}^N$. This leads to the formulation: 
\begin{align}\label{prob:maincap}
\min_{\bbeta\in\mathcal D}\sup_{Q\in\mathcal B(\{\widehat{\bxi}_i\}_{i=1}^N)} \E_Q[\ell(\bbeta,\bxi)\mid\bxi\in \mathcal N],
\end{align}
where $\mathcal  B(\{\widehat{\bxi}_i\}_{i=1}^N)$ denotes an ambiguity set over joint distributions constructed from the empirical data. These models protect against sampling error and misspecification in the target law, but they are designed around a single empirical reference and therefore do not directly address how several biased but informative source datasets should be combined.

We tackle this challenge by developing an optimal transport (OT)-based framework for incorporating heterogeneous joint data into distributionally robust CSO. Specifically, we consider $K$ datasets $\{ \widehat{\boldsymbol\xi}_{k,i} \}_{i=1}^{n_k}$, each drawn from a source distribution $\mathbb{P}_k$, for $k = 1, \dots, K$. These source distributions may differ from the true target distribution due to bias or distributional shifts, but are assumed to retain useful signal. OT provides a natural way to construct ambiguity sets that couple each empirical source law with the unknown target law: by measuring distributional discrepancy through transportation cost, it quantifies how much each source can inform the target while controlling source-target mismatch.

Our framework can be seen as the conditional counterpart to recent multi-source distributionally robust optimization (DRO) models developed by \cite{REMK24} and \cite{L22}, which aggregate heterogeneous data via OT-based ambiguity sets in unconditional settings. \cite{REMK24} construct ambiguity sets as intersections of OT balls centered at each source's empirical distribution, while \cite{L22} propose two formulations based on Wasserstein barycenters: one using a weighted Wasserstein distance, and the other centering a single Wasserstein ball at the barycenter itself. We propose three conditional formulations that naturally parallel these constructions in the structure of their ambiguity sets, but arise in a fundamentally different setting--one that involves conditioning on events or covariates and requires reasoning over sets of conditional distributions. While the models in \cite{REMK24} and \cite{L22} remain close in form to classical Wasserstein DRO and admit similar duality-based analysis, our setting does not. Conditioning introduces structural complexity and breaks key simplifications available in the unconditional case. As a result, our analysis departs from classical tools and necessitates a new line of development.

Among the three conditional formulations, we adopt the one based on the intersection of OT balls as our baseline, as it most directly uses heterogeneous data by requiring simultaneous consistency with every source. We characterize the general $K$-source conditional ambiguity set by decomposing the joint transport plan into inside and outside components and tracking the exact source-wise transportation cost generated by the outside-$\mathcal N$ mass. This characterization yields a general $K$-source robust counterpart. The two-source norm case remains important, not because it is the only tractable case, but because the outside-cost scalarization admits an explicit line-segment representation, leading to a particularly transparent finite reformulation. We further establish inclusion relationships among the three ambiguity sets, clarifying how the weighted-distance and barycenter formulations trade source-level fidelity for simpler aggregate or single-center structures.

Among the remaining two formulations, we find the weighted OT distance model particularly compelling. Its single aggregate transport budget leads to a compact convex reformulation for general $K$, convex conditioning sets $\mathcal N$, and a broad class of transportation cost functions. Most importantly, when compared to the conditional formulation that centers a single Wasserstein ball at the OT barycenter, the weighted-distance formulation avoids a key pitfall: while the former can collapse in degenerate cases--failing to incorporate useful information from the other sources--the weighted formulation ensures that the influence of all data sources is retained. Taken together, these properties make the weighted OT distance formulation a strong and computationally convenient alternative, and in many cases, our recommended default.

Together, these three formulations form a unified and flexible framework for distributionally robust CSO under data heterogeneity. They offer distinct trade-offs between robustness, tractability, and data utilization, allowing decision-makers to select the approach that best aligns with the structure of their problem and computational budget. Our analysis establishes structural relationships between the formulations, characterizes their feasibility and tractability properties, and clarifies the implications of incorporating multiple data sources under distributional uncertainty. This work lays a comprehensive foundation for robust, data-driven decision-making in settings where conditional distributions must be inferred from diverse and imperfect information.

\subsection{Related Literature}

\textbf{Conditional stochastic optimization.} 
The CSO problem in \eqref{intro_f} addresses problems that explicitly incorporate available side information, often in the form of observed covariates, to enhance decision-making under uncertainty. In this setting, $\bxi$ can be decomposed as $\bxi = (\bX,\bY)$, where $\bX$ represents the available side information and $\bY$ corresponds to the associated risk. The conditioning event $\mathcal N$ naturally translates into a subset within the space of side information. Consequently, the corresponding CSO problem can be reformulated as
\begin{align}\label{prob:covariant}
\min_{\bbeta\in\mathcal D}  \E_Q[\ell(\bbeta,\bY) \mid \bX \in \mathcal N].
\end{align}
% This formulation leverages side information explicitly, allowing for decisions to be tailored more precisely according to observed contextual conditions.
The CSO literature has developed several ways to use covariates when estimating conditional objectives. \cite{H10} used kernel-based methods for stochastic optimization with observable state variables. \cite{BK20} proposed predictive prescriptions, where weights derived from machine learning methods, such as $k$-nearest neighbors, kernel regression, and trees, are used to approximate the conditional expectation in \eqref{prob:covariant}, with asymptotic guarantees for the resulting decisions. \cite{BR19} studied conditional formulations of the newsvendor problem using linear decision rules and kernel density estimators, with attention to regularization and generalization.
Subsequent work has developed richer data-driven decision rules: \cite{K23} proposed forest-based decision rules designed to minimize downstream optimization costs, \cite{B19} introduced residual tree methods for multistage stochastic programming with covariates, and \cite{K25} studied sample-average approximation methods for two-stage stochastic programs with concurrent random variables and covariates.
% One early significant contribution in this area was by \cite{H10}, who introduced kernel-based estimation techniques to systematically model and solve CSO problems. In a seminal work, \cite{BK20} developed the framework of predictive prescriptions, in which weights derived from machine learning methods, such as $k$-nearest neighbors, kernel regression, and trees, are used to approximate the conditional expectation in \eqref{prob:covariant}, and established the asymptotic optimality of the resulting decisions. \cite{BR19} later studied conditional formulations of the newsvendor problem, analyzing the efficacy of linear decision rules coupled with kernel density estimators. They also addressed the regularization of generalization error.
% In further developments, \cite{K23} proposed the use of forest-based decision rules specifically designed to minimize optimization costs within CSO problems. \cite{B19} expanded on these ideas by introducing residual tree methods, effectively tackling multi-stage stochastic programming challenges where covariate data informs decision stages. Additionally, data-driven methods for solving two-stage stochastic problems using Sample Average Approximation (SAA) frameworks with concurrent random variables and covariates were explored by \cite{K25}.
Recent studies have also shifted focus toward risk-sensitive formulations. For instance, \cite{LCLS22} considered a CSO problem that involves  conditional Value-at-Risk constraints within the newsvendor model, providing robust solutions that explicitly account for downside risk conditional on covariates. We refer to \cite{SCD25} for a comprehensive survey of contextual optimization methods for decision-making under uncertainty. 
% While many of these works typically follow an “estimate-then-optimize” paradigm (Srivastava et al., 2021; Hu et al., 2022), there has also been a growing interest in developing integrated end-to-end frameworks for CSO, as discussed by Elmachtoub and Grigas (2022) and Donti et al. (2017).

\noindent\textbf{Distributionally robust CSO.} 
To handle uncertainty and model misspecification in conditional problems, DRO methods based on OT ambiguity sets have recently emerged as prominent tools. 
%Under the setting that $\bxi=(\bX,\bY)$, where $\bX$ and $\bY$ are the vector of side information and the associate risk vector, respectively,
For example, \cite{BMS23} examined a Wasserstein distributionally robust formulation of \eqref{prob:covariant} in which the conditioning event is restricted to a single point, $\bX=\bx_0$. They first estimate the conditional distribution of $\bY\mid \bX=\bx_0$ and then construct a type-$1$ Wasserstein ball around this estimated conditional law, an approach referred to by \cite{X25} as a ``predict-then-robustify'' scheme.
%utilized a type-$1$ Wasserstein ball centered at the estimated conditional distribution of  $\bY| \bX=\bx_0$ as the ambiguity set, and the conditional events in the objective functional are restricted to a single point $\bX=\bx_0$. 
In contrast, \cite{NZBDY20,N24} used the joint empirical distribution of $\bxi$ as the reference distribution and constructed OT ambiguity sets around this joint law. In these models, the conditioning event is specialized to a closed ball in the space of side information, and the transportation cost is decomposed between the covariate component $\bX$ and the outcome component $\bY$. Rather than directly solving the distributionally robust CSO problem \eqref{prob:maincap}, \cite{E22} introduced an unconditional trimming-based relaxation, which simplifies both analysis and computation. More recently, \cite{X25} studied conditional risk minimization by characterizing the conditional ambiguity sets induced by OT ambiguity sets as unions of multiple OT balls, encompassing the joint-distribution formulations of \cite{NZBDY20,N24}, the trimming construction of \cite{E22}, and predict-then-robustify schemes as special instances.
%More recently, \cite{X25} studied conditional risk minimization by characterizing the conditional ambiguity sets induced by OT ambiguity sets as unions of multiple OT balls. This induced-set perspective unifies the joint-distribution formulations of \cite{NZBDY20,N24}, the trimming construction of \cite{E22}, and predict-then-robustify schemes based on estimated conditional distributions. It also yields explicit tractable reformulations through strong duality results \citep{EK18,BM19} or regularized equivalence formulations \citep{W22} within OT-based DRO. 
Beyond OT-based constructions, \cite{KBL24} integrated machine learning prediction models into DRO by building Wasserstein and $\phi$-divergence ambiguity sets from regression residuals, with asymptotic and finite-sample guarantees.

% In contrast, \cite{NZBDY20} employed a joint empirical distribution of $\bxi$ as the reference distribution and constructed an ambiguity set via the type-$\infty$ Wasserstein ball, where the set of conditioning events was specialized  as a closed ball in the space of side information.
% Their transportation cost function was explicitly decomposed between the components $\bX$ and $\bY$.
% Extending the framework of \cite{NZBDY20}, \cite{N24} provided a comprehensive and unified approach by combining predictive modeling and decision-making into a single integrated problem formulation. They generalized transportation cost functions to construct a broader and more flexible class of distributional ambiguity sets. Subsequently, \cite{X25} characterized the conditional ambiguity sets proposed in \cite{N24} as unions of multiple OT-based balls. This characterization enables obtaining explicit, tractable reformulations through strong duality results \citep{EK18,BM19} or regularized equivalence formulations \citep{W22} within OT-based DRO.
% Rather than directly solving the distributionally robust CSO problem \eqref{prob:maincap}, \cite{E22} introduced an unconditional relaxation of the original problem by employing a trimming approach, which significantly simplifies the analysis and solution process. Beyond OT-based constructions, \cite{KBL24} integrated machine learning prediction models into DRO by building Wasserstein and $\phi$-divergence ambiguity sets from regression residuals, with asymptotic and finite-sample guarantees.

\noindent\textbf{Multi-source data.}
Multi-source learning problems have attracted considerable attention across statistics and machine learning because they integrate heterogeneous data sources. A central challenge is how to aggregate decision rules or objectives across environments with significant heterogeneity.  \cite{DN18} proposed minimax statistical learning over group-wise distributions. \cite{SKHL19} and \cite{Z20} explored sample reweighting and invariant learning to align source-specific risks. In the statistics literature, transfer learning provides guarantees for exploiting biased yet related source data \citep{B21,LCL22}. Within operations research, \cite{GK22} showed that pooling data across many stochastic optimization problems can strictly outperform solving each problem in isolation, while \cite{BMM25} analyzed the performance of data-driven decisions when historical samples are drawn from heterogeneous environments that deviate from the target distribution. These works, however, do not address conditional decision-making. 

Our work lies at the intersection of distributionally robust CSO and multi-source data. The closest related work is \cite{WCW24}, who also use an intersection of Wasserstein balls in a contextual setting. However, their formulation follows a predict-then-robustify scheme: the balls are centered at parametric and nonparametric estimates of the conditional law $\bY\mid \bX=\bx_0$ constructed from a single dataset, and the resulting DRO problem is then solved over outcome distributions. In contrast, our formulation is directly built from heterogeneous joint source datasets: the empirical joint source distributions themselves serve as references, which retains source-level joint information and avoids committing to a single estimated conditional law before optimization. Conditioning is imposed within the robust optimization problem itself, leading to induced conditional ambiguity sets and requiring new tractability analyses for intersection, weighted-distance, and barycenter-based formulations.

\subsection{Notations}
For a Borel set $\Xi\subseteq \R^d$, denote by $\mathcal M_+(\Xi)$ and $\mathcal P(\Xi)$ the set of all finite Borel measures on $\Xi$ and the set of all probability measures in $\mathcal M_+(\Xi)$, respectively. {\color{black}We use the terms probability measure, distribution, and law interchangeably.}
In our setting, $\Xi=\mathcal X\times \mathcal Y$, where $\mathcal Y\subseteq \R^{d_{\mathcal Y}}$ and $\mathcal X\subseteq \R^{d_{\mathcal X}}$ are the spaces of the associated risk vector and the side-information vector, respectively. 
% For $\mathcal N\subseteq\Xi$ and $a,b\in[0,1]$, we define $Q_{\mathcal N}^{[a,b]}=\{Q\in\mathcal P(\Xi): Q(\mathcal N)\in[a,b]\}$, and simplify let $Q_{\mathcal N}=\{Q\in\mathcal P(\Xi): Q(\mathcal N)=1\}$.
We use $\delta_{\bm\omega}$ to represent the Dirac measure corresponding to $\bm\omega\in\Xi$, i.e., $\delta_{\bm\omega}(E)=1$ if $\bm\omega\in E$ and $\delta_{\bm\omega}(E)=0$ otherwise. 
For a measure $\nu\in\mathcal M_+(\Xi)$ and $\mathcal N\subseteq \Xi$, define
%denote by $\nu^{\mathcal X}$ and $\nu^{\mathcal Z}$ the marginal measures of $\nu$ on $\mathcal X$ and $\mathcal Z$, respectively, and for $\mathcal N\subseteq \mathcal X$, let
\begin{align}\label{eq-cddf}
\nu^{\mathcal N}(A)=\frac{\nu(\mathcal N\cap A)}{\nu(\mathcal N)},~~A\subseteq \Xi
\end{align}
as the conditional measure of $\nu$ on the set $\mathcal N$. 
%It is clear that $\nu^{\mathcal N}(\mathcal N)=\nu^{\mathcal N}(\Xi)$.
%Suppose that $(\bX,\bZ)\sim \nu$. Then, $\nu^{\mathcal X} (\cdot|\mathcal N)$ and $\nu^{\mathcal Z|\mathcal N}(\cdot)$ represent the measures of $\bX$ and $\bZ$ conditional on $\bX\in\mathcal N$, respectively. 
%It is clear that $\nu^{\mathcal X} (\cdot|\mathcal N)\in\mathcal P(\mathcal X)$ and $\nu^{\mathcal Z|\mathcal N} (\cdot)\in\mathcal P(\mathcal Z)$.
For $\nu\in\mathcal M_+(\Xi)$, $\E_{\nu}[\cdot]$ denotes the integral under $\nu$, and if $\nu$ is a probability measure, then $\E_{\nu}[\cdot]$ is an expectation. 
%For a law-invariant mapping $\rho$,\footnote{A mapping $\rho$ is law-invariant if $\rho(X)=\rho(Y)$ whenever random variables $X$ and $Y$ have the same distribution.} we use $\rho_{P}(\cdot)$ to represent the value  of $\rho$ under $P\in\mathcal P(\Xi)$.
For  $\mathcal N\subseteq \Xi$, we denote $\mathcal N^c$, ${\rm cl}(\mathcal N)$ and ${\rm conv}(\mathcal N)$ as the complement, closure and convex hull of $\mathcal N$, respectively.
% We write $\id_{E}$ as the indicator function of the set $E$, i.e., $\id_{E}(\bm\omega)=1$ if $\bm\omega\in E$ and $\id_{E}(\bm\omega)=0$ otherwise. 
Denote by $[N]=\{1,\dots,N\}$ for $N\in\N$.

% Let $\{\p_k\}_{k\in[K]}$ be the set of source distributions. 
% For each $\p_k$, the available data consist of $N_k$ samples, and each sample is denoted by $\widehat{\bm\xi}_{k,j}=(\widehat{\mathbf x}_{k,j},\widehat{\mathbf y}_{k,j})\in \mathcal X\times \mathcal Y$ with a probability $w_{k,j}$, and thus, the empirical distributions are given by
% \begin{align*}
% \widehat{\p}_k=\sum_{j=1}^{N_k}w_{k,j}\delta_{\widehat{\bxi}_{k,j}}~~\forall k\in[K].
% \end{align*}
% Define $\mathcal A=\times_{k=1}^K [N_k]$ as the set of $K$-dimensional multi-indices. Each $\balpha\in\mathcal A$ uniquely identifies a combination of atoms from $\{\widehat{\p}_k\}_{k\in[K]}$. 
% We write $\delta_{\balpha}$ as a shorthand for the Dirac distribution on $\Xi^K$ that concentrates unit mass at $(\widehat{\bxi}_{1, \alpha_1}, \ldots, \widehat{\bxi}_{K, \alpha_K})$.

%\newpage

%%%%%%%%%%%%%%%%%%%%%%%%%%%%%%%%%%%%%%%%%%%%%%%%%%%%%%%%%%%
\section{Conditional Optimization with Heterogeneous Data} \label{sec:framework}
%For each $k \in [K]$, we denote by $\widehat{\mathbb{P}}_k = \frac{1}{N_k} \sum_{i=1}^{N_k} \delta_{\widehat{\boldsymbol{\xi}}_{k,i}}$ the empirical distribution constructed from $N_k$ samples drawn from the source distribution $\mathbb{P}_k$. 

We now specialize the conditional decision problem in \eqref{intro_f} to the multi-source data setting studied in this paper. The target joint distribution is observed only through limited data, or not observed directly, while related joint data are available from multiple sources.

Let $\{\p_k\}_{k\in[K]}$ be the set of source distributions. 
For each $\p_k$, the available data consist of $N_k$ samples, where sample $j$ is denoted by $\widehat{\bm\xi}_{k,j}=(\widehat{\mathbf x}_{k,j},\widehat{\mathbf y}_{k,j})\in \mathcal X\times \mathcal Y$ and has weight $w_{k,j}$. The empirical distributions are
\begin{align*}
\widehat{\p}_k=\sum_{j=1}^{N_k}w_{k,j}\delta_{\widehat{\bxi}_{k,j}},\quad  k\in[K].
\end{align*}
These empirical source laws serve as references for plausible target joint laws. We next formulate the corresponding multi-source conditional DRO problem and introduce three optimal-transport ambiguity sets, each encoding a different way to use the heterogeneous sources.

\subsection{A Multi-Source Distributionally Robust Formulation}

% To formalize distributionally robust conditional models that incorporate data from multiple sources, we begin by recalling the definition of optimal transport (OT) and the standard ambiguity set constructed around a single distribution.

% \begin{definition}[Optimal transport cost]
% Let $c : \Xi \times \Xi \to \mathbb{R}_+$ be a \emph{cost function}, assumed to be continuous and such that $c(\boldsymbol{\xi}_1, \boldsymbol{\xi}_2) = 0$ if and only if $\boldsymbol{\xi}_1 = \boldsymbol{\xi}_2$. The optimal transport (OT) cost between two distributions $P$ and $Q$ supported on $\Xi$ is defined as
% $$
% W(P, Q) := \inf_{\pi \in \Pi(P, Q)} \mathbb{E}_{\pi}\left[c(\boldsymbol{\xi}_1, \boldsymbol{\xi}_2)\right],
% $$
% where $\Pi(P, Q)$ denotes the set of all couplings on $\Xi \times \Xi$ with marginals $P$ and $Q$, respectively.
% \end{definition}

% For any $P\in\mathcal P(\Xi)$, we define the ambiguity set as the OT ball centered at $P$ with radius $\epsilon$:\begin{align*}
% \mathcal B\left(P,\epsilon\right)=\left\{Q\in \mathcal P(\Xi):~W(Q,P)\le \epsilon\right\}.
% \end{align*}
% %where $W(Q,P)$ denotes the optimal transport distance between $Q$ and $P$. 
% The corresponding open OT ball is denoted as $\mathcal B^{\circ}(P,\epsilon)$ with the form:
% \begin{align*}%\label{eq-openball}
% \mathcal B^{\circ}(P,\epsilon)=\left\{Q\in \mathcal P(\Xi):~W(Q,P)<\epsilon\right\}.
% \end{align*}

Replacing the single empirical reference in \eqref{prob:maincap} with the collection of empirical source laws gives the multi-source distributionally robust conditional optimization problem
\begin{align}\label{prob-mainprob}
\min_{\bbeta \in \mathcal{D}} \;
\sup_{\substack{Q \in \mathcal{B}\left(\{ \widehat{\p}_k \}_{k \in [K]}, \boldsymbol{\epsilon} \right) \\ Q(\mathcal{N}) \in [a, b]}} 
\E_Q[\ell(\bbeta, \boldsymbol{\xi}) \mid \boldsymbol{\xi} \in \mathcal{N}].
\end{align}
Here $\mathcal B(\{\widehat{\p}_k\}_{k\in[K]},\bepsilon)$ is an ambiguity set constructed from the empirical source laws, and the budget parameter $\bepsilon$ controls its size. {\color{black}Throughout, we assume that the loss function $\bxi\mapsto \ell(\bbeta,\bxi)$ is real-valued, measurable and upper semicontinuous for all $\bbeta\in\mathcal D$.}
The condition $Q(\mathcal N)\in[a,b]$, where $0<a\le b\le 1$, restricts attention to joint laws assigning a plausible amount of mass to the conditioning event; the positive lower bound ensures that the conditional expectation is well defined.

% Note that for any $Q\in\mathcal P(\Xi)$,
% the value of $\E_Q[\ell(\bbeta,\bxi)|\bxi\in \mathcal N]$ is determined by the conditional distribution $Q^{\mathcal N}\in\mathcal P(\Xi)$ that is defined in \eqref{eq-cddf} with the specific form:
% $$
% Q^{\mathcal N}(A)=\frac{Q(A\cap \mathcal N)}{Q(\mathcal N)},~~A\subseteq \Xi.
% $$ 
% This observation motives us to consider a subset of $\mathcal P(\Xi)$:
% \begin{align}\label{eq-conset}
% \mathcal B_{\mathcal N}\left(\{\widehat{\p}_k\}_{k\in[K]},\bm\epsilon\right)
% =\left\{Q^{\mathcal N}\in\mathcal P(\Xi):~Q\in \mathcal B\left(\{\widehat{\p}_k\}_{k\in[K]},\bm\epsilon\right),~Q(\mathcal N)\in[a,b]\right\}.
% \end{align}
% The problem in \eqref{prob-mainprob} can be equivalently reformulated as the following unconditional problem, where the ambiguity set is specified in \eqref{eq-conset}:
% \begin{align}\label{prob:unconditional}
% \min_{\bbeta\in\mathcal D}
% \sup_{Q\in\mathcal B_{\mathcal N}\left(\{\widehat{\p}_k\}_{k\in[K]},\bm\epsilon\right)} \E_Q[\ell(\bbeta,\bxi)].
% \end{align}

We use optimal transport to measure discrepancies between candidate target joint laws and empirical source laws. We recall its definition and the corresponding single-source OT ball.

\begin{definition}[Optimal transport cost]
Let $c : \Xi \times \Xi \to \mathbb{R}_+$ be a \emph{cost function}, assumed to be {\color{black}measurable}, lower semicontinuous and such that $c(\boldsymbol{\xi}_1, \boldsymbol{\xi}_2) = 0$ if and only if $\boldsymbol{\xi}_1 = \boldsymbol{\xi}_2$. The optimal transport (OT) cost between two distributions $P$ and $Q$ supported on $\Xi$ is defined as
$$
W(P, Q) := \inf_{\pi \in \Pi(P, Q)} \mathbb{E}_{\pi}\left[c(\boldsymbol{\xi}_1, \boldsymbol{\xi}_2)\right],
$$
where $\Pi(P, Q)$ denotes the set of all couplings on $\Xi \times \Xi$ with marginals $P$ and $Q$, respectively.
\end{definition}

For any $P\in\mathcal P(\Xi)$, the OT ball centered at $P$ with radius $\epsilon$ is defined as
\begin{align*}
\mathcal B\left(P,\epsilon\right)=\left\{Q\in \mathcal P(\Xi):~W(Q,P)\le \epsilon\right\}.
\end{align*}

%and the corresponding open OT ball is
% $\mathcal B^{\circ}(P,\epsilon)=\left\{Q\in \mathcal P(\Xi):~W(Q,P)<\epsilon\right\}$.
%where $W(Q,P)$ denotes the optimal transport distance between $Q$ and $P$. 
% The corresponding open OT ball is denoted as $\mathcal B^{\circ}(P,\epsilon)$ with the form:
% \begin{align*}%\label{eq-openball}
% \mathcal B^{\circ}(P,\epsilon)=\left\{Q\in \mathcal P(\Xi):~W(Q,P)<\epsilon\right\}.
% \end{align*}

%%%%%%%%%%%%%%%%%%%%%%%%%%%%%%%%%%%%%%%%%%%%%%%%%%%%%%%%%%%
\subsection{Optimal-Transport Ambiguity Sets}
The modeling choice in \eqref{prob-mainprob} is the ambiguity set $\mathcal{B}(\{\widehat{\p}_k\}_{k\in[K]},\bepsilon)$. We consider three optimal-transport constructions, corresponding to three ways of using heterogeneous sources: preserving source-by-source consistency, aggregating source discrepancies, and summarizing the sources through a representative law.

\vspace{1ex}
\noindent\textbf{Intersection of OT Balls.}  
This ambiguity set is defined as the intersection of $K$ OT balls, each centered at an empirical distribution:
\begin{align}\label{eq-defI}
\mathcal{B}^{\mathrm{I}}\left(\{\widehat{\p}_k\}_{k \in [K]}, \boldsymbol{\epsilon} \right) := \bigcap_{k \in [K]} \mathcal{B}(\widehat{\p}_k, \epsilon_k),
\end{align}
where $\boldsymbol{\epsilon} = (\epsilon_1, \ldots, \epsilon_K)$ is a vector of transport radii.
This formulation imposes source-by-source consistency: a plausible target law must be close to every empirical source distribution.

\vspace{1ex}
\noindent\textbf{Weighted OT Distance.}  
This ambiguity set aggregates distances to each source using a weighted sum:
\begin{align}\label{eq-defII}
\mathcal{B}_{\btheta}^{\mathrm{II}}\left(\{\widehat{\p}_k\}_{k \in [K]}, \epsilon \right)
:= \left\{ Q \in \mathcal{P}(\Xi) \,:\, \sum_{k = 1}^K \theta_k W(Q, \widehat{\p}_k) \leq \epsilon \right\},
\end{align}
where $\boldsymbol{\theta} \in \mathbb{R}_+^K$ is a vector of non-negative weights. For simplicity, we omit the subscript $\btheta$ and write $\mathcal B^{\rm II}$ when referring to the weighted-distance balls and their induced conditional sets throughout the paper.
This formulation imposes aggregate consistency: discrepancies from different sources can trade off through the weights.

\vspace{1ex}
\noindent\textbf{OT Barycenter Ball.}  
This ambiguity set is centered at an OT barycenter of the source distributions. Let $\widehat{\mathbb{P}}_{\mathrm{bar}}$ be any minimizer of
\begin{align*}%\label{def:defIII}
\widehat{\mathbb{P}}_{\mathrm{bar}} \in \argmin_{P \in \mathcal{P}(\Xi)} \sum_{k=1}^K \theta_k W(P, \widehat{\p}_k),
\end{align*}
with the same weighting vector $\boldsymbol{\theta}$. The ambiguity set is then
$$
\mathcal{B}^{\mathrm{III}}\left(\{\widehat{\p}_k\}_{k \in [K]}, \epsilon \right)
:= \left\{ Q \in \mathcal{P}(\Xi) \,:\, W(Q, \widehat{\mathbb{P}}_{\mathrm{bar}}) \leq \epsilon \right\}.
$$
This formulation imposes consistency with a source summary: the sources are compressed into a representative joint law, around which a standard single-center OT ball is constructed.

%%%%%%%%%%%%%%%%%%%%%%%%%%%%%%%%%%%%%%%%%%%%%%%%%%%%%%%%%%
\subsection{Modeling Trade-offs Among the Three Formulations}
The three ambiguity sets differ in how much source-level information they preserve. The intersection formulation is the most direct multi-source model because it requires simultaneous consistency with all sources. The weighted-distance formulation relaxes this requirement by allowing discrepancies across sources to trade off. The barycenter formulation goes further by replacing the sources with a single representative distribution, which is parsimonious but may lose heterogeneity.

\vspace{1ex}
\noindent\textbf{Intersection versus weighted distance.}
To relate the parameters in \eqref{eq-defII} to those in \eqref{eq-defI}, we propose a principled scheme for selecting the weights $\bm\theta = \{\theta_1, \dots, \theta_K\}$ and common radius $\epsilon$ based on the vector $\bm\epsilon = \{\epsilon_1, \dots, \epsilon_K\}$. Let
\begin{align}\label{eq-epsilonmin}
\epsilon_{\min} = \min\left\{ \epsilon_k : k \in [K] \right\},
\end{align}
and define
\begin{align}\label{eq-theta}
\theta_k = \frac{\epsilon_{\min}}{\epsilon_k}, \quad  k \in [K].
\end{align}
This choice assigns larger weights to sources with smaller prescribed radii, aligning the weighted formulation with the trust encoded by the intersection benchmark.

\begin{proposition}\label{prop:weightset}
Let $\epsilon_k>0$ for $k\in[K]$, and let $\epsilon_{\min}$ and $\theta_{k}$ be defined in \eqref{eq-epsilonmin} and \eqref{eq-theta}, respectively. The following relations hold:
\begin{align*}
\mathcal B^{\rm{II}}\left(\{\widehat{\p}_k\}_{k\in[K]},\epsilon_{\min}\right)\subseteq\mathcal B^{\rm I}\left(\{\widehat{\p}_k\}_{k\in[K]},\bm\epsilon\right)\subseteq \mathcal B^{\rm{II}}\left(\{\widehat{\p}_k\}_{k\in[K]},K\epsilon_{\min}\right).
\end{align*}
% As a result, we have
% \begin{align*}
% \mathcal B_{\mathcal N}^{\rm{II}}\left(\{\widehat{\p}_k\}_{k\in[K]},\epsilon_{\min}\right)\subseteq\mathcal B_{\mathcal N}^{\rm I}\left(\{\widehat{\p}_k\}_{k\in[K]},\bm\epsilon\right)\subseteq \mathcal B_{\mathcal N}^{\rm{II}}\left(\{\widehat{\p}_k\}_{k\in[K]},K\epsilon_{\min}\right).
% \end{align*}
\end{proposition}

Thus, the weighted-distance formulation can be calibrated around the intersection benchmark, but it enforces aggregate rather than source-by-source consistency.

The next example shows why this distinction matters: when multiple sources are informative about the target law, simultaneous consistency can yield a sharper ambiguity set.

\begin{example}[Intersection vs. weighted distance]\label{ex:compare12}
Suppose in this example that the cost function $c$ is a metric on $\Xi$ (e.g., a norm), so that the OT cost $W$ is jointly convex and satisfies the triangle inequality. Let $P_1$ and $P_2$ denote two source distributions with $W(P_1,P_2)>0$, and suppose that the true distribution is given by $P^* = (P_1 + P_2)/2$. Since $P_1$ and $P_2$ play symmetric roles in the construction of $P^*$, it is natural to set $\epsilon_1 = \epsilon_2$ for $\mathcal B^{\rm I}(\{P_1, P_2\}, \bepsilon)$. By our construction of the weights $\btheta = \{\theta_1, \theta_2\}$ in $\mathcal B^{\rm II}(\{P_1, P_2\}, \epsilon)$, it holds that $\theta_1 = \theta_2 = 1$.
% A central motivation for using multi-source data is to reduce conservativeness by leveraging additional information when the true distribution $P^*$ is guaranteed to lie within the ambiguity set. 
%Thus, it is natural that when , a smaller radius implies a less conservative and more informative model.
The smallest common radius for which $\mathcal B^{\rm I}(\{P_1, P_2\}, \bepsilon)$ contains $P^*$ is $W(P_1,P_2)/2$, whereas the smallest radius for which $\mathcal B^{\rm II}(\{P_1, P_2\}, \epsilon)$ contains $P^*$ is $W(P_1,P_2)$. When these minimal radii are adopted, Proposition~\ref{prop:weightset} shows that $\mathcal B^{\rm II}(\{P_1, P_2\}, \epsilon)$ is larger than $\mathcal B^{\rm I}(\{P_1, P_2\}, \bepsilon)$.
In particular, $P_1$ and $P_2$ lie outside $\mathcal B^{\rm I}(\{P_1, P_2\}, \bepsilon)$ but inside $\mathcal B^{\rm II}(\{P_1, P_2\}, \epsilon)$. Thus, compared with the intersection formulation, the weighted-distance formulation may be more conservative by admitting distributions that remain close in aggregate but are not simultaneously consistent with all sources.
\end{example}

\vspace{1ex}
\noindent\textbf{Weighted distance versus barycenter.}
The weighted-distance and barycenter formulations are both connected to an OT barycenter in the small-radius limit, but they use this information differently. The weighted-distance formulation retains distances to each source, whereas the barycenter formulation replaces the sources by a single representative law. This compression can simplify the model, but it may also discard source-specific information. The following example illustrates the trade-off and shows that, when the true distribution lies within both ambiguity sets, the weighted-distance set can be less conservative than the barycenter-based set.

\begin{example}[Weighted distance vs. barycenter]
Consider the case where $K=2$ and the cost function is $c(\bxi_1,\bxi_2)=\|\bxi_1-\bxi_2\|$ for some norm $\|\cdot\|$ on $\Xi$. By the triangle inequality, if one barycenter weight is strictly larger than the other, the barycenter objective is minimized at the more heavily weighted source distribution. Thus, when $\theta_1>\theta_2$, the barycenter is $P_1$; when $\theta_2>\theta_1$, it is $P_2$. This illustrates how the barycenter formulation can discard source-level heterogeneity. 
Assume now that $\theta_1>\theta_2$ and the true distribution is $P^*=P_2$. In this case, the barycenter becomes $P_1$, and thus $\mathcal B^{\rm III}(\{P_1,P_2\},\epsilon)=\{Q:W(Q,P_1)\le \epsilon\}$. The minimal radius making $\mathcal B^{\rm II}(\{P_1, P_2\}, \epsilon)$ contain $P^*$ is $\theta_1W(P_1, P_2)$, whereas the corresponding radius for $\mathcal B^{\rm III}(\{P_1, P_2\}, \epsilon)$ is $W(P_1, P_2)$. With these respective minimal radii, the ambiguity sets become
\begin{align*}
\mathcal B^{\rm II}\left(\{P_1, P_2\}, \theta_1 W(P_1,P_2)\right)=\{Q: \theta_1 W(Q,P_1)+\theta_2 W(Q,P_2)\le \theta_1 W(P_1,P_2)\}
\end{align*}
and 
\begin{align*}
\mathcal B^{\rm III}(\{P_1,P_2\},W(P_1,P_2))&=\{Q:W(Q,P_1)\le W(P_1,P_2)\}\\
&=\{Q:\theta_1 W(Q,P_1)\le \theta_1 W(P_1,P_2)\}.
\end{align*}
Hence, $\mathcal B^{\rm II}\left(\{P_1, P_2\}, \theta_1 W(P_1,P_2)\right)\subseteq \mathcal B^{\rm III}(\{P_1,P_2\},W(P_1,P_2))$, showing that the barycenter-based formulation may be more conservative than the weighted-distance formulation.
\end{example}

These comparisons motivate the intersection formulation as the benchmark multi-source model. It preserves simultaneous consistency with all sources, but this fidelity also makes the model technically more challenging. The next section studies its tractability.

%\newpage
%%%%%%%%%%%%%%%%%%%%%%%%%%%%%%%%%%%%%%%%%%%%%%%%%%%%%%%%%%
\section{Tractable Reformulation of the Intersection Model}\label{sec:tractbilityB1}\label{sec:tractbilityB1_general}

This section establishes a tractable counterpart for the distributionally robust conditional optimization problem \eqref{prob-mainprob} under the intersection ambiguity set \eqref{eq-defI}. The key difficulty is that the OT constraints are imposed on a joint target law, whereas performance is evaluated only through the conditional law induced on $\mathcal N$. Thus, the reformulation must keep track of how probability mass on $\mathcal N$ and on $\mathcal N^c$ share the same source-wise transport budgets.

We first express the conditional problem directly through the conditional laws induced by feasible joint laws. Since the objective depends on a joint law only through its conditional law on $\mathcal N$, the intersection model is equivalent to
\begin{align}
\min_{\bbeta\in\mathcal D}
\sup_{Q\in\mathcal B_{\mathcal N}^{\rm I}\left(\{\widehat{\p}_k\}_{k\in[K]},\bm\epsilon\right)}
\E_Q[\ell(\bbeta,\bxi)],
\label{prob-intersection-induced}
\end{align}
where the induced conditional ambiguity set is
\begin{align}
\mathcal B_{\mathcal N}^{\rm I}\left(\{\widehat{\p}_k\}_{k\in[K]},\bm\epsilon\right)
:=\left\{Q^{\mathcal N}: Q\in\mathcal B^{\rm I}\left(\{\widehat{\p}_k\}_{k\in[K]},\bm\epsilon\right),\ Q(\mathcal N)\in[a,b]\right\},
\label{eq-BI-N}
\end{align}
and $Q^{\mathcal N}$ is the conditional measure defined in \eqref{eq-cddf}. The induced set $\mathcal B_{\mathcal N}^{\rm I}$ is therefore the object to characterize: it contains exactly the conditional laws that can arise from a joint law satisfying all source-wise OT constraints and the mass bound $Q(\mathcal N)\in[a,b]$. This characterization entails a decomposition principle analogous to the one used for multi-source DRO in the unconditional setting \citep{REMK24}. Let $\mathcal A$ be the set of all multi-indices $\balpha=(\alpha_1,\ldots,\alpha_K)$ with $\alpha_k\in[N_k]$ for each $k\in[K]$. Each $\balpha\in\mathcal A$ selects one observation from each source and forms the tuple $(\widehat\bxi_{1,\alpha_1},\ldots,\widehat\bxi_{K,\alpha_K})$. The theorem below uses this indexing to separate the mass kept on $\mathcal N$ from the mass placed on $\mathcal N^c$, while preserving the source-wise transport budgets.

\begin{theorem}[Exact characterization of the intersection conditional set]\label{th-analysisB1}
The intersection-based conditional set defined in \eqref{eq-BI-N} can be represented as follows:
\begin{align}\label{eq-mainthB1-exact}
&\mathcal B_{\mathcal N}^{\rm I}\left(\{\widehat{\p}_k\}_{k\in[K]},\bm\epsilon\right)\notag\\
&=
\left\{
\begin{array}{ll}
Q: Q\in\mathcal P(\Xi),~Q(\mathcal N)=1,~{\rm and~there~exist}~m,\\
\{a_{k,j}\}_{j\in[N_k],k\in[K]},~\{r_{\balpha}\}_{\balpha\in\mathcal A},~{\rm and}~\{\bs_{\balpha}\}_{\balpha\in\mathcal A}~{\rm such~that}\\
m\in[1/b,1/a],\\
a_{k,j}\ge 0 & \forall j\in[N_k],~k\in[K],\\
r_{\balpha}\ge 0 & \forall \balpha\in\mathcal A,\\
\sum_{j=1}^{N_k}a_{k,j}=1 & \forall k\in[K],\\
a_{k,j}+\sum_{\balpha\in\mathcal A:\alpha_k=j}r_{\balpha}=m w_{k,j} & \forall j\in[N_k],~k\in[K],\\
\bs_{\balpha}:=(s_{\balpha,1},\ldots,s_{\balpha,K})\in r_{\balpha}\mathcal V_{\balpha} & \forall \balpha\in\mathcal A,\\
Q\in\displaystyle\bigcap_{k=1}^K \mathcal B\left(\sum_{j=1}^{N_k}a_{k,j}\delta_{\widehat\bxi_{k,j}},\;m\epsilon_k-\sum_{\balpha\in\mathcal A}s_{\balpha,k}\right)
\end{array}
\right\}.
\end{align}
Here, for each $\balpha\in\mathcal A$,
\begin{align}
\mathcal V_{\balpha}
=
\operatorname{conv}\left(
\left\{
\left(
c(\bxi,\widehat\bxi_{1,\alpha_1}),\ldots,c(\bxi,\widehat\bxi_{K,\alpha_K})
\right):\bxi\in\mathcal N^c
\right\}
\right).
\label{eq-outside-cost-vector}
\end{align}
We use the convention $0\mathcal V_{\balpha}:=\{0\}$.
\end{theorem}

The first implication of Theorem~\ref{th-analysisB1} is conceptual: the induced conditional set is a union of ordinary intersection balls. To make this interpretation explicit, let $\mathfrak F^{\rm I}$ be the set of all $(m,\ba,\br,\bs)$ satisfying
\begin{align*}
\mathfrak F^{\rm I}=
\left\{
\begin{array}{lll}
(m,\ba,\br,\bs):&
m\in[1/b,1/a],\quad a_{k,j}\ge0,\quad r_{\balpha}\ge0,\\
&\sum_{j=1}^{N_k}a_{k,j}=1
&\quad \forall k\in[K],\\
&a_{k,j}+\sum_{\balpha\in\mathcal A:\alpha_k=j}r_{\balpha}=m w_{k,j}
&\quad \forall j\in[N_k],~k\in[K],\\
&\bs_{\balpha}\in r_{\balpha}\mathcal V_{\balpha}
&\quad \forall \balpha\in\mathcal A.
\end{array}
\right\}.
%\label{eq-union-index-I}
\end{align*}
Then \eqref{eq-mainthB1-exact} is equivalently
\begin{align}
\mathcal B_{\mathcal N}^{\rm I}\left(\{\widehat{\p}_k\}_{k\in[K]},\bm\epsilon\right)
=
\bigcup_{(m,\ba,\br,\bs)\in\mathfrak F^{\rm I}}
\left[
\mathcal B^{\rm I}
\left(
\left\{\sum_{j=1}^{N_k}a_{k,j}\delta_{\widehat\bxi_{k,j}}\right\}_{k\in[K]},
m\bepsilon-\sum_{\balpha\in\mathcal A}\bs_{\balpha}
\right)
\cap
\{Q\in\mathcal P(\Xi):Q(\mathcal N)=1\}
\right].
\label{eq-unionI-exact}
\end{align}
For each feasible $(m,\ba,\br,\bs)$, the expression inside the union is an intersection of $K$ OT balls restricted to probability laws supported on $\mathcal N$. The scalar $m$ is the inverse of the probability assigned to $\mathcal N$ before conditioning. For each source $k$, the weights $\{a_{k,j}\}_{j\in[N_k]}$ define the source-$k$ reference distribution after conditioning, $\sum_{j=1}^{N_k} a_{k,j}\delta_{\widehat\bxi_{k,j}}$.

The remaining variables describe how the source mass not used on $\mathcal N$ consumes transport budget. The balance equation splits the scaled empirical mass $mw_{k,j}$ into the part $a_{k,j}$ used on $\mathcal N$ and the parts $r_{\balpha}$ assigned to $\mathcal N^c$ through tuples with $\alpha_k=j$. The vector $\bs_{\balpha}$ records the source-wise transport budget consumed by the outside-$\mathcal N$ mass assigned to tuple $\balpha$, so the radii available on $\mathcal N$ are $m\bepsilon-\sum_{\balpha\in\mathcal A}\bs_{\balpha}$.

The key constraint is $\bs_{\balpha}\in r_{\balpha}\mathcal V_{\balpha}$: it links the amount of mass assigned to $\mathcal N^c$ with the transport budget that this mass can consume. The set $\mathcal V_{\balpha}$ gives the possible source-wise transport cost vectors for tuple $\balpha$. Indeed,
\begin{align*}
\mathcal V_{\balpha}
&=
\left\{
\left(
W(P,\delta_{\widehat\bxi_{1,\alpha_1}}),\ldots,
W(P,\delta_{\widehat\bxi_{K,\alpha_K}})
\right):P\in\mathcal P(\mathcal N^c)
\right\}\\
&=
\left\{
\left(
\E_P[c(\bxi,\widehat\bxi_{1,\alpha_1})],\ldots,
\E_P[c(\bxi,\widehat\bxi_{K,\alpha_K})]
\right):P\in\mathcal P(\mathcal N^c)
\right\}.
\end{align*}
The equality with expectations follows because transporting a distribution to a Dirac measure has a unique coupling; the convex hull in \eqref{eq-outside-cost-vector} arises because the outside-$\mathcal N$ mass may be randomized over points in $\mathcal N^c$. Since $\mathcal V_{\balpha}$ is convex by construction, the constraint $\bs_{\balpha}\in r_{\balpha}\mathcal V_{\balpha}$ is a perspective constraint and is convex in $(r_{\balpha},\bs_{\balpha})$.

The second implication is computational. With the characterization in Theorem~\ref{th-analysisB1}, fixing $(m,\ba,\br,\bs)$ leaves an intersection DRO problem with fixed centers and radii, supported on $\mathcal N$. 
A tractable formulation can be obtained by 
applying fixed-center intersection duality to this inner problem and subsequently dualizing the finite-dimensional allocation problem.
To establish this tractable reformulation, we impose the following three assumptions.
{\color{black}
The first two assumptions ensure strong duality for the inner fixed-center intersection DRO problems for all admissible radii, including those on the boundary of their feasible regions.} The second and third assumptions allow the outer allocation set $\mathcal V_{\balpha}$ to be equivalently replaced by a compact set, thereby enabling the application of the minimax theorem.

{
\color{black}

\begin{assumption}[Cost-Lipschitz continuity]\label{assump:growthI}
% For every $\bbeta\in\mathcal D$, the function
% $\bxi\mapsto \ell(\bbeta,\bxi)$ is upper semicontinuous on $\Xi$.
% Moreover, 
For every $\bbeta\in\mathcal D$, there exists $L_{\bbeta}>0$ such that for any $\bxi\in\mathcal N$ and any source observation $\widehat{\bxi}$,
$$
\ell(\bbeta,\bxi)-\ell(\bbeta,\widehat{\bxi})
\le
L_{\bbeta}c(\bxi,\widehat{\bxi}).
$$
\end{assumption}
}

\begin{assumption}[Coercivity]\label{assump:costI}
For every source observation $\widehat\bxi$, the function $\bxi\mapsto c(\bxi,\widehat\bxi)$ is finite and coercive; that is, $c(\bxi,\widehat\bxi)\to \infty$ whenever $\|\bxi\|\to \infty$.
\end{assumption}

\begin{assumption}[Slater condition for intersection model]\label{assump:SlaterI}
There exists $P^{\circ}\in\mathcal P(\Xi)$
such that
\begin{align*}
P^{\circ}(\mathcal{N}) \in[a,b];\quad W\left(P^{\circ}, \widehat{\P}_k\right) < \varepsilon_k,\quad  k \in[K];\quad
\E^{P^{\circ}}[\ell(\bbeta,\bxi)\mid\bxi\in\mathcal N]<\infty,\quad \bbeta\in\mathcal D.
\end{align*}
\end{assumption}

{\color{black}
Assumption~\ref{assump:growthI} imposes a one-sided Lipschitz continuity condition on the loss function with respect to the transportation cost. Together with Assumption~\ref{assump:costI}, it provides sufficient conditions for the strong duality of unconditional DRO problems over intersection-based OT balls in \cite{REMK24} (see also \cite{WCW24}) to admissible radii that may lie on the boundary. This extension is formalized in Theorem~\ref{lm:boundaryI} of Section~\ref{sec:proofTHintersectionDual}. It is essential for our study because the inner intersection-based OT balls appearing in the union representation \eqref{eq-unionI-exact} may have boundary radii.

Assumption~\ref{assump:costI} is a mild coercivity condition on the transportation cost. It is satisfied, for example, by norm-type costs on $\mathbb R^d$.
The role of this assumption is twofold. First, when combined with Assumption~\ref{assump:growthI}, it guarantees tightness of sequences of feasible distributions whose OT radii are uniformly bounded. Hence Prokhorov's theorem can be applied to extract weakly convergent subsequences. This is the key compactness ingredient in the proof of the boundary-radii strong duality argument. Second, when combined with Assumption~\ref{assump:SlaterI}, it allows the generally unbounded sets $\mathcal V_{\balpha}$ to be replaced, without changing the robust value, by compact substitutes. This compactification is needed for the minimax argument in the derivation of the final reformulation. 

Assumption~\ref{assump:SlaterI} is a Slater-type condition for the intersection-based set. It ensures the existence of a joint law $P^\circ$ satisfying the mass constraint $P^\circ(\mathcal N)\in[a,b]$ and lying strictly inside all source-wise Wasserstein balls. This strict feasibility provides the slack needed to approximate closure-induced allocations by feasible joint laws. The integrability condition
$
\E^{P^\circ}[\ell(\bbeta,\bxi)\mid\bxi\in\mathcal N]<\infty
$
for all $\bbeta\in\mathcal D$ ensures that this approximation does not introduce an infinite conditional loss.
}

\begin{theorem}[General intersection robust counterpart]\label{th:intersectionDual}
{Suppose that Assumptions \ref{assump:growthI}, \ref{assump:costI} and \ref{assump:SlaterI} hold.}
Then, the intersection-based conditional DRO problem
\begin{align*}
\min_{\bbeta\in\mathcal D}
\sup_{\substack{Q\in\mathcal B^{\rm I}(\{\widehat{\p}_k\}_{k\in[K]},\bepsilon)\\ Q(\mathcal N)\in[a,b]}}
\E_Q[\ell(\bbeta,\bxi)\mid \bxi\in\mathcal N]
\end{align*}
is equivalent to the following finite-dimensional robust counterpart:
\begin{align}\label{prob-intersection-general}
\tag{$\mathrm P^{\rm I}_K$}
\begin{array}{lll}
\inf
&\quad
\kappa+\eta^+/a-\eta^-/b \nonumber\\
{\rm s.t.}
&\quad 
\bbeta\in\mathcal D, \kappa\in\R,\eta^+\ge0,\eta^-\ge0 \\
&\quad 
\lambda_k\ge0, \bm\varphi_k\in \R^{N_k}  &\forall k\in[K]\\
&\quad
 \sum_{k=1}^K\lambda_k\epsilon_k-\sum_{k=1}^K\sum_{j=1}^{N_k}\varphi_{k,j}w_{k,j}=\eta^+-\eta^-\nonumber\\
&\quad \sup_{\bxi\in\mathcal N}\left\{\ell(\bbeta,\bxi)-\sum_{k=1}^K\lambda_k c(\bxi,\widehat\bxi_{k,\alpha_k})\right\}
 +\sum_{k=1}^K\varphi_{k,\alpha_k}\le \kappa~~~~~~ &\forall \balpha\in\mathcal A\nonumber\\
&\quad \sum_{k=1}^K\varphi_{k,\alpha_k}\le d_{\balpha}^{\blambda}\quad &\forall \balpha\in\mathcal A,
\end{array}
\end{align}
where
\begin{align}
d_{\balpha}^{\blambda}:=
\inf_{\bgamma\in\mathcal V_{\balpha}}\blambda^\top\bgamma
=\inf_{\bxi\in\mathcal N^c}\sum_{k=1}^K\lambda_k c(\bxi,\widehat\bxi_{k,\alpha_k}),
\quad \balpha\in\mathcal A,\ \blambda\in\mathbb R_+^K.
\label{eq-dB1cap}
\end{align}
When $\mathcal D$ is convex and the displayed supremum constraints are convex in $\bbeta$, this counterpart is a convex program.
\end{theorem}

The counterpart has two recognizable parts. The supremum constraints over $\mathcal N$ are the usual Wasserstein-DRO worst-case loss constraints and can often be reduced further when the loss and cost have tractable forms. The genuinely new part is the scalarization of the outside-$\mathcal N$ transport costs. For fixed nonnegative weights $\blambda$, $d_{\balpha}^{\blambda}$ is the smallest weighted sum of transport costs from a point in $\mathcal N^c$ to the observations in tuple $\balpha$. Thus, the remaining source of complexity is the family of constraints
\begin{align}
\sum_{k=1}^K\varphi_{k,\alpha_k}\le d_{\balpha}^{\blambda},
\quad \balpha\in\mathcal A.
\label{eq-intersection-outside-cut}
\end{align}
These constraints are convex in $(\bm\varphi,\blambda)$ because $d_{\balpha}^{\blambda}$ is concave in $\blambda$. The computational question is how to evaluate or represent $d_{\balpha}^{\blambda}$ as $\blambda$ varies. When $\operatorname{cl}(\mathcal N^c)=\Xi$ and $\Xi$ is convex, as in singleton conditioning on a continuous support, this scalarization can be written directly through finite convex constraints.

\begin{proposition}[Convex representation of the scalarization constraints]\label{prop:scalarized-outside}
% { Suppose that Assumptions \ref{assump:growthI}, \ref{assump:costI} and \ref{assump:SlaterI} hold.}

Suppose that $\Xi$ is nonempty, closed and convex, and 
${\rm cl}(\mathcal N^c)=\Xi$. 
Then, for every $\balpha\in\mathcal A$ and $\blambda\in\mathbb R_+^K$,
\begin{align}
d_{\balpha}^{\blambda}
=\inf_{\bxi\in\Xi}\sum_{k=1}^K\lambda_k c(\bxi,\widehat\bxi_{k,\alpha_k}).
\label{eq-convex-dalpha}
\end{align}
Moreover, the constraint \eqref{eq-intersection-outside-cut} is equivalently represented by the existence of $z_{\balpha,k}$ and $v_{\balpha}$ satisfying
\begin{align}
&\sum_{k=1}^K\varphi_{k,\alpha_k}
+\sum_{k=1}^K\left(\lambda_k c(\cdot,\widehat\bxi_{k,\alpha_k})\right)^*(z_{\balpha,k})
+\sigma_{\Xi}(v_{\balpha})\le 0,\label{eq-fenchel-outside}\\
&\sum_{k=1}^K z_{\balpha,k}+v_{\balpha}=0,\label{eq-fenchel-balance}
\end{align}
where $f^*$ denotes the convex conjugate and $\sigma_{\Xi}$ is the support function of $\Xi$. For norm costs $c(\bxi,\widehat\bxi)=\|\bxi-\widehat\bxi\|$, these constraints reduce to
\begin{align}
&\sum_{k=1}^K\varphi_{k,\alpha_k}
+\sum_{k=1}^K z_{\balpha,k}^\top\widehat\bxi_{k,\alpha_k}
+\sigma_{\Xi}(v_{\balpha})\le 0,\label{eq-norm-outside}\\
&\sum_{k=1}^K z_{\balpha,k}+v_{\balpha}=0,~
\|z_{\balpha,k}\|_*\le \lambda_k\quad \forall k\in[K].\label{eq-norm-outside-balance}
\end{align}
\end{proposition}

The preceding proposition applies to any number of sources. We close this section with the two-source norm case, not as a limitation of the general theory, but because it gives a transparent benchmark used later in the paper. In this case, the scalarization becomes explicit: the relevant outside-cost geometry reduces to the line segment joining the cost vectors $(d_{\balpha},0)$ and $(0,d_{\balpha})$, and the final robust counterpart becomes a compact finite formulation.

\begin{corollary}[Two-source norm case]\label{cor:two-source-intersection}\label{th:SDK=2}
{Suppose that Assumptions \ref{assump:growthI} and \ref{assump:SlaterI} hold,} and in addition that $\Xi$ is convex, 
%all source observations belong to $\Xi$, 
$\operatorname{cl}(\mathcal N^c)=\Xi$, $K=2$, and $c(\bxi,\widehat\bxi)=\|\bxi-\widehat\bxi\|$ for some norm on the ambient space. For $\balpha=(\alpha_1,\alpha_2)$, define
\begin{align}
d_{\balpha}:=\left\|\widehat\bxi_{1,\alpha_1}-\widehat\bxi_{2,\alpha_2}\right\|.
\label{eq-dthmainB1K=2}
\end{align}
The quantity required in Theorem~\ref{th:intersectionDual} has the closed form
\begin{align}
d_{\balpha}^{\blambda}=\min\{\lambda_1,\lambda_2\}d_{\balpha},
\quad \blambda\in\mathbb R_+^2.
\label{eq-two-source-dlambda}
\end{align}
Consequently, 
%for any loss $\ell:\Xi\to\R$, 
the intersection-based conditional DRO problem
\begin{align}\label{prob:B1innerK=2}
{ \min_{\bbeta\in\mathcal D}}\sup_{Q\in\mathcal B^{\rm I}\left(\{\widehat{\p}_1,\widehat{\p}_2\},\bm\epsilon\right),Q(\mathcal N)\in[a,b]} \E_Q[\ell(\bbeta,\bxi)\mid\bxi\in \mathcal N]
\end{align}
is equivalent to
\begin{align}\tag{${\rm P}^{\rm I}_2$}\label{prob-capK=2}
\begin{array}{lll}\inf &
\varrho_1+\varrho_2-\frac{1}{b}\tau_1+\frac{1}{a}\tau_2\\
 {\rm s.t. } 
 &\bbeta\in\mathcal D, \tau_1,\tau_2\in\R_+\\
 & \varrho_k\in\R, \bm\varphi_{k}\in\R^{N_k}, \bm\psi_k\in\R_+^{N_k}, \zeta_k\in\R_+, \lambda_k \in \mathbb{R}_{+}, \bm\gamma_k \in \mathbb{R}^{N_k} ~~~~~~&\forall k=1,2\\
&\eta_{\balpha}\in\R_+ &\forall \balpha\in\mathcal A\\
& \sum_{k=1}^2 \lambda_k\epsilon_k-\sum_{k=1}^2\sum_{j=1}^{N_k}{\varphi_{k,j}}{w_{k,j}}+\tau_1-\tau_2=0\\
&-\lambda_1+\lambda_2+\zeta_1-\zeta_2=0\\
&\gamma_{k,j}-\varrho_k+\varphi_{k,j}+\psi_{k,j}=0 &\forall j\in[N_k],~k=1,2\\
&(-\lambda_1+\zeta_1)d_{\balpha}+\sum_{k=1}^2\varphi_{k,\alpha_k}+\eta_{\balpha}=0 &\forall \balpha\in\mathcal A\\
& \sup _{\bxi \in \mathcal N} \left\{\ell(\bbeta,\bxi)-\sum_{k=1}^2 \lambda_k \|\bxi-\widehat{\bxi}_{k, \alpha_k}\|\right\} \leq \sum_{k=1}^2 \gamma_{k, \alpha_k} & \forall \balpha \in \mathcal{A}.
\end{array}
\end{align}
\end{corollary}

\section{{Tractable Reformulation of the Weighted-Distance Model}}\label{sec:tractabilityB2}
We now turn to the weighted-distance formulation. Because it aggregates source discrepancies through a single transport budget, this model has a simpler computational structure than the intersection model, especially when many sources are present. As in Section \ref{sec:tractbilityB1_general}, we analyze the induced conditional set
\begin{align}\label{eq-BII-N}
\mathcal B^{\rm II}_{\mathcal N}\left(\{\widehat{\p}_k\}_{k\in[K]},\epsilon\right)
:=\left\{Q^{\mathcal N}\in\mathcal P(\Xi):~Q\in \mathcal B^{\rm II}\left(\{\widehat{\p}_k\}_{k\in[K]},\epsilon\right),~Q(\mathcal N)\in[a,b]\right\}.
\end{align}
The next theorem characterizes this induced set. 
%The only possible gap between the inner and outer descriptions comes from whether the outside-cost infima $d_{\balpha}^{\btheta}$ are attained.

\begin{theorem}[Exact characterization of weighted-distance conditional set]\label{th-analysisB2}
Let $\mathcal V_{\balpha}$ be defined in \eqref{eq-outside-cost-vector}.
The weighted-distance conditional set defined in \eqref{eq-BII-N} can be represented as follows:
\begin{align}\label{eq-mainthB2}
 &\mathcal B^{\rm{II}}_{\mathcal N}\left(\{\widehat{\p}_k\}_{k\in[K]},\epsilon\right)\notag\\
 &=
\left\{
\begin{array}{ll}
Q: Q\in\mathcal P(\Xi),~Q(\mathcal N)=1,~{\rm and~there~exist}~m,\\
\{a_{k,j}\}_{j\in[N_k],k\in[K]},~\{r_{\balpha}\}_{\balpha\in\mathcal A},~{\rm and}~\{\bs_{\balpha}\}_{\balpha\in\mathcal A}~{\rm such~that}\\
m\in[1/b,1/a],\\
a_{k,j}\ge 0 & \forall j\in[N_k],~k\in[K],\\
r_{\balpha}\ge 0 & \forall \balpha\in\mathcal A,\\
\sum_{j=1}^{N_k}a_{k,j}=1 & \forall k\in[K],\\
a_{k,j}+\sum_{\balpha\in\mathcal A:\alpha_k=j}r_{\balpha}=m w_{k,j} & \forall j\in[N_k],~k\in[K],\\
\bs_{\balpha}:=(s_{\balpha,1},\ldots,s_{\balpha,K})\in r_{\balpha}\mathcal V_{\balpha} & \forall \balpha\in\mathcal A,\\
\sum_{k=1}^K \theta_k W\left(Q,\sum_{j=1}^{N_k}a_{k,j}\delta_{\widehat{\bxi}_{k,j}}\right)\le m\epsilon -\sum_{\balpha\in\mathcal A}\btheta^{\top} \bs_{\balpha}
\end{array}
\right\}.
\end{align}
\end{theorem}

The characterization of the weighted-distance conditional set has an implication analogous to that in Section~\ref{sec:tractbilityB1}. It shows that the induced conditional set can be viewed as a union of ordinary weighted-distance balls. More precisely, it has the same union structure as in \eqref{eq-unionI-exact}, except that the intersection balls are replaced by weighted-distance balls centered at
$
\left\{\sum_{j=1}^{N_k}a_{k,j}\delta_{\widehat{\bxi}_{k,j}}\right\}_{k\in[K]}
$
with radius
$
m\epsilon-\sum_{\balpha\in\mathcal A}\btheta^\top\bs_{\balpha}.
$

For fixed $(m,\ba,\br)$, 
the only remaining role of $\bs_{\balpha}$ is through the term $\btheta^\top\bs_{\balpha}$ that  records the transport budget consumed by the outside-$\mathcal N$ mass assigned to tuple $\balpha$.
%which reduces the available radius.
Hence the largest radius is obtained by minimizing this term over $r_{\balpha}\mathcal V_{\balpha}$:
\begin{align}\label{eq-BII-infradius}
\inf_{\bs_{\balpha}\in r_{\balpha}\mathcal V_{\balpha}}
\btheta^\top\bs_{\balpha}
=
r_{\balpha}
\inf_{\bgamma\in\mathcal V_{\balpha}}
\btheta^\top\bgamma
=
r_{\balpha}d_{\balpha}^{\btheta},
\end{align}
where $d_{\balpha}^{\btheta}$ is defined in \eqref{eq-dB1cap}. Therefore, if the infimum in \eqref{eq-BII-infradius} is attained for every $\balpha\in\mathcal A$, the union over $\{\bs_{\balpha}\}_{\balpha\in\mathcal A}$ can be written equivalently by replacing the radius with
$
m\epsilon-\sum_{\balpha\in\mathcal A}r_{\balpha}d_{\balpha}^{\btheta}.
$
This removes the variables $\{\bs_{\balpha}\}_{\balpha\in\mathcal A}$ from the representation and yields the following simplified representation:
\begin{align}
 &\overline{\mathcal B}^{\rm II}_{\mathcal N}
 \left(\{\widehat{\p}_k\}_{k\in[K]},\epsilon\right)\notag\\
 &:=
\left\{
\begin{array}{ll}
Q: Q\in\mathcal P(\Xi),~Q(\mathcal N)=1,~\text{and there exist }m,\\
\{a_{k,j}\}_{j\in[N_k],k\in[K]}
\text{ and }
\{r_{\balpha}\}_{\balpha\in\mathcal A}
\text{ such that}\\
m\in[1/b,1/a],\\
a_{k,j}\ge0 & \forall j\in[N_k],~k\in[K],\\
r_{\balpha}\ge0 & \forall \balpha\in\mathcal A,\\
\sum_{j=1}^{N_k}a_{k,j}=1 & \forall k\in[K],\\
a_{k,j}+\sum_{\balpha\in\mathcal A:\alpha_k=j}r_{\balpha}=m w_{k,j}
& \forall j\in[N_k],~k\in[K],\\
\sum_{k=1}^K\theta_k
W\left(Q,\sum_{j=1}^{N_k}a_{k,j}\delta_{\widehat{\bxi}_{k,j}}\right)
\le
m\epsilon-\sum_{\balpha\in\mathcal A}r_{\balpha}d_{\balpha}^{\btheta}
\end{array}
\right\}\label{eq2-B2-reduction}\\
&=\bigcup_{(m,\ba,\br)\in\mathfrak F^{\rm II}}
\left[
\mathcal B^{\rm II}
\left(
\left\{\sum_{j=1}^{N_k}a_{k,j}\delta_{\widehat\bxi_{k,j}}\right\}_{k\in[K]},
m\epsilon-\sum_{\balpha\in\mathcal A}r_{\balpha}d_{\balpha}^{\btheta}
\right)
\cap
\{Q\in\mathcal P(\Xi):Q(\mathcal N)=1\}
\right]\notag%\label{eq2-B2-reduction-union},
\end{align}
where
\begin{align*}
\mathfrak F^{\rm II}=
\left\{
\begin{array}{lll}
(m,\ba,\br):&
m\in[1/b,1/a],\quad a_{k,j}\ge0,\quad r_{\balpha}\ge0,\\
&\sum_{j=1}^{N_k}a_{k,j}=1
&\quad \forall k\in[K],\\
&a_{k,j}+\sum_{\balpha\in\mathcal A:\alpha_k=j}r_{\balpha}=m w_{k,j}
&\quad \forall j\in[N_k],~k\in[K]
\end{array}
\right\}.
%\label{eq-union-index-II}
\end{align*}

If the infimum in \eqref{eq-BII-infradius} is not attained, the boundary radius
$
m\epsilon-\sum_{\balpha\in\mathcal A}r_{\balpha}d_{\balpha}^{\btheta}
$
may not be attainable by any admissible choice of $\{\bs_{\balpha}\}_{\balpha\in\mathcal A}$. In that case, the exact union corresponds to the associated open weighted-distance balls, whereas \eqref{eq2-B2-reduction} gives their closed-radius enlargement. We therefore introduce a Slater-type condition for the original weighted-distance ambiguity set to ensure that passing from the open-radius representation to the closed-radius representation does not change the value of the inner robust optimization for any $\bbeta\in\mathcal D$.

\begin{assumption}[Slater condition for weighted-distance model]\label{assump:SlaterII}
There exists $P^{\circ}\in\mathcal P(\Xi)$
such that
\begin{align*}
P^{\circ}(\mathcal{N}) \in[a,b];\quad \sum_{k=1}^K\theta_kW\left(P^{\circ}, \widehat{\P}_k\right) < \varepsilon;\quad
\E^{P^{\circ}}[\ell(\bbeta,\bxi)\mid \bxi\in\mathcal N]<\infty,\quad  \bbeta\in\mathcal D.
\end{align*}
\end{assumption}

\begin{lemma}\label{lm-BII-eqclosure}
Let $\overline{\mathcal B}^{\rm II}_{\mathcal N}
 \left(\{\widehat{\p}_k\}_{k\in[K]},\epsilon\right)$ be defined in \eqref{eq2-B2-reduction}, and suppose that Assumption \ref{assump:SlaterII} holds.
 Then,
 \begin{align*}
\sup_{Q\in \mathcal B^{\rm II}_{\mathcal N}
 \left(\{\widehat{\p}_k\}_{k\in[K]},\epsilon\right)}\E_{Q}[\ell(\bbeta,\bxi)]=\sup_{Q\in \overline{\mathcal B}^{\rm II}_{\mathcal N}
 \left(\{\widehat{\p}_k\}_{k\in[K]},\epsilon\right)}\E_{Q}[\ell(\bbeta,\bxi)]\quad \forall \bbeta\in\mathcal D.
 \end{align*}
\end{lemma}

{\color{black}
We are now ready to derive a tractable reformulation of the weighted-distance robust problem under Assumption~\ref{assump:SlaterII}. Lemma~\ref{lm-BII-eqclosure} allows us to work with the DRO problem induced by the conditional set defined in \eqref{eq2-B2-reduction}.
The argument follows the same sequence as in Section~\ref{sec:tractbilityB1}: after fixing the outer variables in \eqref{eq2-B2-reduction}, the inner maximization becomes an unconditional weighted-distance DRO problem supported on $\mathcal N$, with fixed reference distributions
$
\sum_j a_{k,j}\delta_{\widehat\bxi_{k,j}}
$
and radius
$
m\epsilon-\sum_{\balpha}r_{\balpha}d_{\balpha}^{\btheta}.
$
We first dualize this fixed inner problem, then interchange the outer maximization with the dual minimization, and finally dualize the remaining finite-dimensional linear program.

A technical point is that, to the best of our knowledge, a strong duality result for the unconditional weighted-distance DRO problem needed in the first step is not directly available in the literature. We therefore establish this result in Theorem~\ref{th-unconB2} of Section~\ref{app:pfTHsum}. Similar to the argument in Section \ref{sec:tractbilityB1}, Assumptions~\ref{assump:growthI} and \ref{assump:costI} are used to ensure that strong duality remains valid when the radius lies on the boundary of its feasible region.
}

\begin{theorem}[Weighted-distance robust counterpart]\label{th:sumB2Duality}
Suppose that {\color{black}Assumptions \ref{assump:growthI}, \ref{assump:costI} and \ref{assump:SlaterII} hold}, and
let $d_{\balpha}^{\btheta}$ be defined in \eqref{eq-dB1cap}. 
%Suppose that the outer description in \eqref{eq2-mainthB2} is value-exact for the loss $\ell(\bbeta,\cdot)$ and that the fixed-center weighted-distance duality used in the proof is exact for every outer tuple generated by \eqref{eq2-mainthB2}. Then, 
The weighted-distance conditional DRO problem
\begin{align*}%\label{prob:B2inner}
\min_{\bbeta\in\mathcal D}
\sup_{Q\in\mathcal B^{\rm II}\left(\{\widehat{\p}_k\}_{k\in[K]},\epsilon\right),Q(\mathcal N)\in[a,b]} \E_Q[\ell(\bbeta,\bxi)\mid\bxi\in \mathcal N]
\end{align*} 
is equivalent to
\begin{align}\tag{${\rm P}_K^{\rm II}$}\label{prob-sum}
\begin{array}{lll}\inf & 
\sum_{k=1}^K\varrho_k-\frac{1}{b}\tau_1+\frac{1}{a}\tau_2\\
\text{\rm s.t.} & \bbeta\in\mathcal D,\ \lambda, \tau_1, \tau_2\in\R_+\\
&\varrho_k\in\R,~\bm\varphi_{k}\in\R^{N_k},~\bm\psi_k\in\R_+^{N_k},~\bm\gamma_k\in\R^{N_k} &\forall k\in[K]\\
&\eta_{\balpha}\in\R_+ &\forall \balpha\in\mathcal A\\ 
& \lambda \epsilon-\sum_{k=1}^K\sum_{j=1}^{N_k}{\varphi_{k,j}}{w_{k,j}}+\tau_1-\tau_2=0\\
&\gamma_{k,j}-\varrho_k+\varphi_{k,j}+\psi_{k,j}=0 &\forall j\in[N_k],~k\in[K]\\
&-\lambda d_{\balpha}^{\btheta}+\sum_{k=1}^K\varphi_{k,\alpha_k}+\eta_{\balpha}=0 &\forall \balpha\in\mathcal A\\
& \sup _{\bxi \in \mathcal N} \left\{\ell(\bbeta,\bxi)-\lambda\sum_{k=1}^K \theta_k c(\bxi,\widehat{\bxi}_{k, \alpha_k})\right\} \leq \sum_{k=1}^K \gamma_{k, \alpha_k} ~~~~~~& \forall \balpha \in \mathcal{A}.
\end{array}
\end{align}
\end{theorem}

Formulation \eqref{prob-sum} has the same structure as the intersection counterparts in \eqref{prob-intersection-general} and \eqref{prob-capK=2}. The finite-dimensional variables encode how empirical mass is retained inside $\mathcal N$ and how the transport budget is priced, while the remaining analytical work is confined to the supremum over $\mathcal N$. Appendix \ref{app:morerefined} gives more explicit reformulations of this supremum under additional structure on the loss function and conditioning event. %These reformulations allow for direct implementation of the numerical experiments discussed in Section~\ref{sec:num}.

\section{Tractable Reformulation of the Barycenter-Based Model}\label{sec:tractabilityB3}
In this section,
we consider the barycenter-based formulation. Once the Wasserstein barycenter has been computed, the multi-source ambiguity set reduces to a single-source OT ball centered at the barycenter. The tractability question therefore becomes a single-source conditional DRO problem, with the barycenter serving as the reference empirical measure.

%A direct result from Theorem \ref{th-analysisB1} is to consider the situation that there is a single source data, which is denoted by $\{\widehat{\bxi}_{j}\}_{j\in[N]}$. 

Throughout this section, suppose that the barycenter has the form 
\begin{align}\label{eq-barycenter}
\widehat{\p}_{\rm bar}=\sum_{j=1}^N w_j\delta_{\widehat{\bxi}_{j}},
\end{align}
where $\{\widehat{\bxi}_j\}_{j\in[N]}\subseteq \Xi$, $\bw\ge 0$ and $\sum_{j=1}^N w_j=1$. 
In this case,
$$
\mathcal B^{\rm III}\left(\{\widehat{\p}_k\}_{k\in[K]},\epsilon\right)=\mathcal B(\widehat{\p}_{\rm bar},\epsilon).
%=\mathcal B^{\rm I}(\widehat{\p}_{\rm bar},\epsilon)
$$
As in Section \ref{sec:tractbilityB1_general}, we focus on the induced conditional set
\begin{align*}%\label{eq-singleuncertaintyN}
\mathcal B^{\rm{III}}_{\mathcal N}\left(\{\widehat{\p}_k\}_{k\in[K]},\epsilon\right)
%=\mathcal B_{\mathcal N}^{\rm I}(\widehat{\p}_{\rm bar},\epsilon)
:=\left\{Q^{\mathcal N}\in\mathcal P(\Xi):~Q\in \mathcal B(\widehat{\p}_{\rm bar},\epsilon),~Q(\mathcal N)\in[a,b]\right\}.
\end{align*}
In the single-source case, where the source distribution is the barycenter $\widehat{\p}_{\rm bar}$, the three ambiguity-set constructions coincide. Consequently, $\mathcal B_{\mathcal N}^{\rm III}$ can be characterized either from Theorem~\ref{th-analysisB1} or from Theorem~\ref{th-analysisB2}.
The following corollary summarizes the resulting characterization. 
Its proof is omitted.

% Appendix~\ref{app:proofBarycenter} gives two derivations, one from each of Theorems~\ref{th-analysisB1} and \ref{th-analysisB2}, showing that the two multi-source characterizations agree in this single-source specialization.

\begin{corollary}[Exact characterization of single-source conditional set]\label{co-singledata}
Suppose that the barycenter $\widehat{\p}_{\rm bar}$ is defined in \eqref{eq-barycenter}. The barycenter-based conditional set can be represented as follows:
\begin{align*}%\label{eq-cosingleIII}
\mathcal B^{\rm{III}}_{\mathcal N}\left(\{\widehat{\p}_k\}_{k\in[K]},\epsilon\right)
=\left\{\begin{array}{ll}
Q\in\mathcal P(\Xi): Q(\mathcal N)=1,~{\rm and~
there~exist}\\
m,
\{a_j\}_{j\in[N]}~and~\{v_j\}_{j\in[N]}~{\rm such~that}\\
m\in\left[{1}/{b},{1}/{a}\right],~\sum_{j=1}^N a_j=1,\\
~0\le  a_j\le mw_j,~v_j\in\mathcal V_j~~&\forall j\in[N],\\
Q\in\mathcal B\left(\sum_{{j}=1}^Na_{j} \delta_{\widehat{\bxi}_{j}},m\epsilon -\sum_{j=1}^N (mw_j-a_j)v_j\right)
\end{array}\right\},
\end{align*}
where
$
\mathcal V_j={\rm conv}(\{c(\bxi,\widehat{\bxi}_{j}): \bxi\in\mathcal N^c\})
$ for $j\in[N]$.
\end{corollary}

Define
\begin{align}\label{eq-certaindiSD}
d_j=\inf_{\bxi\in\mathcal N^c} c(\bxi,\widehat{\bxi}_j),
\qquad j\in[N].
\end{align}
Observe that the term $mw_j-a_j$ is always nonnegative.
If the optimization problem in \eqref{eq-certaindiSD} admits a minimizer for every $j\in[N]$, then the single-source conditional set admits the following simplified representation
\begin{align}
\overline{\mathcal B}^{\rm{III}}_{\mathcal N}\left(\{\widehat{\p}_k\}_{k\in[K]},\epsilon\right)
&:=\left\{\begin{array}{ll}
Q\in\mathcal P(\Xi): Q(\mathcal N)=1,~{\rm and~
there~exist}~m~{\rm and}~
\{a_j\}_{j\in[N]}\\
{\rm such~that}~
m\in\left[{1}/{b},{1}/{a}\right],~\sum_{j=1}^N a_j=1,~0\le  a_j\le mw_j~~&\forall j\in[N],\\
Q\in\mathcal B\left(\sum_{{j}=1}^Na_{j} \delta_{\widehat{\bxi}_{j}},m\epsilon -\sum_{j=1}^N (mw_j-a_j)d_j\right)
\end{array}\right\}\label{eq-B3closure}\\
&=\bigcup_{(m,\ba)\in\mathfrak F^{\rm III}}
\left[
\mathcal B
\left(
\sum_{j=1}^{N}a_{j}\delta_{\widehat\bxi_{j}},
m\epsilon-\sum_{j=1}^N (mw_j-a_j)d_j
\right)
\cap
\{Q\in\mathcal P(\Xi):Q(\mathcal N)=1\}
\right],\notag
\end{align}
where 
\begin{align*}
\mathfrak F^{\rm III}=
\left\{
\begin{array}{l}
(m,\ba):
m\in[1/b,1/a],\quad\sum_{j=1}^N a_j=1,\quad 0\le a_{j}\le mw_j\quad\forall j\in[N]
\end{array}
\right\}.
\end{align*}

\begin{remark}[Trimming method]
\cite{E22} consider conditional DRO problems of the form
$\sup_{Q\in\mathcal B}\E_Q[\ell(\bm\beta,\bxi)\mid\bxi\in\mathcal N]$, where the ambiguity set $\mathcal B$ consists of all distributions $Q$ in $\mathcal B(\widehat{\p},\epsilon)$ satisfying $Q(\mathcal N)=\eta$, and $\widehat{\p}$ denotes the empirical distribution with $w_j=1/N$ for all $j\in[N]$.
%^with $\widehat{\p}=(1/N)\sum_{j=1}^N \delta_{\widehat{\bxi}_j}$. 
This conditional problem can be written equivalently as the following unconditional problem:
\begin{align*}
\sup_{Q\in\mathcal B_{\mathcal N}}\E_Q[\ell(\bm\beta,\bxi)],
\end{align*}
where $\mathcal B_{\mathcal N}:=\{Q^{\mathcal N}: Q\in\mathcal B\}$.
% , and $Q^\mathcal N$ denotes the conditional measure of $Q$ on $\mathcal N$ defined as $Q^{\mathcal N}(A)=Q(A\cap \mathcal N)/Q(\mathcal N)$ for all $A\subseteq \Xi$.
Their trimming approach can be viewed through the induced-set lens above. Instead of deriving an exact characterization of $\mathcal B_{\mathcal N}$,
% the set of all joint conditional distributions $Q^{\mathcal N}$ from the ambiguity set $\mathcal B$, that is equivalent to the set $\mathcal B_{\mathcal N}^{\rm III}(\{\widehat{\p}_k\}_{k\in[K]},\epsilon)$ defined in \eqref{eq-singleuncertaintyN} with $\widehat{\p}_{\rm bar}=\widehat{\p}$ and $a=b=\eta$, 
\cite{E22} propose a relaxation based on trimming:
\begin{align}\label{eq-deftrimmingset}
\left\{\begin{array}{l}
Q\in\mathcal P(\Xi):
Q(\mathcal N)=1,~and~there~exist~\mathbf a=\{a_j\}_{j\in[N]}~such~that\\
Q\in\mathcal B\left(\sum_{j=1}^Na_{j} \delta_{\widehat{\bxi}_{j}},\frac{\epsilon}{\eta}\right),~
\sum_{j=1}^N a_j=1,~0\le \mathbf a\le \frac{1}{N\eta}
\end{array}\right\}. 
\end{align}
Specializing the representation \eqref{eq-B3closure} to $\widehat{\p}_{\rm bar}=\widehat{\p}$ and $a=b=\eta$ gives the corresponding closed description of the induced set $\mathcal B_{\mathcal N}$, which is exact when the infima defining $d_j$ are attained:
\begin{align*}
\left\{\begin{array}{l}
Q\in\mathcal P(\Xi): Q(\mathcal N)=1,~and~
there~exist~\mathbf a=\{a_j\}_{j\in[N]}~such~that\\
Q\in\mathcal B\left(\sum_{j=1}^Na_{j} \delta_{\widehat{\bxi}_{j}},\frac{\epsilon}{\eta}-\sum_{j=1}^N\left(\frac{1}{N\eta}-a_j\right)d_j \right),~\sum_{j=1}^N a_j=1,~0\le \mathbf a\le \frac{1}{N\eta}
\end{array}\right\}. 
\end{align*}
In contrast, the relaxed set in \cite{E22} defined by \eqref{eq-deftrimmingset} omits the additional cost term $\sum_{j=1}^N (1/(N\eta)- a_{j})d_{j}$ that appears in the radius of each OT ball. This term accounts for the transport budget assigned to the part of the joint law outside $\mathcal N$.
\end{remark}

{\color{black}
Finally, we obtain the tractable reformulation of the barycenter-based robust problem by specializing Theorem~\ref{th:sumB2Duality} to the single-source ambiguity set centered at $\widehat{\p}_{\rm bar}$. As in Sections~\ref{sec:tractbilityB1} and \ref{sec:tractabilityB2}, we impose Assumption~\ref{assump:growthI}, Assumption~\ref{assump:costI}, and the following Slater-type condition for the barycenter-based model.

\begin{assumption}[Slater condition for barycenter-based model]\label{assump:SlaterIII}
There exists $P^{\circ}\in\mathcal P(\Xi)$
such that
\begin{align*}
P^{\circ}(\mathcal{N}) \in[a,b];\quad W\left(P^{\circ}, \widehat{\P}_{\rm bar}\right) < \varepsilon;\quad
\E^{P^{\circ}}[\ell(\bbeta,\bxi)\mid\bxi\in\mathcal N]<\infty,\quad  \bbeta\in\mathcal D.
\end{align*}
\end{assumption}

}

\begin{corollary}[Barycenter robust counterpart]\label{co:barycenterB3Duality}
Let the barycenter $\widehat{\p}_{\rm bar}$, $d_j$ and $\overline{\mathcal B}_{\mathcal N}^{\rm III}$ be defined in \eqref{eq-barycenter}, \eqref{eq-certaindiSD} and \eqref{eq-B3closure}, respectively. Suppose that Assumption  \ref{assump:SlaterIII} holds.
Then
 \begin{align}\label{eq-B3eqclosure}
\sup_{Q\in \mathcal B^{\rm III}_{\mathcal N}
 \left(\{\widehat{\p}_k\}_{k\in[K]},\epsilon\right)}\E_{Q}[\ell(\bbeta,\bxi)]=\sup_{Q\in \overline{\mathcal B}^{\rm III}_{\mathcal N}
 \left(\{\widehat{\p}_k\}_{k\in[K]},\epsilon\right)}\E_{Q}[\ell(\bbeta,\bxi)]\quad \forall \bbeta\in\mathcal D.
 \end{align}
Suppose further that Assumptions \ref{assump:growthI} and \ref{assump:costI} hold. Then the barycenter-based conditional DRO problem
\begin{align}\label{prob:B3inner}
\min_{\bbeta\in\mathcal D}
\sup_{Q\in\mathcal B^{\rm III}\left(\{\widehat{\p}_k\}_{k\in[K]},\epsilon\right),Q(\mathcal N)\in[a,b]} \E_Q[\ell(\bbeta,\bxi)\mid\bxi\in \mathcal N]
\end{align} 
is equivalent to
\begin{align}\tag{${\rm P}^{\rm III}$}\label{prob-barycenterB3}
\begin{array}{lll}\inf & 
\varrho-\frac{1}{b}\tau_1+\frac{1}{a}\tau_2\\
\text{\rm s.t.} & \bbeta\in\mathcal D,\ \lambda,\tau_1,\tau_2\in\R_+,~\varrho\in\R\\
& \varphi_j,\psi_j\in\R_+,~\gamma_j\in\R &\forall j\in[N]\\
& \lambda \epsilon-\lambda\sum_{j=1}^Nw_jd_j+\tau_1-\tau_2+\sum_{j=1}^N w_j\varphi_j=0~~~~~~\\
&\gamma_{j}+\lambda d_j-\varrho-\varphi_{j}+\psi_{j}=0 &\forall j\in[N]\\
& \sup _{\bxi \in \mathcal N} \left\{\ell(\bbeta,\bxi)-\lambda  c(\bxi,\widehat{\bxi}_{j})\right\} \leq \gamma_j & \forall j \in [N].
\end{array}
\end{align}
\end{corollary}

{\color{black}
\begin{remark}
Note that, in Corollary~\ref{co:barycenterB3Duality}, Assumptions~\ref{assump:growthI} and~\ref{assump:costI} are used to ensure that strong duality for the associated unconditional DRO problems holds for all admissible radii. In the single-source case, the boundary issue reduces to the right-continuity of the corresponding worst-case value at zero radius. For this purpose, the two assumptions can be weakened. Specifically, \citet[Remark~3]{ZYG24} provide the following sufficient condition: there exists a continuous concave function $\varphi:[0,\infty)\to[0,\infty)$ with $\varphi(0)=0$ such that
$
\ell(\bxi)-\ell(\widehat{\bxi})
\le
\varphi(c(\bxi,\widehat{\bxi}))
$
for all $\bxi\in\Xi$ and all source observations $\widehat{\bxi}$.
This condition is more general than the combined requirements imposed in Assumptions~\ref{assump:growthI} and~\ref{assump:costI} in two respects. First, it does not require the cost function to be coercive. Second, it replaces the linear cost-Lipschitz bound in Assumption~\ref{assump:growthI} by a concave one. In particular, it allows  $\varphi$ that may have infinite derivative at zero, and hence covers loss functions that are not cost-Lipschitz in the linear sense.
\end{remark}

}

%%%%%%%%%%%%%%%%%%%%%%%%%%%%%%%%%%%%%%%%%%%%%%%%%%%%%%%%%%
\section{Feasible Regions of Radii}\label{sec:feasible}
%\subsection{Feasible radii for two source distributions}
The radii in the ambiguity sets must be large enough for the model to admit at least one joint law that assigns probability in $[a,b]$ to the conditioning event. This section characterizes the corresponding feasible regions and provides a basis for selecting radii in implementation. For the three models, define
\begin{align}\label{eq-radiiregion1}
\mathcal E^{\rm I}=\left\{\bepsilon\in\R_+^K: \exists Q\in\mathcal P(\Xi) \text{ s.t. }  Q\in\mathcal B^{\rm I}\left(\{\widehat{\p}_k\}_{k\in[K]},\bepsilon\right)\text{ and }Q(\mathcal N)\in[a,b]\right\},
\end{align}
\begin{align}\label{eq-radiiregion2}
\mathcal E^{\rm II}=\left\{\epsilon\in\R_+: \exists Q\in\mathcal P(\Xi) \text{ s.t. }  Q\in\mathcal B^{\rm II}\left(\{\widehat{\p}_k\}_{k\in[K]},\epsilon\right)\text{ and }Q(\mathcal N)\in[a,b]\right\}
\end{align}
and 
\begin{align}\label{eq-radiiregion3}
\mathcal E^{\rm III}=\left\{\epsilon\in\R_+: \exists Q\in\mathcal P(\Xi) \text{ s.t. }  Q\in\mathcal B^{\rm III}\left(\{\widehat{\p}_k\}_{k\in[K]},\epsilon\right)\text{ and }Q(\mathcal N)\in[a,b]\right\}.
\end{align}

% To characterize the regions defined above, we focus on deriving their Pareto frontiers. 
% Conceptually, obtaining these Pareto frontiers entails a level of complexity comparable to solving problem \eqref{prob-mainprob} with the corresponding ambiguity sets. 

%Characterizing these regions is itself an optimization problem of comparable difficulty to the original robust optimization problem \eqref{prob-mainprob}. 

For $\mathcal E^{\rm I}$, we focus on the two-source case ($K=2$), where Corollary~\ref{cor:two-source-intersection} provides an explicit line-segment structure. The following result characterizes the interior of $\mathcal E^{\rm I}$, denoted by $(\mathcal E^{\rm I})^{\circ}$.

% In this section, we aim to establish the region of radii $\bepsilon$ such that the feasible set is not empty for the case of $K=2$. 
% The following result characterizes the interior of $\mathcal E^{\rm I}$, which is denoted by $(\mathcal E^{\rm I})^{\circ}$.

\begin{proposition}\label{prop:radiiregion1}
For $K=2$, let $\mathcal E^{\rm I}$ be defined in \eqref{eq-radiiregion1}.
% {\color{red}Suppose that the Lagrangian duality step used in the proof is exact for every $\epsilon_1$ strictly above the left endpoint in \eqref{eq-regionradii1}.}
The radii $\bepsilon$ lie in  $(\mathcal E^{\rm I})^{\circ}$ if and only if
%there exists a permutation $\sigma$ of $[K]$ such that 
\begin{align}\label{eq-regionradii1}
\begin{array}{llll}
\epsilon_1
> &\inf &\sum_{j=1}^{N_{1}}p_j\inf_{\bxi\in\mathcal
N}c(\bxi,\widehat{\bxi}_{1,j})+\sum_{j=1}^{N_{1}}q_j\inf_{\bxi\in\mathcal N^c}c(\bxi,\widehat{\bxi}_{1,j})\\
&{\rm s.t.} &\sum_{j=1}^{N_{1}}p_j\in[a,b]\\
&&p_j+q_j=w_{1,j},~p_j\ge 0,~q_j\ge 0 &\forall j\in[N_{1}]
\end{array}
\end{align}
and
\begin{align}\label{eq-regionradii2}
\begin{array}{llll}
\epsilon_2
>\sup_{\lambda\ge 0}&\inf 
&\sum_{\balpha\in\mathcal A}p_{\balpha}\left(c(\bxi_{\balpha},\widehat{\bxi}_{2,\alpha_{2}})+\lambda c(\bxi_{\balpha},\widehat{\bxi}_{1,\alpha_{1}})\right)\\
&&+\sum_{\balpha\in\mathcal A}q_{\balpha}\left(c(\bxi_{\balpha}',\widehat{\bxi}_{2,\alpha_{2}})+\lambda c(\bxi_{\balpha}',\widehat{\bxi}_{1,\alpha_{1}})\right)-\lambda\epsilon_1\\
&{\rm s.t.} &\sum_{\balpha\in\mathcal A}p_{\balpha}\in[a,b]\\
&&p_{\balpha}\ge 0,~q_{\balpha}\ge 0,~\bxi_{\balpha}\in\mathcal N,~\bxi_{\balpha}'\in\mathcal N^c
%~\bxi_{\balpha}\in\mathcal N,~\bxi_{\balpha}'\in\mathcal N^c 
&\forall \balpha\in\mathcal A\\
&&\sum_{\balpha\in\mathcal A,\alpha_{k}=j} \{p_{\balpha}+q_{\balpha}\}=w_{k,j}&\forall j\in[N_{k}], k=1,2.
% && s_{\balpha}\ge c(\bxi_{\balpha},\widehat{\bxi}_{2,\alpha_{2}})+\lambda c(\bxi_{\balpha},\widehat{\bxi}_{1,\alpha_{1}}) &\forall \balpha\in\mathcal A\\
% && t_{\balpha}\ge c(\bxi_{\balpha}',\widehat{\bxi}_{2,\alpha_{2}})+\lambda c(\bxi_{\balpha}',\widehat{\bxi}_{1,\alpha_{1}}) &\forall \balpha\in\mathcal A.
\end{array}
\end{align}
%where $\mathcal A_k=\times_{i=1}^k [N_{i}]$ as the set of $k$-dimensional multi-indices.
\end{proposition}

The optimization problems in \eqref{eq-regionradii1} and \eqref{eq-regionradii2} are generally difficult for arbitrary cost functions $c$ and conditioning sets $\mathcal N$. The following corollary addresses a practically important two-source norm case in the spirit of Corollary~\ref{cor:two-source-intersection}. Under these conditions, both problems reduce to linear programs.

\begin{corollary}\label{prop:special}\label{co:special}
Suppose $\mathcal N = \{\bx_0\} \times \mathcal Y$ for some $\bx_0 \in \mathcal X$, $\operatorname{cl}(\mathcal N^c)=\Xi$, and the cost function is separable as $c(\bxi, \bxi') = \|\bx - \bx'\|_{\mathcal X} + \|\by - \by'\|_{\mathcal Y}$, where $\|\cdot\|_{\mathcal X}$ and $\|\cdot\|_{\mathcal Y}$ are norms on $\mathcal X$ and $\mathcal Y$, respectively. Then, both problems in \eqref{eq-regionradii1} and \eqref{eq-regionradii2} can be reformulated as linear programs. The problem in \eqref{eq-regionradii1} is equivalent to
\begin{align*}
\begin{array}{lll}
\inf & \sum_{j=1}^{N_1}p_j \|\widehat{\bx}_{1,j}-\bx_0\|_{\mathcal X}\\
{\rm s.t.} & \sum_{j=1}^{N_1}p_j\in[a,b]\\
&0\le p_j\le w_{1,j} &\forall j\in[N_1],
\end{array}
\end{align*}
and the problem in \eqref{eq-regionradii2} is equivalent to
\begin{align}\label{eq-feasibleradiiI}
\begin{array}{lll}
\inf & 
\sum_{\balpha\in\mathcal A} 
p_{\balpha}\|\widehat{\bx}_{2,\alpha_2}-\bx_0\|_{\mathcal X}+\tau_3\\
{\rm s.t.} &\bm\tau\in \R_+^3\\
&\sum_{\balpha\in\mathcal A}p_{\balpha}\in[a,b]\\
&p_{\balpha}\ge 0,~q_{\balpha}\ge 0
%~\bxi_{\balpha}\in\mathcal N,~\bxi_{\balpha}'\in\mathcal N^c 
&\forall \balpha\in\mathcal A\\
&\sum_{\balpha\in\mathcal A,\alpha_{k}=j} \{p_{\balpha}+q_{\balpha}\}=w_{k,j}&\forall j\in[N_{k}], k=1,2\\
& \sum_{\balpha\in\mathcal A}\{p_{\balpha}d_{\balpha}^{\mathcal Y}+q_{\balpha}d_{\balpha}\}-\tau_2-\tau_3=0\\
& \sum_{\balpha\in\mathcal A} p_{\balpha}\|\widehat{\bx}_{1,\balpha_1}-\bx_0\|_{\mathcal X}-\epsilon_1+\tau_1+\tau_2=0,\\
\end{array}
\end{align}
where $d_{\balpha}=c(\widehat{\bxi}_{1,\alpha_1},\widehat{\bxi}_{2,\alpha_2})$ that has been defined in \eqref{eq-dthmainB1K=2}, and $d_{\balpha}^{\mathcal Y}=\|\widehat{\by}_{1,\alpha_1}-\widehat{\by}_{2,\alpha_2}\|_{\mathcal Y}$ for $\balpha\in\mathcal A$.
\end{corollary}

%\subsection{Feasible region of the radius}
% \textbf{Discussion on the feasible region of the radius $\epsilon$.}
% To ensure that the problem~\eqref{prob:B2inner} is well-posed, the radius parameter $\epsilon$ must be chosen such that the corresponding ambiguity set is nonempty. Specifically, $\epsilon$ should lie in the following admissible set:
% \begin{align}\label{eq-radiiregion2}
% \mathcal{E}^{\rm II} := \left\{ \epsilon \in \mathbb{R}_+ : \exists Q \in \mathcal{P}(\Xi) \text{ s.t. } Q \in \mathcal{B}^{\rm II} \left( \{ \widehat{\p}_k \}_{k \in [K]}, \epsilon \right) \text{ and } Q(\mathcal{N}) \in [a, b] \right\}.
% \end{align}

We next turn to $\mathcal E^{\rm II}$. Since increasing the radius preserves feasibility, $\mathcal E^{\rm II}$ is an interval of all radii above a threshold. The central task is therefore to identify its left endpoint, which the following result expresses as a linear program.
% Define
% \begin{align}\label{eq-disN}
% h_{\balpha}^{\btheta}=\inf_{\bxi\in \mathcal N} \sum_{k=1}^K \theta_k c(\bxi,\widehat{\bxi}_{k,\alpha_k}).
% \end{align}

\begin{proposition}\label{prop:feasibleradiiII}
The left endpoint of $\mathcal E^{\rm II}$ defined by \eqref{eq-radiiregion2} is equal to the optimal value of the following linear program:
\begin{align*}%\label{eq-regionradiiII}
\begin{array}{llll}
&\inf
&\sum_{\balpha\in\mathcal A}p_{\balpha}h_{\balpha}^{\btheta}+\sum_{\balpha\in\mathcal A}q_{\balpha}d_{\balpha}^{\btheta}\\
&{\rm s.t.} &\sum_{\balpha\in\mathcal A}p_{\balpha}\in[a,b]\\
&&p_{\balpha}\ge 0,~q_{\balpha}\ge 0
%~\bxi_{\balpha}\in\mathcal N,~\bxi_{\balpha}'\in\mathcal N^c 
&\forall \balpha\in\mathcal A\\
&&\sum_{\balpha\in\mathcal A,\alpha_{k}=j} \{p_{\balpha}+q_{\balpha}\}=w_{k,j}~~~~~~&\forall j\in[N_{k}], k\in[K],
\end{array}
\end{align*}
where $h_{\balpha}^{\btheta}=\inf_{\bxi\in \mathcal N} \sum_{k=1}^K \theta_k c(\bxi,\widehat{\bxi}_{k,\alpha_k})$,
and
$d_{\balpha}^{\btheta}$ is defined in \eqref{eq-dB1cap}.
\end{proposition}

In the two-source separable-cost setting of Corollary~\ref{co:special}, Proposition~\ref{prop:feasibleradiiII} reduces to the following linear program:
\begin{align}\label{eq-regionradiiII}
\begin{array}{llll}
&\inf 
&\sum_{\balpha\in\mathcal A}p_{\balpha}\left(\min\{\theta_1,\theta_2\}\|\widehat{\by}_{1,\alpha_1}-\widehat{\by}_{2,\alpha_2}\|+\sum_{k=1}^2 \theta_k\|\bx_0-\widehat{\bx}_{k,\alpha_k}\|\right)\\
&&+\sum_{\balpha\in\mathcal A}q_{\balpha}\min\{\theta_1,\theta_2\}\|\widehat{\bxi}_{1,\alpha_1}-\widehat{\bxi}_{2,\alpha_2}\|\\
&{\rm s.t.} &\sum_{\balpha\in\mathcal A}p_{\balpha}\in[a,b]\\
&&p_{\balpha}\ge 0,~q_{\balpha}\ge0
&\forall \balpha\in\mathcal A\\
&&\sum_{\balpha\in\mathcal A,\alpha_{k}=j} (p_{\balpha}+q_{\balpha})=w_{k,j}&\forall j\in[N_{k}], k=1,2.
\end{array}
\end{align}

The set $\mathcal E^{\rm III}$ is the single-source counterpart of $\mathcal E^{\rm II}$ after replacing the empirical sources by the barycenter. When the barycenter has the form specified in \eqref{eq-barycenter}, the following corollary follows directly from Proposition \ref{prop:feasibleradiiII}.

\begin{corollary}\label{co-fsradii3}
The left endpoint of $\mathcal E^{\rm III}$ defined by \eqref{eq-radiiregion3} is equal to the optimal value of the following linear program:
\begin{align*}
\begin{array}{llll}
&\inf 
&\sum_{j=1}^N p_{j}h_{j}+\sum_{j=1}^Nq_{j}d_{j}\\
&{\rm s.t.} &\sum_{j=1}^Np_{j}\in[a,b]\\
&&p_{j}\ge 0,~q_{j}\ge 0,~p_{j}+q_{j}=w_{j}~~~~~~&\forall j\in[N],
\end{array}
\end{align*}
where $h_{j}=\inf_{\bxi\in\mathcal N} c(\bxi,\widehat{\bxi}_{j})$, and $d_j$ is defined in \eqref{eq-certaindiSD}.
\end{corollary}

% {\color{black} THIS PART MAY NOT BE NEEDED.

% Below, we provide a practically computable approximation of the smallest radius for which condition~\eqref{eq-radiiregion2} holds via a two-step procedure.

% \begin{enumerate}
% \item \textbf{Compute a lower bound.} 
%     This is the smallest $\epsilon$ such that the ambiguity set $\mathcal B^{\rm II}\left(\{\widehat{\p}_k\}_{k\in [K]},\epsilon\right)$ is nonempty. This value corresponds to the optimal value of the barycenter problem defined in~\eqref{def:barycenter}, and we denote it by $\epsilon_L$.
    
% \item \textbf{Search for the smallest radius in an interval.} 
% We set a large enough radius as an upper bound, denoted $\epsilon_U$, and construct a grid of candidate radii uniformly spaced over the interval $[\epsilon_L, \epsilon_U]$. Let $s_M = (\epsilon_U - \epsilon_L)/M$, and define the candidate radii as:
%     $$
%     \epsilon_m = \epsilon_L + m s_M, \quad m \in[M].
%     $$
%     For each candidate radius $\epsilon_m$, we solve the problem~\eqref{prob-sum} with $\ell \equiv 0$. The smallest $\epsilon_m$ such that the optimal value of the problem~\eqref{prob-sum} is not $\infty$ is then identified as the smallest radius for which~\eqref{eq-radiiregion2} holds.
% \end{enumerate}
% }

%%%%%%%%%%%%%%%%%%%%%%%%%%%%%%%%%%%%%%%%%%%%%%%%%%%%%%%%%%
\section{Beyond Expectation: Conditional Risk Measures}\label{sec:extension}
The preceding sections optimize conditional expected loss. The same framework also applies to conditional risk measures that admit a variational representation. This section states the required minimax result and shows that such risk measures can be handled by the tractable counterparts developed above.

Consider risk functionals of the form
\begin{align}\label{eq-rho}
\mathcal R_{Q}(Z)=\inf_{\mathbf t\in\R^n} \E_Q[\phi(Z,\mathbf t)],~~Q\in\mathcal P(\R),
\end{align}
where $\phi$ is a real-valued function on $\R^{n+1}$. 
Let $\ell:\Xi\to\R$ and $\mathcal N\subseteq \Xi$, and suppose now that $Q\in\mathcal P(\Xi)$.
The corresponding conditional risk is
\begin{align*}
\mathcal R_Q(\ell(\bxi)\mid\bxi\in\mathcal N)=\inf_{\mathbf t\in\R^n} \E_Q[\phi(\ell(\bxi),\mathbf t)\mid\bxi\in\mathcal N].
\end{align*}
The associated worst-case conditional risk over an ambiguity set $\mathcal P\subseteq\mathcal P(\Xi)$ is
\begin{align}\label{prob:conDROriskmeasure}
\sup_{Q\in\mathcal P}\mathcal R_Q(\ell(\bxi)\mid\bxi\in\mathcal N)=\sup_{Q\in\mathcal P}\inf_{\mathbf t\in\R^n} \E_Q[\phi(\ell(\bxi),\mathbf t)\mid\bxi\in\mathcal N].
\end{align}
% A core idea is that if the sup and inf can be exchanged in \eqref{prob:conDROriskmeasure}, then we can first solve the supremum  conditional expectation problem, which has been addressed in our paper and the literature for different ambiguity sets. 
If the supremum and infimum in \eqref{prob:conDROriskmeasure} can be interchanged, then the problem reduces to a family of worst-case conditional expectation problems of the type studied in the preceding sections.

We impose two assumptions on $\phi$ to justify this interchange.

\begin{assumption}\label{ass:phi1}
The function $\phi(z,\mathbf t)$ is convex and lower-semicontinuous in $\mathbf t$ for any fixed $z\in\R$.
\end{assumption}

\begin{assumption}\label{ass:phi2}
For every probability law $P$ on $\mathbb R$ with $R(P,\mathbf t):=\E_P[\phi(Z,\mathbf t)]<\infty$ for at least one $\mathbf t$, the function $R(P,\mathbf t)$ is coercive in $\mathbf t$; that is, $R(P,\mathbf t)\to\infty$ whenever $\|\mathbf t\|\to\infty$.
\end{assumption}

% \begin{example}[Optimized Certainty Equivalent]
% The optimized certainty equivalent (OCE) introduced by  \cite{BT07} has the form of \eqref{eq-rho} with $\phi(z,t)=t+u(z-t)$, where $u:\R\to\R$ is increasing and convex. Note that the left derivative exist for $u$ as it is convex, denoted as $u'_-$, and we assume that $\lim_{z\to\infty}u_-'(z)>1$ and $\lim_{z\to-\infty}u_-'(z)<1$. Under this setting, it is straightforward to verify that Assumption \ref{ass:phi1} and \ref{ass:phi2} hold. In particular, conditional value at risk (CVaR) is a special case of the optimized certainty equivalent with $\phi(z,t)=t+\frac{1}{1-\delta}(z-t)_+$ for $\delta\in(0,1)$.
% \end{example}

These assumptions are satisfied by a broad class of risk functionals used in statistics and finance. We illustrate this with three classical examples.

{\color{black}
\begin{example}[Newsvendor problem]
The classical newsvendor problem fits the form in \eqref{eq-rho}. Let $Z$ denote random demand, and let $t$ be the order quantity. With overage cost $h>0$ and underage cost $b>0$, the newsvendor loss is
$
\phi(z,t)=h(t-z)_+ + b(z-t)_+ .
$
For fixed $z$, the function $t\mapsto \phi(z,t)$ is convex and continuous, and hence satisfies Assumption \ref{ass:phi1}. Assumption \ref{ass:phi2} also holds as $\lim_{|t|\to\infty}\phi(z,t)=\infty$ for any $z$.
\end{example}
}

\begin{example}[Optimized Certainty Equivalent]
The optimized certainty equivalent (OCE), introduced by \cite{BT07}, is a special case of \eqref{eq-rho} with $\phi(z, t) = t + u(z - t)$ for an increasing and convex function $u: \mathbb{R} \to \mathbb{R}$. Hence, $\phi$ satisfies Assumption \ref{ass:phi1}.
Since $u$ is convex, its left derivative $u'_-$ exists everywhere. If $\lim_{z \to \infty} u'_-(z) > 1$ and $\lim_{z \to -\infty} u'_-(z) < 1$, then Assumption \ref{ass:phi2} also holds. 
Conditional value-at-risk (CVaR) is the special case that $\phi(z, t) = t + (z - t)_+/(1 - \eta)$ for confidence level $\eta \in (0,1)$, and satisfies both assumptions. 
\end{example}

\begin{example}[Deviation Measures]
\cite{RU13} proposed a general framework for deviation measures using risk quadrangles. Deviation measures induced by expectation quadrangles can be expressed in the form of \eqref{eq-rho} by setting $\phi(z, t) = u(z - t)$ for a suitable function $u:\R\to\R$. In this context, \cite{RU13} interpreted $\E_Q[u(Z)]$ as quantifying the ``nonzeroness" of the random variable $Z$, making Assumption \ref{ass:phi2} natural.
Mean absolute deviation, with $\phi(z, t) = |z - t|$, and variance, with $\phi(z, t) = (z - t)^2$, both fall within this framework, and satisfy Assumptions \ref{ass:phi1} and \ref{ass:phi2}.
\end{example}

% \begin{example}[Deviation measures]
% \cite{RU13} introduced a criterion to define  deviation measures through risk quadrangles. Deviation measures induced by expectation quadrangles can be defined as the form in \eqref{eq-rho} with $\phi(z,t)=u(z-t)$. Since \cite{RU13} suggested that $\E_Q[\phi(Z)]$ measures the ``nonzeroness" of $Z$, our Assumptions \ref{ass:phi1} and \ref{ass:phi2} are significant in this case. In particular, the mean absolute deviation and  the variance are special cases with $\phi(z,t)=|z-t|$ and $\phi(z,t)=(z-t)^2$, respectively.
% \end{example}

We now establish the minimax result for conditional risk functionals.
\begin{theorem}\label{th-minimax}
Let $A\subseteq (0,1]$ be an interval, and let $\mathcal P\subseteq \mathcal P(\Xi)$, $\ell:\Xi\to\R$, and $\phi:\R^{n+1}\to\R$. Define $\mathcal P_{\mathcal N}^{\ell}=\{Q^{\mathcal N}\circ\ell^{-1}: Q\in\mathcal P,\ Q(\mathcal N)\in A\}$. Suppose that $\mathcal P$ is convex, $\mathcal P_{\mathcal N}^{\ell}$ is nonempty, $R(P,\mathbf t)=\E_P[\phi(Z,\mathbf t)]$ is finite for every $P\in\mathcal P_{\mathcal N}^{\ell}$ and every $\mathbf t\in\R^n$, and $\inf_{\mathbf t\in\R^n}\sup_{P\in\mathcal P_{\mathcal N}^{\ell}} R(P,\mathbf t)<\infty$. Suppose further that $\phi$ satisfies Assumptions \ref{ass:phi1} and \ref{ass:phi2}. Then,
\begin{align*}%\label{eq-minimaxDRO}
\sup_{Q\in\mathcal P,Q(\mathcal N)\in A} \mathcal R_Q(\ell(\bxi)\mid\bxi\in\mathcal N)
% &=\sup_{Q\in\mathcal P, Q(\mathcal N)\in A}\inf_{\mathbf t\in\R^n} \E_Q[\phi(\ell(\bxi),\mathbf t)|\bxi\in\mathcal N]\notag\\
=\inf_{\mathbf t\in\R^n} \sup_{Q\in\mathcal P, Q(\mathcal N)\in A}\E_Q[\phi(\ell(\bxi),\mathbf t)\mid\bxi\in\mathcal N].
\end{align*}
\end{theorem}

A related minimax result for unconditional risk functionals is discussed, for example, in \citet[Section~5.2]{ZYG24}. Suppose that $\mathcal P$ is convex and that $\bt\mapsto\phi(z,\bt)$ is convex and lower semicontinuous for every $z$. Consider the following equation:
\begin{align*}
\sup_{Q\in\mathcal P}\inf_{\bt\in\R^n}\E_Q[\phi(Z,\bt)]
=
\inf_{\bt\in\R^n}\sup_{Q\in\mathcal P}\E_Q[\phi(Z,\bt)].
\end{align*}
A direct sufficient condition for this equality is the existence of a common compact set $C\subseteq\R^n$ such that
\begin{align}\label{eq-UCsuffconminmax}
\inf_{\bt\in\R^n}\E_Q[\phi(Z,\bt)]
=
\inf_{\bt\in C}\E_Q[\phi(Z,\bt)]
\qquad
\forall\,Q\in\mathcal P.
\end{align}
Under this uniform compact-reduction condition, the minimization over $\bt$ can be restricted to $C$, after which a standard minimax theorem applies. The same argument can be used in the conditional setting by replacing $\mathcal P$ with the induced family of conditional loss distributions $\mathcal P_{\mathcal N}^{\ell}$.

However, condition~\eqref{eq-UCsuffconminmax} depends jointly on the structure of $\phi$ and the ambiguity set $\mathcal P$, since a single compact set must work uniformly for all $Q\in\mathcal P$. Verifying such a uniform condition may be difficult, particularly for ambiguity sets with complex multi-source constructions. By contrast, Assumptions~\ref{ass:phi1} and \ref{ass:phi2} are formulated independently of the particular ambiguity set. Together with the convexity and nonemptiness of $\mathcal P_{\mathcal N}^{\ell}$ and the finiteness conditions in Theorem~\ref{th-minimax}, these assumptions justify the minimax interchange without requiring the explicit construction of a compact set that is uniform over the entire ambiguity set.

% Note that the multi-Source ambiguity sets, denoted as $\mathcal B^{\rm I}$, $\mathcal B^{\rm II}$ and $\mathcal B^{\rm III}$, considered in this paper are all convex. By Theorem {th-minimax}, the conditional robust optimization problem in \eqref{prob-mainprob} with expectation replaced by the risk functional defined in \eqref{eq-rho} still admits the form of \eqref{prob-mainprob} with additional decision variable $\bt$ and the loss function replaced by $\phi(\ell(\bxi,\bt))$. This direct result is presented in the following corollary, where its proof is omitted.

The multi-source ambiguity sets $\mathcal{B}^{\rm I}$, $\mathcal{B}^{\rm II}$, and $\mathcal{B}^{\rm III}$ are convex. By Theorem~\ref{th-minimax}, replacing the conditional expectation %\eqref{prob-mainprob} 
by the risk functional in \eqref{eq-rho} preserves the same optimization structure: one adds the auxiliary decision variable $\bt$ and replaces the loss by $\phi(\ell(\bbeta,\bxi),\bt)$. The following corollary records this implication; the proof is immediate from Theorem~\ref{th-minimax}.

\begin{corollary}\label{co:beyondE}
Let $\mathcal B\in \{\mathcal{B}^{\rm I}, \mathcal{B}^{\rm II}, \mathcal{B}^{\rm III}\}$, and let the risk functional $\mathcal R$ be defined by \eqref{eq-rho}, where $\phi$ satisfies Assumptions \ref{ass:phi1} and \ref{ass:phi2}. 
Define
$
\mathcal P_{\mathcal N}^{\ell(\bbeta,\cdot)}
=
\left\{
Q^{\mathcal N}\circ(\ell(\bbeta,\cdot))^{-1}:
Q\in\mathcal P,\ Q(\mathcal N)\in[a,b]
\right\}.
$
Suppose that $R(P,\bt):=\E_P[\phi(Z,\bt)]$ is finite for every $P\in\mathcal P_{\mathcal N}^{\ell(\bbeta,\cdot)}$, $\bbeta\in\mathcal D$ and $\bt\in\R^n$, and that
$
\inf_{\bt\in\R^n}
\sup_{P\in\mathcal P_{\mathcal N}^{\ell(\bbeta,\cdot)}}
R(P,\bt)<\infty
$
for every $\bbeta\in\mathcal D$.
Then, the conditional risk problem
\begin{align*}
\inf_{\bbeta \in \mathcal{D}} \;
\sup_{\substack{Q \in \mathcal{B}\left(\{ \widehat{\mathbb{P}}_k \}_{k \in [K]}, \boldsymbol{\epsilon} \right) \\ Q(\mathcal{N}) \in [a, b]}} 
\mathcal R_Q(\ell(\bbeta, \boldsymbol{\xi}) \mid \boldsymbol{\xi} \in \mathcal{N})
\end{align*}
is equivalent to
\begin{align*}
\inf_{\bbeta\in\mathcal D,\mathbf t\in\R^n} \sup_{\substack{Q \in \mathcal{B}\left(\{ \widehat{\mathbb{P}}_k \}_{k \in [K]}, \boldsymbol{\epsilon} \right) \\ Q(\mathcal{N}) \in [a, b]}}\E_Q[\phi(\ell(\bbeta,\bxi),\mathbf t)\mid\bxi\in\mathcal N].
\end{align*}
\end{corollary}

\section{Numerical Experiments}\label{sec:num}

In this section, we conduct two numerical experiments for an assortment optimization problem. The first is a controlled simulation with one target distribution and two heterogeneous source distributions, covering multiple combinations of source sample sizes and 50 independent replications. The second uses Walmart sales data from the M5 competition and follows a rolling, look-ahead-free evaluation procedure. In both experiments, model parameters are selected using separate validation observations, and the selected models are evaluated on held-out test observations. We report both conditional and unconditional results for the seven methods considered.

\subsection{Assortment Model and Competing Methods}
\label{sec:num-assortment}

The decision maker selects at most $B$ products from a universe of $d$
substitutable products.  Let
$$
\mathcal D_B:=\left\{\bbeta\in\{0,1\}^d:
\sum_{j=1}^d v_j\le B\right\}
$$
denote the feasible assortments, where $v_j=1$ means that product $j$ is
offered. Denote by $\bxi=(\bX,\bY)$, where $\bY\in\R_+^d$ represents random demand shares and $\bX$ is the vector of covariates.
For prices $\bp\in\mathbb R_+^d$ and $\bY\in\mathbb R_+^d$, the realized revenue is
$R(\bbeta,\bY)=\sum_{j=1}^d p_jv_jY_j$.  If the target distribution $\P_T$ were
known, the state-dependent assortment at $\bX=\bx_0$ would solve
\begin{equation*}
\max_{\bbeta\in\mathcal D_B}
\E_{\P_T}\!\left[R(\bbeta,\bY)\mid\bX=\bx_0\right].
%\label{eq:num-oracle}
\end{equation*}

We replace the unknown target law by an ambiguity set $\mathcal B$.
% associated with model family $m$ and candidate parameter vector $h$.  
The
conditional and unconditional prescriptions used in the experiments are
\begin{align*}
&\widehat\bbeta^{\,\mathrm C}(\bx_0)
\in\argmax_{\bbeta\in\mathcal D_B}\;
\inf_{\substack{Q\in\mathcal B\\
Q(\bX=\bx_0)\in[a,1]}}
\E_Q\!\left[R(\bbeta,\bY)\mid\bX=\bx_0\right];\\
%\label{eq:num-conditional}\\
&\widehat\bbeta^{\,\mathrm U}
\in\argmax_{\bbeta\in\mathcal D_B}\;
\inf_{Q\in\mathcal B}\E_Q[R(\bbeta,\bY)].
%\label{eq:num-unconditional}
\end{align*}
% Thus, $a$ is a lower bound on the probability assigned to the conditioning
% event; it is not a bandwidth parameter.  
Conditional models use a separate form of
$\ell_1$ transportation cost on the joint vector $\bxi=(\bX,\bY)$, 
that is, $c(\bxi_1,\bxi_2)=\|\bX_1-\bX_2\|_{\mathcal X}+\|\bY_1-\bY_2\|_{\mathcal Y}$,
whereas their
unconditional counterparts use the $\ell_1$ cost on $\bY$.  We enumerate all
full-budget assortments and solve the reformulations developed in
Sections~\ref{sec:tractbilityB1}--\ref{sec:tractabilityB3}.

We compare seven model families.  \emph{Intersection}, \emph{Weighted distance}, and
\emph{Barycenter} correspond to ambiguity sets $\mathcal B^{\rm I}$,
$\mathcal B^{\rm II}$, and $\mathcal B^{\rm III}$, respectively.
\emph{Single S} and \emph{Single J} place one OT ball around the empirical law
of the corresponding source.  The \emph{target benchmark} uses the same
single-center construction around an independent target sample and provides
an information benchmark.
\emph{Pooled method} discards source labels and centers one OT ball at the combined
empirical law
\begin{equation*}
\widehat \P_{\rm pool}
=\frac{1}{N_S+N_J}\sum_{k\in\{S,J\}}\sum_{j=1}^{N_k}\delta_{\widehat{\bxi}_{k,j}}.
%\label{eq:num-centers}
\end{equation*}
% Except for the prespecified $50{:}50$ sensitivity design in the controlled
% experiment, each pooled observation receives equal mass, so that
% $\omega_S=N_S/(N_S+N_J)$.  We set $\omega_S=0.40$ for that $50{:}50$ design
% and $\omega_S=8/38$ in the M5 experiment.
Following the heterogeneous-data construction in \citet{REMK24}, the
Barycenter method centers a single OT ball at an equal-weight
1-Wasserstein barycenter,
\begin{equation*}
\widehat \P_{\rm bar}\in
\argmin_{Q}\left\{
\frac{1}{2}W(Q,\widehat \P_S)+
\frac{1}{2}W(Q,\widehat \P_J)\right\}.
%\label{eq:num-barycenter}
\end{equation*}
With two empirical measures and the 1-Wasserstein distance, both source
measures are valid barycenter representatives.  We therefore treat the center
choice $\widehat \P_{\rm bar}\in\{\widehat \P_S,\widehat \P_J\}$ as a tuning
decision and select it on the validation sample together with the radius.

\paragraph{\bf Feasible radii and validation.}
The conditional ambiguity set must contain at least one distribution assigning
probability in $[a,1]$ to $\bX=\bx_0$.  We compute the smallest admissible
radii from the feasibility programs in Section~\ref{sec:feasible}.  Let $\epsilon_{\min}(\P;\bx_0,a)$ denote the minimum feasible radius associated with a single reference distribution $\P$. For the intersection model, let $\epsilon_{J,\min}(\epsilon_S;\bx_0,a)$ denote the minimum radius around $\widehat\P_J$ when the radius around $\widehat\P_S$ is fixed at $\epsilon_S$. This boundary is obtained from the optimization problem in \eqref{eq-feasibleradiiI} of Corollary~\ref{co:special}. Finally, let $\epsilon_{\min}^{\rm II}(\bx_0,a,\theta)$ denote the corresponding feasibility boundary for the weighted model, as determined by the optimization problem in \eqref{eq-regionradiiII}.
Writing
$D=W(\widehat \P_S,\widehat \P_J)$, candidate values $(a,\theta,\delta)$ are
mapped to conditional radii as follows:
\begin{equation*}
\begin{aligned}
&\epsilon_S
=\max\!\left\{\theta D,\,
\epsilon_{\min}(\widehat \P_S;\bx_0,a)+10^{-6}\right\},
&
\epsilon_J
&=\epsilon_{J,\min}(\epsilon_S;\bx_0,a)+\delta\epsilon_S,
&&\text{Intersection},\\
&\epsilon
=\epsilon_{\min}^{\rm II}(\bx_0,a,\theta)+\delta,
&&&&\text{Weighted distance},\\
&\epsilon
=\epsilon_{\min}(\P;\bx_0,a)+\delta,
&&&&\text{single-center models}.
\end{aligned}
%\label{eq:num-radius-conditional}
\end{equation*}
Consequently, $\delta$ measures slack above an endogenously determined
feasibility boundary.  For the unconditional comparison, let
$D_{\bY}=W(\widehat \P_S^{\bY},\widehat \P_J^{\bY})$, where $\widehat \P_S^{\bY}$ and $\widehat \P_J^{\bY}$ are marginal distributions of $\widehat \P_S$ and $\widehat \P_J$ on $\bY$, respectively.  Intersection uses
$$
(\epsilon_S,\epsilon_J)
=\bigl(\theta(1+\delta)D_{\bY},\,
(1-\theta)(1+\delta)D_{\bY}\bigr),
$$
weighted model uses $\epsilon=\min\{\theta,1-\theta\}D_{\bY}+\delta$, and every
single-center model uses $\epsilon=\delta$.

Table~\ref{tab:num-grid} gives the complete candidate grid used in both
experiments. 

\begin{table}[t]
\caption{Candidate parameters used for validation}\label{tab:num-grid}
\centering
%\footnotesize
\setlength{\tabcolsep}{4pt}
\begin{tabular}{lcccc}
\hline
Model & $a$ & $\theta$ or center & Conditional $\delta$
& Unconditional $\delta$\\
\hline
Target benchmark & $A$ & -- & $D_{\rm C}$ & $D_{\rm U}$\\
Intersection & $A$ & $\Theta$ & $D_{\rm I}$ & $D_{\rm I}$\\
Weighted distance & $A$ & $\Theta$ & $D_{\rm W}$ & $D_{\rm W}$\\
Pooled method & $A$ & -- & $D_{\rm C}$ & $D_{\rm U}$\\
Barycenter & $A$ & $\{S,J\}$ & $D_{\rm C}$ & $D_{\rm U}$\\
Single S & $A$ & -- & $D_{\rm C}$ & $D_{\rm U}$\\
Single J & $A$ & -- & $D_{\rm C}$ & $D_{\rm U}$\\
\hline
\end{tabular}
\vspace{0.5ex}

\begin{minipage}{0.96\textwidth}\footnotesize\raggedright
\emph{Notes.}
$A=\{0.05,0.10,0.15\}$,
$\Theta=\{0.10,0.30,0.50,0.70,0.90\}$,
$D_{\rm C}=D_{\rm I}=D_{\rm W}=\{10^{-4}, 10^{-3},10^{-2},10^{-1},1\}$, and
$D_{\rm U}=\{10^{-4},5\times10^{-4},10^{-3},10^{-2},10^{-1},
1,5,10,20,50\}$.
\end{minipage}
\end{table}

% For model $m$, replication or fold $r$, and validation sample
% $\{(\bX_i^{\rm val},\bY_i^{\rm val})\}_{i=1}^{n_{\rm val}}$, we select
% \begin{equation}
% \widehat h_{m,r}\in\argmax_{h\in\mathcal H_m}
% \frac{1}{n_{\rm val}}\sum_{i=1}^{n_{\rm val}}
% R\!\left(\widehat\bv_{m,h}(\bX_i^{\rm val}),
% \bY_i^{\rm val}\right),
% \label{eq:num-validation}
% \end{equation}
% where $\widehat\bv_{m,h}$ is the conditional prescription in
% \eqref{eq:num-conditional}; for an unconditional model, the same fixed
% decision $\widehat\bv_{m,h}^{\,\mathrm U}$ is evaluated at every validation
% outcome.  Ties are resolved deterministically by the order of the grid.

\subsection{Controlled simulation}\label{sec:num-controlled}

\paragraph{\bf Data-generating process.}
We consider four products with unit prices and set $B=1$.  The scalar state
has the same marginal distribution under the target and both sources:
$X\sim\mathrm{Unif}[-1,1]$.  Let
$\bc=(-0.75,-0.25,0.25,0.75)$ denote product locations and
$\bg=(1,0.5,-0.5,-1)$ encode the direction of the source distortion.  Conditional
on $X=x$, the purchase probabilities in domain $k\in\{T,S,J\}$ are
\begin{equation}
\pi_j^k(x)=
\frac{\exp\{-5(x-c_j)^2+b_kg_j\}}
{\sum_{s=1}^4\exp\{-5(x-c_s)^2+b_kg_s\}},
\qquad
b_T=0,\quad b_S=3,\quad b_J=-3.
\label{eq:num-dgp}
\end{equation}
The target conditional law is thus located between two sources with
equal-and-opposite biases.  After drawing
$\mathbf C\sim\mathrm{Multinomial}(80,\bm\pi^k(x))$, we set
$\bY=\mathbf C/80$.  The optimal product changes with $x$, while neither
source is uniformly closer to the target over the entire state space.

\paragraph{\bf Sampling and evaluation.}
We study the following combinations of sample sizes:
$$
(N_S,N_J)\in
\{(10,10),(10,30),(10,50),(30,30),(30,50),(50,50)\}.
$$
For every size pair, we generate 50 independent replications.  Single S and
Single J receive $N_S$ and $N_J$ observations, respectively.  The target
benchmark receives $N_S+N_J$ independent target observations, matching the
total amount of data available to a multi-source method.

Within each replication, an additional 20 target observations are used only for validation.  
% Equation~\eqref{eq:num-validation} is computed from the 20
% realized rewards for every candidate in Table~\ref{tab:num-grid}.  
The selected
candidate is then evaluated on 100 new target observations.  Source,
benchmark, validation, and test samples are mutually independent.
Table~\ref{tab:num-controlled} reports the mean out-of-sample normalized revenue, together with its standard error over 50 replications.

% If $Z_{m,r}$ denotes the average of the 100 test
% rewards for method $m$ in replication $r$, Table~\ref{tab:num-controlled}
% reports
% \begin{equation}
% \overline Z_m=\frac{1}{50}\sum_{r=1}^{50}Z_{m,r},
% \qquad
% \operatorname{SE}(\overline Z_m)=
% \left\{\frac{1}{50\cdot49}
% \sum_{r=1}^{50}(Z_{m,r}-\overline Z_m)^2\right\}^{1/2}.
% \label{eq:num-se}
% \end{equation}

\begin{table}[ht]
\caption{Out-of-sample normalized revenue
(mean and standard error over 50 replications)}
\label{tab:num-controlled}
\centering
\scriptsize
\setlength{\tabcolsep}{1.5pt}
\begin{tabular*}{\textwidth}{
@{\extracolsep{\fill}}cc@{\hspace{3pt}\vrule\hspace{5pt}}ccccccc@{}}
\hline
$N_S$ & $N_J$ &
\shortstack{Target benchmark} &
Intersection & Weighted distance & Pooled method & Barycenter &
\shortstack{Single S} & \shortstack{Single J}\\
\hline
\multicolumn{9}{l}{\emph{Panel A: Conditional models}}\\[0.2ex]
10 & 10 & 65.7 (0.18) & \textbf{59.4 (0.94)} & 57.9 (1.09) & 55.4 (0.81) & 46.2 (0.78) & 45.0 (0.97) & 43.6 (0.75)\\
10 & 30 & 66.0 (0.21) & 59.7 (1.32) & \textbf{59.8 (0.97)} & 54.4 (0.78) & 44.8 (0.67) & 43.3 (0.86) & 43.0 (0.64)\\
10 & 50 & 66.2 (0.14) & \textbf{61.4 (1.07)} & 60.8 (0.98) & 51.6 (0.58) & 45.5 (0.73) & 44.0 (0.92) & 45.0 (0.46)\\
30 & 30 & 65.9 (0.20) & \textbf{61.2 (0.88)} & 60.5 (0.81) & 60.3 (0.63) & 45.5 (0.50) & 44.6 (0.58) & 45.2 (0.52)\\
30 & 50 & 66.1 (0.19) & \textbf{62.6 (0.61)} & 62.0 (0.63) & 60.8 (0.63) & 44.5 (0.53) & 43.7 (0.57) & 44.6 (0.54)\\
50 & 50 & 65.8 (0.19) & \textbf{62.8 (0.41)} & 62.7 (0.31) & 61.2 (0.57) & 45.0 (0.41) & 44.8 (0.43) & 44.9 (0.41)\\[0.5ex]
\multicolumn{9}{l}{\emph{Panel B: Unconditional models}}\\[0.2ex]
10 & 10 & 25.3 (0.40) & 24.6 (0.41) & 24.7 (0.41) & 25.1 (0.42) & 24.6 (0.41) & \textbf{25.3 (0.39)} & 24.3 (0.43)\\
10 & 30 & 24.7 (0.37) & \textbf{25.1 (0.38)} & 24.7 (0.38) & 24.8 (0.37) & 24.9 (0.38) & 24.6 (0.43) & 24.7 (0.38)\\
10 & 50 & 24.6 (0.41) & \textbf{25.3 (0.41)} & 25.1 (0.44) & 24.9 (0.44) & 24.8 (0.45) & 24.1 (0.47) & 24.7 (0.45)\\
30 & 30 & 24.4 (0.42) & \textbf{25.5 (0.39)} & 25.1 (0.41) & 24.8 (0.41) & 24.8 (0.43) & 25.0 (0.40) & 24.6 (0.43)\\
30 & 50 & 24.2 (0.47) & 25.1 (0.46) & \textbf{25.2 (0.45)} & 24.9 (0.50) & 25.1 (0.48) & 24.0 (0.49) & 24.8 (0.50)\\
50 & 50 & 24.4 (0.48) & \textbf{24.8 (0.45)} & 24.6 (0.47) & 24.7 (0.44) & 24.3 (0.49) & 24.4 (0.47) & 23.8 (0.49)\\
\hline
\end{tabular*}
\vspace{0.5ex}

\begin{minipage}{0.96\textwidth}\footnotesize\raggedright
\emph{Notes.}
All reported means and standard errors have been multiplied by 100.
Boldface marks the best data-driven method in that row, excluding the target
benchmark.
\end{minipage}
\end{table}

Panel~A of Table~\ref{tab:num-controlled} shows that the Intersection and Weighted models outperform all other conditional methods across every size pair.
Intersection ranks first in five of the six designs, while Weighted is
marginally better at $(N_S,N_J)=(10,30)$.  
The difference is economically
meaningful when the source sizes are unbalanced.  
At $(N_S,N_J)=(10,50)$, for example,
Intersection and Weighted achieve 0.614 and 0.608, compared with 0.516 for
Pooled and at most 0.450 for either single-source method.  Because Source J is both larger and systematically biased, assigning equal weight to every pooled observation gives this source a disproportionate influence and shifts the pooled reference distribution away from the target. By contrast, Intersection and Weighted retain the two empirical source distributions separately: Intersection imposes an individual Wasserstein constraint relative to each source, whereas Weighted controls a weighted combination of the two source-wise Wasserstein distances.

The advantage also becomes more stable as both source laws are estimated more
accurately.  From $(10,10)$ to $(50,50)$, the Intersection mean rises from
0.594 to 0.628 and its standard error decreases from 0.0094 to 0.0041.
Panel~B provides a useful contrast: all unconditional rewards remain close to
0.25 and no method dominates uniformly.  Hence, the gains in Panel~A are not
explained simply by access to more observations.  They arise from combining
source-level restrictions with the state-dependent preference structure in
\eqref{eq:num-dgp}.

\subsection{M5 Rolling Experiment}\label{sec:num-m5}

\paragraph{\bf Data and state construction.}
The second experiment uses the public M5 Forecasting--Accuracy data, which
contain daily Walmart unit sales, weekly sell prices, and calendar variables.
The target store is WI\_3 and the source stores are CA\_3 and TX\_2.  We retain
six items from department FOODS\_1:
FOODS\_1\_032, FOODS\_1\_004, FOODS\_1\_204, FOODS\_1\_043,
FOODS\_1\_046, and FOODS\_1\_018.  The decision maker selects $B=2$ products.
For each store-day, item sales are multiplied by contemporaneous prices and
normalized by total revenue across the six items.  The realized reward is
therefore the fraction of six-item revenue captured by the selected
assortment.

The conditioning variable is a scalar pre-decision market state.  Its raw
features include day-of-week and month effects, SNAP and event indicators,
current sell prices, lagged total revenue, and lagged and seven-day rolling
revenue shares.  At each validation or test origin, these features are
standardized using only the trailing estimation window.  On the target store,
ridge regression maps the standardized features to the first principal
component of historical revenue shares.  The fitted scalar index is then
standardized on the same target history and applied to both source stores.
All demand-based inputs are lagged through day $t-1$, so the state used for the
decision on day $t$ contains no contemporaneous demand information.

\paragraph{\bf Rolling protocol.}
The department, six products, and two sources are fixed using the initial
design period $d_1$--$d_{1000}$.  Hyperparameters are selected from
Table~\ref{tab:num-grid} using 18 chronological validation origins, arranged
as three folds of six evenly spaced dates.  At every origin $t$, the state map
and empirical distributions are re-estimated from the 730-day window ending
at $t-1$.  State-stratified sampling then selects $N_S=8$ observations from
CA\_3, $N_J=30$ from TX\_2, and 60 target observations for the benchmark.
The validation score is the average realized revenue over the 18 origins, and
one candidate is retained for each model.

The selected candidate is held fixed over 40 evenly spaced test origins from
$d_{1531}$ to $d_{1913}$.  At every test date we roll the 730-day window
forward and recompute the state map, empirical samples, ambiguity set, and
optimal assortment using only information available before the test outcome.
Table~\ref{tab:num-m5} reports the average of the 40 test rewards.  Parentheses
contain standard errors across test origins.  
% Because these rewards constitute
% one sequential backtest rather than independent replications, the reported
% standard errors summarize temporal variation and should not be interpreted as
% independent Monte Carlo uncertainty.

\begin{table}[t]
\caption{Out-of-sample normalized revenue
(mean and standard error over 40 test origins)}
\label{tab:num-m5}
\centering
\begin{tabular}{lcc}
\hline
Model & Conditional & Unconditional\\
\hline
Target benchmark~~~ & .547 (.031) & .410 (.043)\\
Intersection & \textbf{.509 (.035)} & .360 (.041)\\
Weighted distance & .504 (.034) & .409 (.043)\\
Pooled method & .467 (.038) & \textbf{.416 (.042)}\\
Barycenter & .454 (.038) & .360 (.041)\\
Single S & .480 (.038) & .360 (.041)\\
Single J & .454 (.038) & .400 (.032)\\
\hline
\end{tabular}
\vspace{0.5ex}

\begin{minipage}{0.82\textwidth}\footnotesize\raggedright
\emph{Notes.} Boldface marks the best data-driven model in each column,
excluding the target benchmark.
\end{minipage}
\end{table}

The conditional results in Table~3 are consistent with those of the controlled experiment. Excluding the target benchmark, Intersection and Weighted achieve the highest average rewards, at 0.509 and 0.504, respectively. 
% These values represent improvements of 9.1\% and 8.0\% over conditional Pooled. Both methods also outperform the conditional Barycenter and the two single-source benchmarks. 
These findings suggest that retaining source-specific distributional information can improve conditional decisions even when the source samples are small and unbalanced, with $(N_S,N_J)=(8,30)$, and the conditioning state must be estimated from historical sales data.

The unconditional results provide a different perspective from the conditional comparison. Among the unconditional methods, Pooled method achieves the highest average reward of 0.416, followed by Weighted at 0.409, whereas Intersection achieves 0.360. Because these methods all use the same two source samples, the superior conditional performance of Intersection and Weighted cannot be attributed simply to their access to multiple datasets. The results suggest that their advantage is associated with retaining source-specific distributional information while adapting the assortment decision to the observed market state.

\appendix

\section{Proofs of results in Section \ref{sec:framework}}

\begin{proof}[Proof of Proposition \ref{prop:weightset}]
For $Q\in \mathcal B^{\rm{II}}(\{\widehat{\p}_k\}_{k\in[K]},\epsilon_{\min})$, we have
\begin{align*}
\epsilon_{\min}\ge \sum_{k=1}^K  \theta_k W(Q,\widehat{\p}_k)\ge \theta_kW(Q,\widehat{\p}_k)=\frac{\epsilon_{\min}}{\epsilon_k}W(Q,\widehat{\p}_k)~~\forall k\in[K].
\end{align*}
This yields
$
W(Q,\widehat{\p}_k)\le \epsilon_k
$
for all $k\in[K]$.
Hence, $Q\in \mathcal B^{\rm I}(\{\widehat{\p}_k\}_{k\in[K]},\bm\epsilon)$, which implies that $\mathcal B^{\rm{II}}(\{\widehat{\p}_k\}_{k\in[K]},\epsilon_{\min})\subseteq\mathcal B^{\rm I}(\{\widehat{\p}_k\}_{k\in[K]},\bm\epsilon)$.

For $Q\in \mathcal B^{\rm I}(\{\widehat{\p}_k\}_{k\in[K]},\bm\epsilon)$, we have
$
{ W(Q,\widehat{\p}_k)}/{\epsilon_k}\le 1
$
for all $ k\in[K]$,
and thus,
\begin{align*}
\sum_{k=1}^K  \theta_k W(Q,\widehat{\p}_k)=\epsilon_{\min}\sum_{k=1}^K  \frac{ W(Q,\widehat{\p}_k)}{\epsilon_k}\le K\epsilon_{\min}.
\end{align*}
This yields $Q\in \mathcal B^{\rm II}(\{\widehat{\p}_k\}_{k\in[K]},K\epsilon_{\min})$, which further implies that 
$$\mathcal B^{\rm I}(\{\widehat{\p}_k\}_{k\in[K]},\bm\epsilon)\subseteq\mathcal B^{\rm{II}}(\{\widehat{\p}_k\}_{k\in[K]},K\epsilon_{\min}).$$ Hence, we complete the proof.
\end{proof}

\section{Proofs of results in Section \ref{sec:tractbilityB1}}

\subsection{Proof of Theorem \ref{th-analysisB1}}

The following decomposition result will be used in the proofs of both Theorems~\ref{th-analysisB1} and \ref{th-analysisB2}.

\begin{lemma}[Multi-margin decomposition]\label{lm:DMM}
For any $\overline{\pi}\in \Pi(\mathbb P,\widehat{\mathbb P}_1,\ldots,\widehat{\mathbb P}_K)$, there exist finite Borel measures $\nu_{\balpha}\in\mathcal M_+(\Xi)$, $\balpha\in\mathcal A$, such that
\begin{align*}
\overline{\pi}
=\sum_{\balpha\in\mathcal A}
\nu_{\balpha}\otimes
\delta_{(\widehat{\bxi}_{1,\alpha_1},\ldots,\widehat{\bxi}_{K,\alpha_K})}.
\end{align*}
\end{lemma}

\begin{proof}[Proof of Lemma \ref{lm:DMM}]
After aggregating repeated observations, which does not change the empirical laws or the OT balls, we may treat the atoms in each empirical marginal as distinct. Because each empirical marginal $\widehat{\mathbb P}_k$ is finitely supported, define for each $\balpha\in\mathcal A$ the finite measure
\begin{align*}
\nu_{\balpha}(A)
:=\overline{\pi}\left(A\times\{\widehat{\bxi}_{1,\alpha_1}\}\times\cdots\times\{\widehat{\bxi}_{K,\alpha_K}\}\right),
\quad A\subseteq\Xi.
\end{align*}
The sets
$\{\widehat{\bxi}_{1,\alpha_1}\}\times\cdots\times\{\widehat{\bxi}_{K,\alpha_K}\}$,
$\balpha\in\mathcal A$, partition the support of the discrete source marginals. Hence, for every measurable rectangle $A\times A_1\times\cdots\times A_K$,
\begin{align*}
\overline{\pi}(A\times A_1\times\cdots\times A_K)
=\sum_{\balpha\in\mathcal A}
\nu_{\balpha}(A)\prod_{k=1}^K\id_{\{\widehat{\bxi}_{k,\alpha_k}\in A_k\}},
\end{align*}
which proves the displayed decomposition by the uniqueness of finite measures. 
\end{proof}

~\\
\begin{proof}[Proof of Theorem \ref{th-analysisB1}]
To simplify the notation, let
$\mathcal B^{\rm I}:=\mathcal B^{\rm I}(\{\widehat{\p}_k\}_{k\in[K]},\bm\epsilon)$,
$\mathcal B_{\mathcal N}^{\rm I}:=\mathcal B_{\mathcal N}^{\rm I}(\{\widehat{\p}_k\}_{k\in[K]},\bm\epsilon)$,
and let $\mathcal R$ denote the set on the right-hand side of
\eqref{eq-mainthB1-exact}. Recall that
\begin{align*}
\mathcal B_{\mathcal N}^{\rm I}
=\{P^{\mathcal N}:P\in\mathcal B^{\rm I},\ P(\mathcal N)\in[a,b]\}.
\end{align*}
We first derive an equivalent representation of the joint laws in $\mathcal B^{\rm I}$ whose mass on $\mathcal N$ lies in $[a,b]$. This representation will be used in both directions of the proof. By the optimal-coupling result of \citet[Theorem 4.1]{V08}, $W(P,\widehat{\p}_k)\le\epsilon_k$ if and only if there exists $\pi_k\in\Pi(P,\widehat{\p}_k)$ with
$\E_{\pi_k}[c(\bxi,\bxi_k)]\le\epsilon_k$. Given such plans
$\{\pi_k\}_{k\in[K]}$, the gluing lemma yields
$\overline{\pi}\in\Pi(P,\widehat{\p}_1,\ldots,\widehat{\p}_K)$ whose
$(\bxi,\bxi_k)$-marginal is $\pi_k$ for every $k$. Therefore,
\begin{align*}
\left\{P\in\mathcal B^{\rm I}:P(\mathcal N)\in[a,b]\right\}
=\left\{
\begin{array}{l}
P\in\mathcal P(\Xi):~\overline{\pi}\in
\Pi(P,\widehat{\p}_1,\ldots,\widehat{\p}_K),~P(\mathcal N)\in[a,b],\\
\displaystyle
\int_{\Xi^{K+1}}c(\bxi,\bxi_k)\,d\overline{\pi}(\bxi,\bxi_1,\ldots,\bxi_K)
\le\epsilon_k,\quad k\in[K]
\end{array}
\right\}.
\end{align*}
Applying Lemma~\ref{lm:DMM}, the same set is equivalently the set of laws
$P=\sum_{\balpha\in\mathcal A}\nu_{\balpha}$ for which
\begin{align}
&\nu_{\balpha}\in\mathcal M_+(\Xi) && \forall \balpha\in\mathcal A,\label{eq-capnufinitemeasure}\\
&\sum_{\balpha\in\mathcal A:\alpha_k=j}\nu_{\balpha}(\Xi)=w_{k,j}
&& \forall j\in[N_k],~k\in[K],\label{eq-capsumnu}\\
&\sum_{\balpha\in\mathcal A}\nu_{\balpha}(\mathcal N)\in[a,b],\label{eq-capconstraintprob}\\
&\sum_{\balpha\in\mathcal A}\int_{\Xi}c(\bxi,\widehat{\bxi}_{k,\alpha_k})\,d\nu_{\balpha}(\bxi)
\le\epsilon_k
&& \forall k\in[K].\label{eq-capconstraint}
\end{align}
Consequently, the induced conditional set can be written as
\begin{align}\label{eq-reBZ}
\mathcal B_{\mathcal N}^{\rm I}
=\left\{
\frac{\sum_{\balpha\in\mathcal A}\nu_{\balpha}(\mathcal N)\nu_{\balpha}^{\mathcal N}}
{\sum_{\balpha\in\mathcal A}\nu_{\balpha}(\mathcal N)}
:\ \eqref{eq-capnufinitemeasure}-\eqref{eq-capconstraint}\ {\rm hold}
\right\},
\end{align}
where terms with $\nu_{\balpha}(\mathcal N)=0$ are immaterial.

We now show $\mathcal B_{\mathcal N}^{\rm I}\subseteq\mathcal R$. Let
$R\in\mathcal B_{\mathcal N}^{\rm I}$, and choose
$\{\nu_{\balpha}\}_{\balpha\in\mathcal A}$ satisfying
\eqref{eq-capnufinitemeasure}-\eqref{eq-capconstraint} and representing
$R$ as in \eqref{eq-reBZ}. Define
\begin{align*}
\begin{array}{ll}
p_{\balpha}:=\nu_{\balpha}(\mathcal N),\quad
q_{\balpha}:=\nu_{\balpha}(\mathcal N^c),\quad
t_{\balpha}:=mp_{\balpha},\quad
r_{\balpha}:=mq_{\balpha},\quad
\mu_{\balpha}:=\nu_{\balpha}^{\mathcal N},
& \balpha\in\mathcal A,\\[2mm]
m:=\displaystyle\frac{1}{\sum_{\balpha\in\mathcal A}p_{\balpha}},\quad
a_{k,j}:=\displaystyle\sum_{\balpha\in\mathcal A:\alpha_k=j}t_{\balpha},
& j\in[N_k],~k\in[K].
\end{array}
\end{align*}
where $\mu_{\balpha}$ may be chosen arbitrarily on $\mathcal N$ when
$p_{\balpha}=0$, since it is then multiplied by $t_{\balpha}=0$.
Then $R=\sum_{\balpha\in\mathcal A}t_{\balpha}\mu_{\balpha}$,
$R(\mathcal N)=1$, $m\in[1/b,1/a]$, and
\begin{align*}
\sum_{j=1}^{N_k}a_{k,j}=1,\quad
a_{k,j}+\sum_{\balpha\in\mathcal A:\alpha_k=j}r_{\balpha}=mw_{k,j},
\quad j\in[N_k],~k\in[K],
\end{align*}
where the second equality follows from \eqref{eq-capsumnu}.

It remains to identify the exact outside-cost variables and verify the inside OT balls. For every $\balpha$ with $r_{\balpha}>0$, define
\begin{align*}
v_{\balpha,k}:=
\int_{\Xi}c(\bxi,\widehat{\bxi}_{k,\alpha_k})\,d\nu_{\balpha}^{\mathcal N^c}(\bxi),
\quad k\in[K].
\end{align*}
Then $\bv_{\balpha}:=(v_{\balpha,1},\ldots,v_{\balpha,K})\in\mathcal V_{\balpha}$ by the definition of $\mathcal V_{\balpha}$. Set
$s_{\balpha}:=r_{\balpha}\bv_{\balpha}$ when $r_{\balpha}>0$ and
$s_{\balpha}:=0$ when $r_{\balpha}=0$. Hence
$s_{\balpha}\in r_{\balpha}\mathcal V_{\balpha}$ for all $\balpha\in\mathcal A$.
Applying \eqref{eq-capconstraint}, we obtain, for every $k\in[K]$,
\begin{align*}
\epsilon_k
&\ge
\sum_{\balpha\in\mathcal A}\int_{\Xi}
c(\bxi,\widehat{\bxi}_{k,\alpha_k})\,\d\nu_{\balpha}(\bxi)\\
&=\sum_{\balpha\in\mathcal A}\int_{\Xi}
c(\bxi,\widehat{\bxi}_{k,\alpha_k})\,\d\left(\frac{t_{\balpha}\mu_{\balpha}+r_{\balpha}\nu_{\balpha}^{\mathcal N^c}}{m}\right)(\bxi)\\
&=
\frac{1}{m}\sum_{\balpha\in\mathcal A}t_{\balpha}
\int_{\Xi}c(\bxi,\widehat{\bxi}_{k,\alpha_k})\,d\mu_{\balpha}(\bxi)
+\frac{1}{m}\sum_{\balpha\in\mathcal A}s_{\balpha,k}\\
&=\frac{1}{m}\sum_{j=1}^{N_k}   \sum_{\balpha\in\mathcal A,\alpha_k=j}
\int_{\Xi}c(\bxi,\widehat{\bxi}_{k,j})\,\d \left(t_{\balpha}\mu_{\balpha}\right)(\bxi)
+\frac{1}{m}\sum_{\balpha\in\mathcal A}s_{\balpha,k}\\
&=
\frac{1}{m}\sum_{j=1}^{N_k}
a_{k,j}\int_{\Xi}c(\bxi,\widehat{\bxi}_{k,j})\,
\d\left(\frac{\sum_{\balpha\in\mathcal A:\alpha_k=j}t_{\balpha}\mu_{\balpha}}{a_{k,j}}\right)(\bxi)
+\frac{1}{m}\sum_{\balpha\in\mathcal A}s_{\balpha,k}\\
&\ge
\frac{1}{m}
W\left(R,\sum_{j=1}^{N_k}a_{k,j}\delta_{\widehat{\bxi}_{k,j}}\right)
+\frac{1}{m}\sum_{\balpha\in\mathcal A}s_{\balpha,k}.
\end{align*}
The last inequality follows from the definition of OT cost and the identity
\begin{align*}
\sum_{j=1}^{N_k}
a_{k,j}\left(
\frac{\sum_{\balpha\in\mathcal A:\alpha_k=j}t_{\balpha}\mu_{\balpha}}{a_{k,j}}
\right)
=\sum_{\balpha\in\mathcal A}t_{\balpha}\mu_{\balpha}=R,
\end{align*}
with zero-$a_{k,j}$ terms omitted. Therefore,
\begin{align*}
W\left(R,\sum_{j=1}^{N_k}a_{k,j}\delta_{\widehat{\bxi}_{k,j}}\right)
\le
m\epsilon_k-\sum_{\balpha\in\mathcal A}s_{\balpha,k},
\quad k\in[K],
\end{align*}
which proves $R\in\mathcal R$.

We next prove the reverse inclusion $\mathcal R\subseteq\mathcal B_{\mathcal N}^{\rm I}$. Let $R\in\mathcal R$, and let
$m$, $\{a_{k,j}\}$, $\{r_{\balpha}\}$, and $\{\bs_{\balpha}\}$ satisfy the conditions in \eqref{eq-mainthB1-exact}. For each source $k$, the inside-ball constraint gives a transport plan from $R$ to $\sum_j a_{k,j}\delta_{\widehat\bxi_{k,j}}$ with cost at most $m\epsilon_k-\sum_{\balpha}s_{\balpha,k}$. Gluing these $K$ plans along their common first marginal $R$ and applying Lemma~\ref{lm:DMM}, there exist finite measures
$\mu_{\balpha}$ supported on $\mathcal N$ such that
\begin{align}
&\sum_{\balpha\in\mathcal A}\mu_{\balpha}=R,\label{eq-conB1converse12a}\\
&\sum_{\balpha\in\mathcal A:\alpha_k=j}\mu_{\balpha}(\mathcal N)=a_{k,j}
&& \forall j\in[N_k],~k\in[K],\label{eq-conB1converse12b}\\
&\sum_{\balpha\in\mathcal A}\int_{\Xi}
c(\bxi,\widehat{\bxi}_{k,\alpha_k})\,d\mu_{\balpha}(\bxi)
\le m\epsilon_k-\sum_{\balpha\in\mathcal A}s_{\balpha,k}
&& \forall k\in[K].\label{eq-conB1converse12c}
\end{align}
Since $\bs_{\balpha}\in r_{\balpha}\mathcal V_{\balpha}$, for every
$\balpha$ with $r_{\balpha}>0$ there exists
$P_{\balpha}\in\mathcal P(\mathcal N^c)$ such that
\begin{align}
s_{\balpha,k}
=r_{\balpha}\E_{P_{\balpha}}
\left[c(\bxi,\widehat{\bxi}_{k,\alpha_k})\right],
\quad k\in[K].\label{eq-conB1conversesk}
\end{align}
This is exactly the meaning of $\mathcal V_{\balpha}$ as the set of attainable source-wise outside-cost vectors.
Set $\zeta_{\balpha}:=r_{\balpha}P_{\balpha}$ when $r_{\balpha}>0$ and
$\zeta_{\balpha}:=0$ when $r_{\balpha}=0$.
Define finite measures
\begin{align*}
\nu_{\balpha}:=\frac{1}{m}\mu_{\balpha}
+\frac{1}{m}\zeta_{\balpha}.
\end{align*}
We verify that these measures satisfy \eqref{eq-capnufinitemeasure}-\eqref{eq-capconstraint}.
First, using \eqref{eq-conB1converse12b} and the source-balance condition
in \eqref{eq-mainthB1-exact},
\begin{align*}
\sum_{\balpha\in\mathcal A:\alpha_k=j}\nu_{\balpha}(\Xi)
=\frac{1}{m}\left(
a_{k,j}+\sum_{\balpha\in\mathcal A:\alpha_k=j}r_{\balpha}
\right)
=w_{k,j}.
\end{align*}
Second,
\begin{align*}
\sum_{\balpha\in\mathcal A}\nu_{\balpha}(\mathcal N)
=\frac{1}{m}\sum_{\balpha\in\mathcal A}\mu_{\balpha}(\mathcal N)
=\frac{1}{m}\in[a,b],
\end{align*}
and the conditional law induced by $\sum_{\balpha}\nu_{\balpha}$ is
\begin{align*}
\frac{\sum_{\balpha\in\mathcal A}\nu_{\balpha}(\mathcal N)\nu_{\balpha}^{\mathcal N}}
{\sum_{\balpha\in\mathcal A}\nu_{\balpha}(\mathcal N)}
=\sum_{\balpha\in\mathcal A}\mu_{\balpha}=R.
\end{align*}
Finally, for each $k\in[K]$, \eqref{eq-conB1converse12c} and
\eqref{eq-conB1conversesk} give
\begin{align*}
\sum_{\balpha\in\mathcal A}\int_{\Xi}
c(\bxi,\widehat{\bxi}_{k,\alpha_k})\,d\nu_{\balpha}(\bxi)
&=
\frac{1}{m}\sum_{\balpha\in\mathcal A}\int_{\Xi}
c(\bxi,\widehat{\bxi}_{k,\alpha_k})\,d\mu_{\balpha}(\bxi)
+\frac{1}{m}\sum_{\balpha\in\mathcal A}s_{\balpha,k}\le \epsilon_k.
\end{align*}
Thus \eqref{eq-capnufinitemeasure}-\eqref{eq-capconstraint} hold, and the representation \eqref{eq-reBZ} implies
$R\in\mathcal B_{\mathcal N}^{\rm I}$. This proves
$\mathcal R\subseteq\mathcal B_{\mathcal N}^{\rm I}$ and completes the proof. 
\end{proof}

\subsection{Proof of Theorem \ref{th:intersectionDual}}\label{sec:proofTHintersectionDual}

{\color{black}
To prove Theorem~\ref{th:intersectionDual}, we first consider an unconditional DRO problem over an intersection of OT balls. Suppose that $\ell:\Xi\to\R$ is measurable and upper semicontinuous. Let $\{\p_k\}_{k\in[K]}$ be discrete reference distributions of the form
$$
\p_k=\sum_{j=1}^{N_k}p_{k,j}\delta_{\widehat{\bxi}_{k,j}},
\qquad k\in[K],
$$
and let $\bepsilon:=(\epsilon_1,\dots,\epsilon_K)\in\R_+^K$ denote the vector of radii. The problem is formulated as
\begin{align}\label{eq-innersupremumI}
\mathcal L^{\rm I}(\bepsilon)
:=
\sup_{Q\in \mathcal B^{\rm I}(\{\p_k\}_{k\in[K]},\bepsilon),~Q(\mathcal N)=1}
\E_Q[\ell(\bxi)] .
\end{align}

\begin{theorem}\label{lm:boundaryI}
The function $\mathcal L^{\rm I}$ is concave on its effective domain. For every $\bepsilon$ in the relative interior of its effective domain, $\mathcal L^{\rm I}(\bepsilon)$ equals the optimal value of the following dual problem:
\begin{align}\label{eq-dualboundaryI}
\begin{array}{lll}
\min & \displaystyle \sum_{k=1}^K\lambda_k\epsilon_k+\sum_{k=1}^K \sum_{j=1}^{N_k} p_{k,j} \gamma_{k,j} & \\[1mm]
{\rm s.t.} 
& \lambda_k \ge 0,\ \bm{\gamma}_k \in \mathbb{R}^{N_k} & \forall k \in[K] \\[1mm]
& \displaystyle 
\sup_{\bxi \in \mathcal N}
\left\{
\ell(\bxi)-\sum_{k=1}^K \lambda_k c\left(\bxi, \widehat{\bxi}_{k, \alpha_k}\right)
\right\}
\le
\sum_{k=1}^K \gamma_{k, \alpha_k}
~~~~~~& \forall \balpha \in \mathcal A .
\end{array}
\end{align}
Moreover, suppose that Assumptions~\ref{assump:growthI} holds for the loss function $\ell$ given in \eqref{eq-innersupremumI}, that is, $\ell(\bxi)-\ell(\widehat{\bxi})\le L c(\bxi,\widehat{\bxi})$ for some $L>0$, and all $\bxi\in\mathcal N$ and all source observations $\widehat{\bxi}$. Suppose also that Assumption \ref{assump:costI} holds.
Then, $\mathcal L^{\rm I}$ is proper and upper semicontinuous. As a result, the above strong duality holds for all $\bepsilon$ in the effective domain of $\mathcal L^{\rm I}$, including boundary radii.
\end{theorem}

% {\color{red} The current proof for the upper semicontinuity is not completed.}

\begin{proof}[Proof of Theorem \ref{lm:boundaryI}]
Denote by $D$ the effective domain of $\mathcal L^{\rm I}$.
Let $\bepsilon,\bepsilon'\in D$, and suppose that $Q\in\mathcal B^{\rm I}(\{\p_k\}_{k\in[K]},\bepsilon)$, $Q'\in\mathcal B^{\rm I}(\{\p_k\}_{k\in[K]},\bepsilon')$ satisfy $Q(\mathcal N)=Q'(\mathcal N)=1$.
By the convexity of the mapping $P\mapsto W(P,\p_k)$ for any $k\in[K]$, it is straightforward to verify that $\lambda Q+(1-\lambda)Q'\in \mathcal B^{\rm I}(\{\p_k\}_{k\in[K]},\lambda\bepsilon+(1-\lambda)\bepsilon')$ for all $\lambda\in(0,1)$. Therefore,
\begin{align*}
\widetilde{\mathcal L}^{\rm I}\left(\bepsilon,Q\right)= \E_{Q}[\ell(\bxi)]-\delta_{\mathcal F(\bepsilon)}(Q),
\end{align*}
where 
$$\mathcal F(\bepsilon):=\mathcal B^{\rm I}\left(\{\p_k\}_{k\in[K]},\bepsilon\right)\cap \{Q: Q(\mathcal N)=1\}$$
and $\delta_{\mathcal F(\epsilon)}(Q)$ is the extended-valued indicator, equal to $0$ if $Q\in\mathcal F(\epsilon)$ and $+\infty$ otherwise, is jointly concave on $D\times \mathcal P(\Xi)$.
This yields the concavity of $\mathcal L^{\rm I}$ on $D$. By Proposition~EC.1 and Theorem~EC.1 of \cite{WCW24}, together with their proofs, the dual formulation in \eqref{eq-dualboundaryI} can be obtained from the biconjugate of $\mathcal L^{\rm I}$. Consequently, $\mathcal L^{\rm I}$ coincides with the optimal value of the dual problem for every $\bepsilon$ in the relative interior of its effective domain.
It remains to show that, under Assumptions~\ref{assump:growthI} and~\ref{assump:costI}, the function $\mathcal L^{\rm I}$ is proper and upper semicontinuous. In this case, $\mathcal L^{\rm I}$ coincides with its biconjugate on its effective domain; see, e.g., \citet[Theorem~12.2]{R70}. Therefore, the same dual formulation remains valid for all $\bepsilon$ in the effective domain, including boundary points. 

To see that $\mathcal L^{\rm I}$ is not identically $-\infty$, note that the  probability measures supported on $\mathcal N$ can be made feasible by increasing the radii. Since $\ell$ is real-valued, this yields $\mathcal L^{\rm I}(\bepsilon)>-\infty$ for some $\bepsilon$. Next we show that $\mathcal L^{\rm I}$ is not identically $\infty$. Following a similar decomposition method in the proof of Theorem 2 \cite{REMK24}, we obtain that $\mathcal L^{\rm I}(\bepsilon)$ is equivalent to 
\begin{align}\label{eq-equnconditionalDROI}
\begin{array}{lll}\text { sup } & \sum_{\balpha \in \mathcal{A}} \int_{\mathcal N} \ell(\bxi) \mathrm{d} \nu_{\balpha}(\bxi) & \\ \text { s.t. } & \nu_{\balpha} \in \mathcal{M}_{+}(\mathcal N) & \forall \alpha \in \mathcal{A} \\ & \sum_{\substack{\balpha \in \mathcal{A}, \alpha_k=j}}\nu_{\balpha}(\mathcal N) =p_{k, j} & \forall j \in\left[N_k\right], \forall k \in[K] \\ & \sum_{\balpha \in \mathcal{A}} \int_{\mathcal N} c\left(\bxi, \widehat{\bxi}_{k, \alpha_k}\right) \mathrm{d} \nu_{\balpha}(\bxi) \leq \varepsilon_k ~~~& \forall k \in[K].
\end{array}
\end{align}
Choose $\bepsilon$ satisfying $\mathcal L^{\rm I}(\bepsilon)>-\infty$.
For any feasible solution $\{\nu_{\balpha}\}_{\balpha\in\mathcal A}$ of \eqref{eq-equnconditionalDROI}, there exists $L>0$ such that
\begin{align*}
\sum_{\balpha \in \mathcal{A}} \int_{\mathcal N} \ell(\bxi) \mathrm{d} \nu_{\balpha}(\bxi)
&\le \sum_{\balpha\in\mathcal A} \int_{\mathcal N}\left(\ell(\widehat{\bxi}_{1,\alpha_1})+L c(\bxi,\widehat{\bxi}_{1,\alpha_1})\right)\d \nu_{\balpha}(\bxi)\\
&=\ell(\widehat{\bxi}_{1,\alpha_1})+L\sum_{\balpha\in\mathcal A}\int_{\mathcal N}c(\bxi,\widehat{\bxi}_{1,\alpha_1}))\d \nu_{\balpha}(\bxi)\\
&\le \ell(\widehat{\bxi}_{1,\alpha_1})+L\epsilon_1<\infty,
\end{align*}
where the first inequality follows from Assumption~\ref{assump:growthI} applied to $\ell$ and from the finiteness of $\mathcal A$. This yields $\mathcal L^{\rm I}(\bepsilon)\in\R$.  Hence, we have proved that $\mathcal L^{\rm I}$ is proper. 

We now prove the upper semicontinuity of $\mathcal L^{\rm I}$ on $D$.
Let $\bepsilon^n\to\bepsilon$ with $\bepsilon^n,\bepsilon\in D$. Denote by 
\begin{align*}
r:=\limsup_{n\to\infty} \mathcal L^{\rm I}(\bepsilon^n).
\end{align*}
% Passing
% to a subsequence if necessary, we may assume that
% $$
% \lim_{n\to\infty}\mathcal L^{\rm I}(\bepsilon^n)
% =
% \limsup_{n\to\infty}\mathcal L^{\rm I}(\bepsilon^n)
% =:r .
% $$
It suffices to prove that $r\le \mathcal L^{\rm I}(\bepsilon)$.
For each $n$, choose a feasible solution
$\{\nu^n_{\balpha}\}_{\balpha\in\mathcal A}$ of
\eqref{eq-equnconditionalDROI} with $\bepsilon$ replaced by $\bepsilon^n$
such that
$$
R_n:=
\sum_{\balpha\in\mathcal A}
\int_{\mathcal N}\ell(\bxi)\d\nu^n_{\balpha}(\bxi)
\ge
\mathcal L^{\rm I}(\bepsilon^n)-\frac1n .
$$
Hence, we have
$$
\limsup_{n\to\infty} R_n=r.
$$
The marginal constraints of \eqref{eq-equnconditionalDROI} imply
$
\sum_{\balpha\in\mathcal A}\nu^n_{\balpha}(\mathcal N)=1 .
$
It follows from $\bepsilon^n\to\bepsilon$ that the cost constraints of \eqref{eq-equnconditionalDROI} imply that,
for every $k\in[K]$,
$$
\sup_n
\sum_{\balpha\in\mathcal A}
\int_{\mathcal N}
c(\bxi,\widehat{\bxi}_{k,\alpha_k})\d\nu^n_{\balpha}(\bxi)
<\infty .
$$
By Assumption~\ref{assump:costI}, this yields tightness of
$\{\nu^n_{\balpha}\}_{n\in\mathbb N}$ for every
$\balpha\in\mathcal A$. Since $\mathcal A$ is finite, Prokhorov's theorem yields a subsequence, still denoted by $n$, and
finite measures $\{\nu_{\balpha}\}_{\balpha\in\mathcal A}\subseteq \mathcal M_+(\Xi)$ such that
$$
\nu^n_{\balpha}\Rightarrow\nu_{\balpha},
\qquad \balpha\in\mathcal A,
$$
where $\Rightarrow$ denotes the weak convergence. This implies that $\{\nu_{\balpha}\}_{\balpha\in\mathcal A}$ satisfies the marginal constraints of \eqref{eq-equnconditionalDROI}, that is, $\sum_{\substack{\balpha \in \mathcal{A}, \alpha_k=j}}\nu_{\balpha}(\mathcal N) =p_{k, j}$  for all $j \in\left[N_k\right]$ and $k \in[K]$.
Define
\begin{align*}
C_k=
\sum_{\balpha\in\mathcal A}
\int_{\mathcal N}
c(\bxi,\widehat{\bxi}_{k,\alpha_k})\d\nu_{\balpha}(\bxi)\quad{\rm and}\quad C^n_k=
\sum_{\balpha\in\mathcal A}
\int_{\mathcal N}
c(\bxi,\widehat{\bxi}_{k,\alpha_k})\d\nu^n_{\balpha}(\bxi),\quad k\in[K],n\in\N.
\end{align*}
By the Portmanteau theorem (see e.g., Theorem~2.1 of \cite{B13}) applied to the nonnegative lower semicontinuous
cost functions, for every $k\in[K]$, we have
\begin{align*}
C_k\le
\liminf_{n\to\infty}
\sum_{\balpha\in\mathcal A}
\int_{\mathcal N}
c(\bxi,\widehat{\bxi}_{k,\alpha_k})\d\nu^n_{\balpha}(\bxi)
\le \liminf_{n\to\infty}\epsilon_k^n=\epsilon_k.
\end{align*}
Hence, $\{\nu_{\balpha}\}_{\balpha\in\mathcal A}$ is feasible for
\eqref{eq-equnconditionalDROI} with radius vector $\bepsilon$. Denote its
objective value by
$$
R:=
\sum_{\balpha\in\mathcal A}
\int_{\mathcal N}\ell(\bxi)\d\nu_{\balpha}(\bxi).
$$
It is straightforward to see that $R\le \mathcal L^{\rm I}(\bepsilon)$ as $\{\nu_{\balpha}\}_{\balpha\in\mathcal A}$ is feasible.
Further,
for each $k\in[K]$ and
$\balpha\in\mathcal A$, define
$$
h_{k,\balpha}(\bxi):=
\ell(\widehat{\bxi}_{k,\alpha_k})
+
L c(\bxi,\widehat{\bxi}_{k,\alpha_k})
-
\ell(\bxi),
\qquad \bxi\in\mathcal N,
$$
where $L>0$ is the Lipschitz constant in Assumption \ref{assump:growthI}, and thus, $h_{k,\balpha}(\bxi)\ge 0$ for all $\bxi\in\mathcal N$.
Define
\begin{align*}
H_k=
\sum_{\balpha\in\mathcal A}\int_{\mathcal N}
h_{k,\balpha}(\bxi)\d\nu_{\balpha}(\bxi)\quad{\rm and}\quad 
H^n_k=\sum_{\balpha\in\mathcal A}\int_{\mathcal N}
h_{k,\balpha}(\bxi)\d\nu^n_{\balpha}(\bxi),\quad k\in[K],n\in\N.
\end{align*}
Since
$c(\cdot,\widehat{\bxi}_{k,\alpha_k})$ is lower semicontinuous and
$\ell$ is upper semicontinuous, $h_{k,\balpha}$ is lower semicontinuous.
Thus, the Portmanteau theorem (see e.g., Theorem~2.1 of \cite{B13}) gives
\begin{align}\label{eq-lowersemiH_k}
H_k\le \liminf_{n\to\infty} H_k^n\quad \forall k\in[K].
\end{align}
Using the marginal constraints of \eqref{eq-equnconditionalDROI}, for every $k\in[K]$ one can check that
$$
R_n=\sum_{\balpha\in\mathcal A}
\int_{\mathcal N}\ell(\bxi)\d\nu^n_{\balpha}(\bxi)
=
\sum_{j=1}^{N_k}p_{k,j}\ell(\widehat{\bxi}_{k,j})
+
L C^n_k-H^n_k.
$$
Similarly,
$$
R=\sum_{\balpha\in\mathcal A}
\int_{\mathcal N}\ell(\bxi)\d\nu_{\balpha}(\bxi)=
\sum_{j=1}^{N_k}p_{k,j}\ell(\widehat{\bxi}_{k,j})
+
L C_k-H_k.
$$
It follows
that
\begin{align}\label{eq-rRinequality}
r-R&=\limsup_{n\to\infty} R_n-R\notag\\
&=\limsup_{n\to\infty} \left\{\sum_{j=1}^{N_k}p_{k,j}\ell(\widehat{\bxi}_{k,j})
+ L C^n_k-H^n_k\right\}-\sum_{j=1}^{N_k}p_{k,j}\ell(\widehat{\bxi}_{k,j})
-
L C_k+H_k \notag\\
&= L(\epsilon_k-C_k)-\liminf_{n\to\infty}H_k^n+H_k\notag\\
&\le L(\epsilon_k-C_k),
\qquad k\in[K],
\end{align}
where the first inequality follows from $C^n_k\le\epsilon^n_k$ and $\epsilon^n_k\to\epsilon_k$, and the last inequality is due to \eqref{eq-lowersemiH_k}.
If $r\le R$, then $r\le R\le\mathcal L^{\rm I}(\bepsilon)$, and the desired inequality follows. Suppose
instead that $r>R$. Then \eqref{eq-rRinequality} implies
$$
\epsilon_k-C_k\ge \frac{r-R}{L}>0,
\qquad k\in[K].
$$
For sufficiently large $n$, define
\begin{align}\label{eq-deftn}
t_n=
\max_{k\in[K]}
\frac{(\epsilon^n_k-\epsilon_k)_+}
{(\epsilon^n_k-\epsilon_k)_+ + \epsilon_k-C_k},
\end{align}
where $x_+=\max\{0,x\}$ denotes the positive part of $x\in\R$.
Then, $0\le t_n\le1$ and $t_n\to0$. Set
$$
\widetilde\nu^n_{\balpha}:=
(1-t_n)\nu^n_{\balpha}+t_n\nu_{\balpha},
\qquad \balpha\in\mathcal A .
$$
The marginal constraints of \eqref{eq-equnconditionalDROI} are preserved because both $\{\nu_{\balpha}\}_{\balpha\in\mathcal A}$ and $\{\nu_{\balpha}^n\}_{\balpha\in\mathcal A}$ satisfy this constraints. 
Below we show that the cost constraints of \eqref{eq-equnconditionalDROI} with radii $\bepsilon$ also holds for $\{\widetilde\nu_{\balpha}\}_{\balpha\in\mathcal A}$. By the definition of $t_n$ in\eqref{eq-deftn}, we have
\begin{align}\label{eq-tninequality}
(\epsilon_k-C_k)t_n\ge (1-t_n)(\epsilon_k^n-\epsilon_k)_+\quad \forall k\in[K].
\end{align}
Therefore,
for every $k\in[K]$,
\begin{align*}
\sum_{\balpha\in\mathcal A}
\int_{\mathcal N}
c(\bxi,\widehat{\bxi}_{k,\alpha_k})\d\widetilde\nu^n_{\balpha}(\bxi)-\epsilon_k
&=(1-t_n)C^n_k+t_nC_k-\epsilon_k\\
&\le (1-t_n)\epsilon_k^n+t_nC_k-\epsilon_k\\
&=(1-t_n)(\epsilon_k^n-\epsilon_k)-t_n(\epsilon_k-C_k) \\
&\le (1-t_n)(\epsilon_k^n-\epsilon_k)_+-t_n(\epsilon_k-C_k)\le 0,
\end{align*}
where we have used \eqref{eq-tninequality} in the last inequality. 
This yields that $\{\widetilde\nu^n_{\balpha}\}_{\balpha\in\mathcal A}$ is feasible for
\eqref{eq-equnconditionalDROI} with radii  $\bepsilon$. Its
objective value is $(1-t_n)R_n+t_nR\le \mathcal L^{\rm I}(\bepsilon)$. On the other hand, 
\begin{align*}
\limsup_{n\to\infty}\{(1-t_n)R_n+t_nR\}=\limsup_{n\to\infty}R_n=r.
\end{align*}
Hence, $\mathcal L^{\rm I}(\bepsilon)\ge r$. 
Hence $\mathcal L^{\rm I}$ is upper semicontinuous on $D$. This competes the proof.
\end{proof}

}

~\\
\begin{proof}[Proof of Theorem \ref{th:intersectionDual}]
Fix any $\bar{\bbeta}\in\mathcal D$ and write
$\ell_{\bar{\bbeta}}(\bxi):=\ell(\bar{\bbeta},\bxi)$. To avoid notational clutter in the derivation, we write $\ell(\bxi)$ for this fixed loss and restore the minimization over $\bbeta$ at the end. Let
\begin{align*}
V^{\rm I}(\bar{\bbeta}):=\sup_{Q\in\mathcal B_{\mathcal N}^{\rm I}}\int_{\mathcal N}\ell_{\bar{\bbeta}}(\bxi)\,dQ(\bxi).
\end{align*}
By \eqref{prob-intersection-induced}, the original intersection-based conditional DRO problem is $\min_{\bbeta\in\mathcal D}V^{\rm I}(\bbeta)$.
The proof proceeds by following the two reductions developed in the main text: first use Theorem~\ref{th-analysisB1} to express the conditional set as a union of fixed-center intersection problems, then dualize the resulting finite-dimensional allocation problem.

Let $C$ denote the set of outer variables $(m,\ba,\br,\bs)$ satisfying
\begin{align*}
m\in[1/b,1/a],\quad
a_{k,j}\ge0,\quad r_{\balpha}\ge0,\quad
\sum_{j=1}^{N_k}a_{k,j}=1,
\end{align*}
and
\begin{align*}
a_{k,j}+\sum_{\balpha\in\mathcal A:\alpha_k=j}r_{\balpha}=m w_{k,j},
\quad
\bs_{\balpha}\in r_{\balpha}\mathcal V_{\balpha},
\end{align*}
for all admissible $k,j,\balpha$. For $(m,\ba,\br,\bs)\in C$, define
\begin{align*}
\widehat{\p}_k^{\ba}:=\sum_{j=1}^{N_k}a_{k,j}\delta_{\widehat\bxi_{k,j}},
\qquad
\rho_k(m,\ba,\br,\bs):=m\epsilon_k-\sum_{\balpha\in\mathcal A}s_{\balpha,k}.
\end{align*}
Theorem~\ref{th-analysisB1} turns the conditional ambiguity set into a union over the finite-dimensional allocation variables. Hence
\begin{align}
V^{\rm I}(\bar{\bbeta})
=
\sup_{(m,\ba,\br,\bs)\in C}
\sup_{\substack{Q\in\cap_{k=1}^K \mathcal B(\widehat{\p}_k^{\ba},\rho_k(m,\ba,\br,\bs))\\
Q(\mathcal N)=1}}
\E_Q[\ell(\bxi)].
\label{eq-pf-T2-fixed-outer}
\end{align}
For a fixed allocation $(m,\ba,\br,\bs)$, the inner supremum in \eqref{eq-pf-T2-fixed-outer} is a fixed-center intersection DRO problem on $\mathcal N$. By Theorem \ref{lm:boundaryI}, this inner problem is equivalent to
\begin{align}
&\inf_{\substack{\blambda\in\mathbb R_+^K\\ \gamma_{k,j}\in\mathbb R}}
\left\{
\sum_{k=1}^K\lambda_k\left(m\epsilon_k-\sum_{\balpha\in\mathcal A}s_{\balpha,k}\right)
+\sum_{k=1}^K\sum_{j=1}^{N_k}a_{k,j}\gamma_{k,j}
\right\}\notag\\
&\hspace{0.8in}
\text{\rm s.t.}\quad
\sup_{\bxi\in\mathcal N}
\left\{\ell(\bxi)-\sum_{k=1}^K\lambda_kc(\bxi,\widehat{\bxi}_{k,\alpha_k})\right\}
\le \sum_{k=1}^K\gamma_{k,\alpha_k},
\quad \balpha\in\mathcal A.
\label{eq-pf-T2-inner-dual}
\end{align}

{
Denote by $D$ the set of $(\blambda,\bgamma)$ satisfying the constraints in \eqref{eq-pf-T2-inner-dual}. The set $D$ is convex. The set $C$ is also convex because of the convexity of each perspective set $\{(r,s):r\ge0,\ s\in r\mathcal V_{\balpha}\}$. The objective is affine in each block of variables. Therefore, we can apply a standard minimax theorem if $C$ is compact. Note that $m\in[1/b,1/a]$, the simplex constraints on $\ba$ and the identity $\sum_{\balpha}r_{\balpha}=m-1$. A sufficient condition for the compactness of $C$ is that $\mathcal V_{\balpha}$ is compact for all $\balpha\in\mathcal A$.
We next show that, although $\mathcal V_{\balpha}$ is generally noncompact, the noncompact sets $\mathcal V_{\balpha}$ can be replaced by compact sets without changing the value of \eqref{eq-pf-T2-fixed-outer} under Assumption \ref{assump:SlaterI}. The construction proceeds in two steps. We first establish that replacing $\mathcal V_{\balpha}$ with its closure is equivalent, and then show that the closure can be further calibrated to be bounded.

\underline{Equivalence with the closure of $\mathcal V_{\balpha}$}: 
Denote by $\overline{\mathcal V}_{\balpha}={\rm cl}(\mathcal V_{\balpha})$.
Let $\overline{C}$ be the allocation set obtained by replacing
$
\bs_{\balpha}\in r_{\balpha}\mathcal V_{\balpha}
$
with
$
\bs_{\balpha}\in r_{\balpha} \overline{\mathcal V}_{\balpha}.
$
Let $\overline{V}^{\rm I}(\bar\bbeta)$ denote the corresponding value of \eqref{eq-pf-T2-fixed-outer} with $\overline{C}$.
Since $\mathcal V_{\balpha}\subseteq\overline{\mathcal V}_{\balpha}$, we have
$
V^{\rm I}(\bar\bbeta)
\le
\overline{V}^{\rm I}(\bar\bbeta).
$
To see the reverse inequality, take any
$(m,\ba,\br,\bs)\in \overline{C}$
and any feasible $Q$ in the inner problem of \eqref{eq-pf-T2-fixed-outer}, that is,
$$
Q\in
\bigcap_{k=1}^K
\mathcal B\left(\widehat {\p}_k^{\ba},\rho_k\right),
\qquad
Q(\mathcal N)=1,
$$
where $\rho_k:=\rho_k(m,\ba,\br,\bs)$.
% where
% $$
% \widehat P_k^{\ba}:=\sum_{j=1}^{N_k}a_{k,j}\delta_{\widehat\xi_{k,j}},
% \qquad
% \rho_k(m,\ba,\br,\bs)
% :=
% m\varepsilon_k-\sum_{\balpha\in\mathcal A}s_{\balpha,k}.
% $$
Let $p=1/m$. The identity 
$
1+\sum_{\balpha\in\mathcal A}r_{\balpha}=m,
$
implies
$
p\left(1+\sum_{\balpha\in\mathcal A}r_{\balpha}\right)=1.
$
For each $\balpha$, choose
$$
\bs_{\balpha}^n\in r_{\balpha}\mathcal V_{\balpha}
\quad\text{such that}\quad
\bs_{\balpha}^n\to \bs_{\balpha}.
$$
This is possible because
$
\bs_{\balpha}\in r_{\balpha}\overline{\mathcal V}_{\balpha}.
$
By the definition of $\mathcal V_{\balpha}$ in \eqref{eq-outside-cost-vector}, each $\bs_{\balpha}^n$ can be represented by a probability measure $\mu_{\balpha}^n$ on $\mathcal N^c$ such that
\begin{align}\label{eq-20260707-1}
r_{\balpha}
\int_{\mathcal N^c}
c(\bxi,\widehat\bxi_{k,\alpha_k})\,\d\mu_{\balpha}^n(\bxi)
=
s_{\balpha,k}^n,
\qquad k\in[K].
\end{align}
Define
\begin{align}\label{eq-20260707-3}
P_n
=
pQ+
p\sum_{\balpha\in\mathcal A}r_{\balpha}\mu_{\balpha}^n.
\end{align}
Then
$
P_n(\Xi)=p+p\sum_{\balpha\in\mathcal A} r_{\balpha}=1,
$
which implies that
$P_n$ is a probability measure. Moreover, we have
$$
P_n(\mathcal N)=p,
\qquad
(P_n)^{\mathcal N}=Q.
$$
Note that $Q\in\mathcal B(\widehat{\p}^{\ba}_k,\rho_k)$. There 
exist couplings $\pi_k^n$ between $Q$ and $\widehat{\p}_k^{\ba}$  such that
\begin{align}\label{eq-20260707-2}
\int c(\bxi_1, \bxi_2) \d\pi_k^n( \bxi_1,  \bxi_2) \leq \rho_k+\frac{1}{n}.
\end{align}
Constructing a coupling between $P_n$ and $\widehat{\p}_k$  as follows:
$$
\Gamma_k^n=p \pi_k^n+p \sum_{\balpha \in \mathcal{A}} r_{\balpha}\left(  \mu_{\balpha}^n  \otimes \delta_{\widehat{\bxi}_{k, \alpha_k}}\right).
$$
Indeed, the conclusion that $P_n$ is the first marginal is straightforward. 
The second marginal of $\Gamma_k^n$ has the same support of $\widehat{\p}_k$, where the probability mass on $\widehat{\bxi}_{k,j}$ is
\begin{align*}
p a_{k, j}+p \sum_{\balpha: \alpha_k=j} r_{\balpha}=p m w_{k, j}=w_{k, j}.
\end{align*}
This yields that $\widehat{\p}_k$ is the second marginal. Therefore,
\begin{align*}
W\left(P_n, \widehat{\P}_k\right) &\leq \int c(\bxi_1,\bxi_2) \d \Gamma_k^n(\bxi_1,\bxi_2)\\
&=p\int c(\bxi_1,\bxi_2)\d \pi_k^n(\bxi_1,\bxi_2)+p\sum_{\balpha\in\mathcal A}r_{\balpha}\int c(\bxi,\widehat{\bxi}_{k,\alpha_k})\d \mu_{\balpha}^n(\bxi)\\
&\le p \left(\rho_k+\frac{1}{n}\right)+p\sum_{\balpha\in\mathcal A}s_{\balpha,k}^n\\
&=p\left(m\epsilon_k-\sum_{\balpha\in\mathcal A}s_{\balpha,k}+\frac{1}{n}\right)+p\sum_{\balpha\in\mathcal A}s_{\balpha,k}^n\\
&=\epsilon_k+\frac{p}{n}+p\sum_{\balpha\in\mathcal A}(s_{\balpha,k}^n-s_{\balpha,k})\longrightarrow \epsilon_k\quad {\rm as}~n\to\infty,
\end{align*}
where the second inequality follows from \eqref{eq-20260707-1} and \eqref{eq-20260707-2}, we have used $\rho_k=m\epsilon_k-\sum_{\balpha\in\mathcal A}s_{\balpha,k}$ in the second equality, and the convergence is due to $\bs_{\balpha}^n\to \bs_{\balpha}$ for all $\balpha\in\mathcal A$. 
For $\lambda\in(0,1)$,
define 
\begin{align*}
\widetilde{P}_{n,\lambda}=(1-\lambda) P_n+\lambda P^{\circ},
\end{align*}
where $P^\circ$ is the distribution satisfying the conditions in Assumption \ref{assump:SlaterI}. Note that, for all $k\in[K]$, $W(P^\circ, \widehat{\p}_k)<\epsilon_k$, $W(P_n,\widehat{\P}_k)\to \epsilon_k$ as $n\to\infty$, and
\begin{align*}
W(\widetilde{P}_{n,\lambda}, \widehat{\p}_k)
\le (1-\lambda)W({P}_n, \widehat{\p}_k)+\lambda W(P^\circ, \widehat{\p}_k).
\end{align*}
For large enough $n$, it holds that $W(\widetilde{P}_{n,\lambda}, \widehat{\p}_k)\le \epsilon_k$ for all $k\in[K]$. 
Now we are able to choose a sequence of $\{\lambda_n\}_{n\in\N}\subseteq (0,1)$ such that $\lambda_n\downarrow 0$ and $\widetilde{P}_{n,\lambda_n}\in\mathcal B^{\rm I}(\{\widehat{\p}_k\}_{k\in[K]},\bepsilon)$.
Moreover, we have $\widetilde{P}_{n,\lambda_n}(\mathcal N)\in[a,b]$ as ${P}_n(\mathcal N)=p=1/m\in[a,b]$, and ${P}^{\circ}(\mathcal N)\in[a,b]$ from Assumption \ref{assump:SlaterI}. 
% This yields
% \begin{align*}
% \widetilde{P}_{n,\lambda}^{\mathcal N}\in \mathcal B_{\mathcal N}^{\rm I}(\{\widehat{\p}_k\}_{k\in[K]},\bepsilon).
% \end{align*}
In fact, $V^{\rm I}(\bar{\bbeta})$ is the optimal value of the robust optimization problem:
$$V^{\rm I}(\bar{\bbeta})=\sup_{P\in \mathcal B^{\rm I}(\{\widehat{\p}_k\}_{k\in[K]},\bepsilon), P(\mathcal N)\in[a,b]}\E_P[\ell(\bxi)\mid\bxi\in\mathcal N].$$
Because $\widetilde{P}_{n,\lambda_n}$ is included in the feasible set of the above problem,
this gives
\begin{align}\label{eq-20260707-4}
V^{\rm I}(\bar{\bbeta})\ge \E_{\widetilde{P}_{n,\lambda_n}}[\ell(\bxi)\mid\bxi\in\mathcal N],\quad \forall n.
\end{align}
Recall from \eqref{eq-20260707-3} that
$
P_n=pQ+p\sum_{\balpha\in\mathcal A}r_{\balpha}\mu_{\balpha}^n.
$
Among the measures appearing in this decomposition, $Q$ is the only probability measure that assigns positive mass to $\mathcal N$.
By some standard calculation, we have
\begin{align*}
\widetilde{P}_{n,\lambda_n}^{\mathcal N}=
\frac{\left(1-\lambda_n\right) p Q+\lambda_n P^{\circ}(\mathcal N) (P^{\circ})^{\mathcal N}}{\left(1-\lambda_n\right) p+\lambda_n P^{\circ}(\mathcal N)}.
\end{align*}
Then
\begin{align}\label{eq-20260707-5}
\E_{\widetilde{P}_{n,\lambda_n}}[\ell(\bxi)\mid\bxi\in\mathcal N] 
&= \E_{\widetilde{P}_{n,\lambda_n}^{\mathcal N}}[\ell(\bxi)]\notag\\
&=\frac{(1-\lambda_n)p}{\left(1-\lambda_n\right) p+\lambda_n P^{\circ}(\mathcal N)}\E_Q[\ell(\bxi)]+\frac{\lambda_n P_{\circ}(\mathcal N)}{\left(1-\lambda_n\right) p+\lambda_n P^{\circ}(\mathcal N)}\E_{P^\circ}[\ell(\bxi)\mid\bxi\in\mathcal N]\notag\\
&\longrightarrow \E_Q[\ell(\bxi)]\quad{\rm as}~n\to\infty,
\end{align}
where we have used $\E_{P^\circ}[\ell(\bxi)\mid\bxi\in\mathcal N]<\infty$ in the convergence step.
Because $Q$ is an arbitrary element of the feasible set of \eqref{eq-pf-T2-fixed-outer} after replacing $\mathcal V_{\balpha}$ by $\overline{\mathcal V}_{\balpha}$, combining 
\eqref{eq-20260707-4} and \eqref{eq-20260707-5} implies that
$
V^{\rm I}(\bar{\bbeta})\ge \overline{V}^{\rm I}(\bar{\bbeta}).
$
This completes the proof of this step.

\underline{Calibration to a compact set.}
We now show that the closed sets $\overline{\mathcal V}_{\balpha}$ can be further replaced by compact sets. 
Define
$$
T_{\balpha}(\bxi)
:=
\left(
c(\bxi,\widehat\bxi_{1,\alpha_1}),
\dots,
c(\bxi,\widehat\bxi_{K,\alpha_K})
\right).
$$
Choose any $\bxi_{\balpha}^0\in\mathcal N^c$ and set
$$
\bv_{\balpha}^0:=T_{\balpha}(\bxi_{\balpha}^0).
$$
Define
$$
\mathcal K_{\balpha}
:=
\left\{
\bxi\in\mathcal N^c:
c(\bxi,\widehat\bxi_{k,\alpha_k})\le v_{\balpha,k}^0
\text{ for some }k\in[K]
\right\}
=\mathcal N^c\cap \left(\bigcup_{k\in[K]}\left\{\bxi: c(\bxi,\widehat\bxi_{k,\alpha_k})\le v_{\balpha,k}^0\right\}\right)
$$
and 
\begin{align}\label{eq-20260707-6}
\mathcal U_{\balpha}
:=
\operatorname{cl}\left(\operatorname{conv}
\left\{
T_{\balpha}(\bxi):\bxi\in\mathcal K_{\balpha}
\right\}\right).
\end{align}
By the condition that $\bxi\mapsto c(\bxi,\widehat{\bxi})$ is coercive in Assumption \ref{assump:costI}, $T_{\balpha}(\mathcal K_{\balpha})$ is bounded. Hence $\mathcal U_{\balpha}$ is closed, bounded, and convex in $\mathbb R^K$, and is therefore compact. Below we adopt the notion  $\widetilde{C}$ as the allocation set obtained by replacing
$
\bs_{\balpha}\in r_{\balpha}\mathcal V_{\balpha}
$
with
$
\bs_{\balpha}\in r_{\balpha} {\mathcal U}_{\balpha},
$
and 
let $\widetilde{V}^{\rm I}(\bar\bbeta)$ denote the corresponding value of \eqref{eq-pf-T2-fixed-outer} with $\widetilde{C}$.

We claim that for every $\bv\in\overline{\mathcal V}_{\balpha}$, there exists $\bu\in\mathcal U_{\balpha}$ such that
$
\bu\le \bv
$
componentwise. Indeed, if $\bxi\in\mathcal K_{\balpha}$, then $T_{\balpha}(\bxi)\in\mathcal U_{\balpha}$. If $\bxi\notin\mathcal K_{\balpha}$, then
$
T_{\balpha}(\bxi)\ge \bv_{\balpha}^0
$
componentwise, while $\bv_{\balpha}^0\in\mathcal U_{\balpha}$. The claim follows by taking convex combinations and then closures.

Now take any $(m,\ba,\br,\bs)\in \overline{C}$. For each $\balpha$, choose $\tilde{\bs}_{\balpha}\in r_{\balpha}\mathcal U_{\balpha}$ such that
$
\tilde{\bs}_{\balpha}\le \bs_{\balpha}
$
componentwise. Then
$(m,\ba,\br,\tilde{\bs})\in \widetilde{C},$
and
$$
\rho_k(m,\ba,\br,\tilde{\bs})
=
m\varepsilon_k-\sum_{\balpha\in\mathcal A}\tilde s_{\balpha,k}
\ge
m\varepsilon_k-\sum_{\balpha\in\mathcal A}s_{\balpha,k}
=
\rho_k(m,\ba,\br,\bs).
$$
Thus the inner feasible set under $(m,\ba,\br,\tilde{\bs})$ contains the inner feasible set under $(m,\ba,\br,\bs)$. Therefore,
$$
\overline{V}^{\rm I}(\bar\bbeta)
\le
\widetilde{V}^{\rm I}(\bar\bbeta).
$$
The reverse inequality follows from
$
\mathcal U_{\balpha}\subseteq\overline{\mathcal V}_{\balpha}
$ for all $\balpha\in\mathcal A$.
Hence, we have $\widetilde{V}^{I}(\bar\bbeta)
=
\overline{V}^{\rm I}(\bar\bbeta)$. Combining with the result in the first step yields
$$
\widetilde{V}^{\rm I}(\bar\bbeta)
=
\overline{V}^{\rm I}(\bar\bbeta)
=
V^{I}(\bar\bbeta).
$$
This completes the proof of the second step.

Denote the objective function of the problem in \eqref{eq-pf-T2-inner-dual} by $J(m,\ba,\br,\bs;\blambda,\bgamma)$.
Up to now, we have proven that 
\begin{align*}
V^{\rm I}(\bar\bbeta)=\sup_{(m,\ba,\br,\bs)\in \widetilde{C}}\inf_{(\blambda,\bgamma)\in D}J(m,\ba,\br,\bs;\blambda,\bgamma),
\end{align*}
where $\widetilde{C}$ is the allocation set obtained by replacing
$
\bs_{\balpha}\in r_{\balpha}\mathcal V_{\balpha}
$
with
$
\bs_{\balpha}\in r_{\balpha} {\mathcal U}_{\balpha} 
$, and $\mathcal U_{\balpha}$ is defined in \eqref{eq-20260707-6}. The compactness of $\mathcal U_{\balpha}$ implies that its induced feasible set $\widetilde{C}$ is compact. 
Therefore Sion's minimax theorem \citep{S58} gives
\begin{align*}
V^{\rm I}(\bar\bbeta)=\inf_{(\blambda,\bgamma)\in D}\sup_{(m,\ba,\br,\bs)\in \widetilde{C}}J(m,\ba,\br,\bs;\blambda,\bgamma).
\end{align*}
By the arguments in the proof of the step ``Calibration to a compact set,'' fix $(m,\ba,\br)$, and for any $\bs_{\balpha}\in r_{\balpha}\overline{\mathcal V}_{\balpha}$, there exists $\widetilde{\bs}_{\balpha}\in r_{\balpha}\mathcal U_{\balpha}$ such that
$
\widetilde{\bs}_{\balpha}\le \bs_{\balpha}
$
componentwise. Since $J(m,\ba,\br,\bs;\blambda,\bgamma)$ is decreasing in $\bs$ in the componentwise sense, we obtain
\begin{align*}
\sup_{(m,\ba,\br,\bs)\in \widetilde{C}}J(m,\ba,\br,\bs;\blambda,\bgamma)\ge \sup_{(m,\ba,\br,\bs)\in \overline{C}}J(m,\ba,\br,\bs;\blambda,\bgamma),\quad \forall \blambda\in \R_K^+,~\gamma_{k,j}\in\R.
\end{align*}
The reverse inequality also holds as $\widetilde{C}\subseteq \overline{C}$. Therefore, 
\begin{align*}
\inf_{(\blambda,\bgamma)\in D}\sup_{(m,\ba,\br,\bs)\in \overline{C}}J(m,\ba,\br,\bs;\blambda,\bgamma)=\inf_{(\blambda,\bgamma)\in D}\sup_{(m,\ba,\br,\bs)\in \widetilde{C}}J(m,\ba,\br,\bs;\blambda,\bgamma).
\end{align*}
By the continuity of $J(m,\ba,\br,\bs;\blambda,\bgamma)$ in $\bs$, we further obtain
\begin{align*}
\inf_{(\blambda,\bgamma)\in D}\sup_{(m,\ba,\br,\bs)\in C}J(m,\ba,\br,\bs;\blambda,\bgamma)=\inf_{(\blambda,\bgamma)\in D}\sup_{(m,\ba,\br,\bs)\in \overline{C}}J(m,\ba,\br,\bs;\blambda,\bgamma).
\end{align*}
Therefore,
\begin{align}
V^{\rm I}(\bar{\bbeta})
&=\inf_{(\blambda,\bgamma)\in D}\sup_{(m,\ba,\br,\bs)\in C}J(m,\ba,\br,\bs;\blambda,\bgamma)\notag\\
&=
\inf_{(\blambda,\bgamma)\in D}
\sup_{(m,\ba,\br,\bs)\in C}
\left\{
m\sum_{k=1}^K\lambda_k\epsilon_k
-\sum_{\balpha\in\mathcal A}\blambda^\top \bs_{\balpha}
+\sum_{k=1}^K\sum_{j=1}^{N_k}a_{k,j}\gamma_{k,j}
\right\}.
\label{eq-pf-T2-minimax}
\end{align}
}

% Therefore Sion's minimax theorem \citep{S58} gives
% \begin{align}
% V^{\rm I}(\bar{\bbeta})
% =
% \inf_{(\blambda,\bgamma)\in D}
% \sup_{(m,\ba,\br,\bs)\in C}
% \left\{
% m\sum_{k=1}^K\lambda_k\epsilon_k
% -\sum_{\balpha\in\mathcal A}\blambda^\top \bs_{\balpha}
% +\sum_{k=1}^K\sum_{j=1}^{N_k}a_{k,j}\gamma_{k,j}
% \right\}.
% \label{eq-pf-T2-minimax}
% \end{align}

For fixed $(\blambda,\bgamma)\in D$, the perspective constraints $\bs_{\balpha}\in r_{\balpha}\mathcal V_{\balpha}$ can be eliminated before the final LP dualization. If $(m,\ba,\br)$ satisfies all constraints defining $C$ except $\bs_{\balpha}\in r_{\balpha}\mathcal V_{\balpha}$, then
\begin{align*}
\sup_{\{\bs_{\balpha}\in r_{\balpha}\mathcal V_{\balpha}\}_{\balpha\in\mathcal A}}
\left\{-\sum_{\balpha\in\mathcal A}\blambda^\top \bs_{\balpha}\right\}
&=
-\sum_{\balpha\in\mathcal A}r_{\balpha}
\inf_{\bgamma\in\mathcal V_{\balpha}}\blambda^\top\bgamma\\
&=
-\sum_{\balpha\in\mathcal A}r_{\balpha}d_{\balpha}^{\blambda}.
\end{align*}
Thus the remaining supremum in \eqref{eq-pf-T2-minimax} reduces to the finite-dimensional linear program
\begin{align}
\sup_{m,\ba,\br}\quad
&m\sum_{k=1}^K\lambda_k\epsilon_k
+\sum_{k=1}^K\sum_{j=1}^{N_k}a_{k,j}\gamma_{k,j}
-\sum_{\balpha\in\mathcal A}r_{\balpha}d_{\balpha}^{\blambda}\notag\\
\text{\rm s.t.}\quad
&m\in[1/b,1/a],\quad a_{k,j}\ge0,\quad r_{\balpha}\ge0,\notag\\
&\sum_{j=1}^{N_k}a_{k,j}=1,\quad k\in[K],\notag\\
&a_{k,j}+\sum_{\balpha\in\mathcal A:\alpha_k=j}r_{\balpha}=m w_{k,j},
\quad j\in[N_k],\ k\in[K].
\label{eq-pf-T2-outer-lp}
\end{align}
The feasible region of \eqref{eq-pf-T2-outer-lp} is a compact polytope. Indeed, $m$ is bounded, each $a_k$ lies in a simplex, and summing the balance equations over $j$ gives $\sum_{\balpha}r_{\balpha}=m-1$. Hence standard finite-dimensional linear programming duality applies.

To dualize \eqref{eq-pf-T2-outer-lp}, let $\varrho_k\in\R$ be the multiplier for $\sum_j a_{k,j}=1$, written as $1-\sum_j a_{k,j}=0$, and let $\varphi_{k,j}\in\R$ be the multiplier for $a_{k,j}+\sum_{\balpha:\alpha_k=j}r_{\balpha}-m w_{k,j}=0$.
The resulting Lagrangian is
\begin{align*}
&\sum_{k=1}^K\varrho_k
+m\left(\sum_{k=1}^K\lambda_k\epsilon_k-\sum_{k=1}^K\sum_{j=1}^{N_k}\varphi_{k,j}w_{k,j}\right)\\
&\quad+\sum_{k=1}^K\sum_{j=1}^{N_k}a_{k,j}
\left(\gamma_{k,j}-\varrho_k+\varphi_{k,j}\right)
+\sum_{\balpha\in\mathcal A}r_{\balpha}
\left(\sum_{k=1}^K\varphi_{k,\alpha_k}-d_{\balpha}^{\blambda}\right).
\end{align*}
Taking the supremum over $m\in[1/b,1/a]$ is equivalently represented by variables $\eta^+,\eta^-\ge0$ satisfying
\begin{align*}
\sum_{k=1}^K\lambda_k\epsilon_k-\sum_{k=1}^K\sum_{j=1}^{N_k}\varphi_{k,j}w_{k,j}
=\eta^+-\eta^-,
\end{align*}
with objective contribution $\eta^+/a-\eta^-/b$. Taking the supremum over $a_{k,j}\ge0$ and $r_{\balpha}\ge0$ gives the dual feasibility conditions
\begin{align*}
\gamma_{k,j}-\varrho_k+\varphi_{k,j}\le0,\qquad
\sum_{k=1}^K\varphi_{k,\alpha_k}\le d_{\balpha}^{\blambda}.
\end{align*}
Equivalently, introducing $\psi_{k,j}\ge0$, the dual of \eqref{eq-pf-T2-outer-lp} is
\begin{align}
\inf\quad
&\sum_{k=1}^K\varrho_k
+\eta^+/a-\eta^-/b\notag\\
\text{\rm s.t.}\quad
&\sum_{k=1}^K\lambda_k\epsilon_k-\sum_{k=1}^K\sum_{j=1}^{N_k}\varphi_{k,j}w_{k,j}=\eta^+-\eta^-,\notag\\
&\gamma_{k,j}-\varrho_k+\varphi_{k,j}+\psi_{k,j}=0,
\quad j\in[N_k],\ k\in[K],\notag\\
&\sum_{k=1}^K\varphi_{k,\alpha_k}\le d_{\balpha}^{\blambda},
\quad \balpha\in\mathcal A,\notag\\
&\varrho_k\in\mathbb R,\quad \varphi_{k,j}\in\mathbb R,\quad
\eta^+,\eta^-,\psi_{k,j}\ge0.
\label{eq-pf-T2-outer-lp-dual}
\end{align}

Combining \eqref{eq-pf-T2-minimax} with \eqref{eq-pf-T2-outer-lp-dual}, and then eliminating $\bgamma,\bm{\varrho},\bm{\psi}$, yields the compact form in the theorem for the fixed decision $\bar{\bbeta}$. To see the elimination explicitly, let
\begin{align*}
\kappa:=\sum_{k=1}^K\varrho_k.
\end{align*}
From $\gamma_{k,j}-\varrho_k+\varphi_{k,j}+\psi_{k,j}=0$ and $\psi_{k,j}\ge0$, we have
$\gamma_{k,j}\le \varrho_k-\varphi_{k,j}$. Therefore the inner-dual constraint in \eqref{eq-pf-T2-inner-dual} implies
\begin{align*}
\sup_{\bxi\in\mathcal N}
\left\{\ell(\bxi)-\sum_{k=1}^K\lambda_kc(\bxi,\widehat\bxi_{k,\alpha_k})\right\}
+\sum_{k=1}^K\varphi_{k,\alpha_k}
\le \kappa,\qquad \balpha\in\mathcal A.
\end{align*}
Conversely, if this inequality holds for some $\kappa$, choose any $\varrho_k$ with $\sum_k\varrho_k=\kappa$, set $\psi_{k,j}=0$, and set $\gamma_{k,j}=\varrho_k-\varphi_{k,j}$; then the constraints involving $\bgamma,\bm{\varrho},\bm{\psi}$ are satisfied. Thus the eliminated formulation is exactly \eqref{prob-intersection-general} with $\bbeta=\bar{\bbeta}$. Since $\bar{\bbeta}\in\mathcal D$ was arbitrary, minimizing this representation over $\bbeta\in\mathcal D$ gives \eqref{prob-intersection-general}.

\end{proof}

\subsection{Proofs of Proposition \ref{prop:scalarized-outside} and Corollary \ref{cor:two-source-intersection}}

\begin{proof}[Proof of Proposition \ref{prop:scalarized-outside}]
Fix $\balpha\in\mathcal A$ and $\blambda\in\mathbb R_+^K$. Since $\mathcal N^c\subseteq\Xi$, we have
\begin{align*}
\inf_{\bxi\in\Xi}\sum_{k=1}^K\lambda_kc(\bxi,\widehat\bxi_{k,\alpha_k})
\le
d_{\balpha}^{\blambda}.
\end{align*}
For the reverse inequality, fix any $\bar\bxi\in\Xi$. Because $\operatorname{cl}(\mathcal N^c)=\Xi$, there exists a sequence $\bxi_n\in\mathcal N^c$ with $\bxi_n\to\bar\bxi$. Continuity of $c(\cdot,\widehat\bxi_{k,\alpha_k})$ gives
\begin{align*}
\sum_{k=1}^K\lambda_kc(\bxi_n,\widehat\bxi_{k,\alpha_k})
\to
\sum_{k=1}^K\lambda_kc(\bar\bxi,\widehat\bxi_{k,\alpha_k}).
\end{align*}
Hence $d_{\balpha}^{\blambda}\le \sum_{k=1}^K\lambda_kc(\bar\bxi,\widehat\bxi_{k,\alpha_k})$. Taking the infimum over $\bar\bxi\in\Xi$ gives \eqref{eq-convex-dalpha}. The problem in \eqref{eq-convex-dalpha} is convex because $\Xi$ is convex and each function $c(\cdot,\widehat\bxi_{k,\alpha_k})$ is convex with a nonnegative multiplier $\lambda_k$.

Let
\begin{align*}
f_k(\bxi):=\lambda_kc(\bxi,\widehat\bxi_{k,\alpha_k}),\quad k\in[K].
\end{align*}
Using \eqref{eq-convex-dalpha}, the scalarization constraint \eqref{eq-intersection-outside-cut} is
\begin{align*}
\sum_{k=1}^K\varphi_{k,\alpha_k}
\le
\inf_{\bxi\in\Xi}\sum_{k=1}^K f_k(\bxi).
\end{align*}
Equivalently, writing the constraint $\bxi\in\Xi$ through the indicator function $\delta_\Xi$, the right-hand side is
\begin{align*}
\inf_{\bxi}\left\{\sum_{k=1}^K f_k(\bxi)+\delta_\Xi(\bxi)\right\}.
\end{align*}
Since each $f_k$ is finite and continuous convex and $\Xi$ is nonempty, the Fenchel-Rockafellar conditions hold for $\sum_k f_k+\delta_\Xi$. Hence strong duality applies and the infimum above equals the dual value
\begin{align*}
\sup_{\substack{z_{\balpha,k},v_{\balpha}\\ \sum_k z_{\balpha,k}+v_{\balpha}=0}}
-\sum_{k=1}^K f_k^*(z_{\balpha,k})-\sigma_{\Xi}(v_{\balpha}).
\end{align*}
Therefore, \eqref{eq-intersection-outside-cut} is equivalent to the existence of $z_{\balpha,k}$ and $v_{\balpha}$ such that
\begin{align*}
\sum_{k=1}^K\varphi_{k,\alpha_k}
+\sum_{k=1}^K f_k^*(z_{\balpha,k})
+\sigma_{\Xi}(v_{\balpha})\le0,\qquad
\sum_{k=1}^Kz_{\balpha,k}+v_{\balpha}=0,
\end{align*}
which is \eqref{eq-fenchel-outside}--\eqref{eq-fenchel-balance}.

If $c(\bxi,\widehat\bxi)=\|\bxi-\widehat\bxi\|$, then
\begin{align*}
\left(\lambda_k\|\cdot-\widehat\bxi_{k,\alpha_k}\|\right)^*(z)
=z^\top\widehat\bxi_{k,\alpha_k}+\delta_{\{\|z\|_*\le\lambda_k\}}(z),
\end{align*}
which yields \eqref{eq-norm-outside}--\eqref{eq-norm-outside-balance}. 
\end{proof}

~\\
\begin{proof}[Proof of Corollary \ref{cor:two-source-intersection}]
Let $\bx_1:=\widehat\bxi_{1,\alpha_1}$, $\bx_2:=\widehat\bxi_{2,\alpha_2}$, and $d_{\balpha}:=\|\bx_1-\bx_2\|$. For any $\bxi\in\mathcal N^c$,
\begin{align*}
\lambda_1\|\bxi-\bx_1\|+\lambda_2\|\bxi-\bx_2\|
\ge
\min\{\lambda_1,\lambda_2\}\left(\|\bxi-\bx_1\|+\|\bxi-\bx_2\|\right)
\ge
\min\{\lambda_1,\lambda_2\}\|\bx_1-\bx_2\|,
\end{align*}
where the second inequality is the triangle inequality. Hence $d_{\balpha}^{\blambda}\ge \min\{\lambda_1,\lambda_2\}d_{\balpha}$. If $\lambda_1\le\lambda_2$, then $\bx_2\in\Xi=\operatorname{cl}(\mathcal N^c)$, so there exists $\bxi_n\in\mathcal N^c$ with $\bxi_n\to \bx_2$. By continuity of the norm,
\begin{align*}
d_{\balpha}^{\blambda}
\le
\lim_{n\to\infty}
\left(\lambda_1\|\bxi_n-\bx_1\|+\lambda_2\|\bxi_n-\bx_2\|\right)
=\lambda_1d_{\balpha}.
\end{align*}
The case $\lambda_2\le\lambda_1$ follows symmetrically by approximating $\bx_1$. This proves \eqref{eq-two-source-dlambda}.

It remains to derive the two-source robust counterpart. The fixed-loss counterpart underlying Theorem \ref{th:intersectionDual} with $K=2$ contains the scalarization constraint
\begin{align*}
\sum_{k=1}^2\varphi_{k,\alpha_k}\le d_{\balpha}^{\blambda}.
\end{align*}
The identity \eqref{eq-two-source-dlambda} implies
\begin{align*}
d_{\balpha}^{\blambda}=\min\{\lambda_1,\lambda_2\}d_{\balpha}.
\end{align*}
The inequality above is equivalently represented by variables $\zeta_1,\zeta_2,\eta_{\balpha}\ge 0$ satisfying
\begin{align*}
-\lambda_1+\lambda_2+\zeta_1-\zeta_2=0,\qquad
(-\lambda_1+\zeta_1)d_{\balpha}+\sum_{k=1}^2\varphi_{k,\alpha_k}+\eta_{\balpha}=0.
\end{align*}
Indeed, the first equality means that $\theta:=\lambda_1-\zeta_1=\lambda_2-\zeta_2$ for some $\theta\le\min\{\lambda_1,\lambda_2\}$. The second equality and $\eta_{\balpha}\ge0$ impose $\sum_k\varphi_{k,\alpha_k}\le \theta d_{\balpha}$; since $\theta$ can be chosen as $\min\{\lambda_1,\lambda_2\}$, this is equivalent to $\sum_k\varphi_{k,\alpha_k}\le\min\{\lambda_1,\lambda_2\}d_{\balpha}$.

The scalar constraint in \eqref{prob-intersection-general} is equivalently written with $\tau_1,\tau_2\ge0$ as
\begin{align*}
\sum_{k=1}^2 \lambda_k\epsilon_k-\sum_{k=1}^2\sum_{j=1}^{N_k}\varphi_{k,j}w_{k,j}+\tau_1-\tau_2=0,
\end{align*}
with objective contribution $-\tau_1/b+\tau_2/a$. Finally, the inside constraints in \eqref{prob-intersection-general} can be written in the separated form used in \eqref{prob-capK=2}. Specifically,
\begin{align*}
\sup_{\bxi\in\mathcal N}\left\{\ell(\bbeta,\bxi)-\sum_{k=1}^2\lambda_k\|\bxi-\widehat\bxi_{k,\alpha_k}\|\right\}
+\sum_{k=1}^2\varphi_{k,\alpha_k}\le \kappa,\quad \balpha\in\mathcal A,
\end{align*}
is equivalent to the existence of $\gamma_{k,j}$, $\varrho_k$, and $\psi_{k,j}\ge0$ such that
\begin{align*}
\sup_{\bxi\in\mathcal N}\left\{\ell(\bbeta,\bxi)-\sum_{k=1}^2\lambda_k\|\bxi-\widehat\bxi_{k,\alpha_k}\|\right\}
\le \sum_{k=1}^2\gamma_{k,\alpha_k},\qquad
\gamma_{k,j}-\varrho_k+\varphi_{k,j}+\psi_{k,j}=0,
\end{align*}
with $\kappa=\varrho_1+\varrho_2$. One direction follows by summing the inequalities $\gamma_{k,j}+\varphi_{k,j}\le\varrho_k$. For the reverse direction, given any feasible $\kappa$, set $\varrho_1=\kappa$, $\varrho_2=0$, $\gamma_{1,j}=\kappa-\varphi_{1,j}$, $\gamma_{2,j}=-\varphi_{2,j}$, and $\psi_{k,j}=0$. Combining these equivalent representations gives \eqref{prob-capK=2}. 
\end{proof}

\section{Proofs of results in Section \ref{sec:tractabilityB2}}

\subsection{Proofs of Theorem \ref{th-analysisB2} and Lemma \ref{lm-BII-eqclosure}}

\begin{proof}[Proof of Theorem \ref{th-analysisB2}]
To simplify the notation, we let $\mathcal B^{\rm II}:=\mathcal B^{\rm{II}}(\{\widehat{\p}_k\}_{k\in[K]},\epsilon)$, $\mathcal B^{\rm II}_{\mathcal N}:=\mathcal B^{\rm{II}}_{\mathcal N}(\{\widehat{\p}_k\}_{k\in[K]},\epsilon)$, and denote by $\mathcal R$ the set defined in \eqref{eq-mainthB2}.
Recall the relation that $\mathcal B_{\mathcal N}^{\rm II}=\{Q^{\mathcal N}: Q\in\mathcal B^{\rm II}, Q(\mathcal N)\in[a,b]\}$.
Similar to the proof of Theorem \ref{th-analysisB1},
one can verify that 
\begin{align}\label{eq-sumreBZ}
\mathcal B_{\mathcal N}^{\rm II}=\left\{\frac{\sum_{\balpha\in\mathcal A}\nu_{\balpha}(\mathcal N)\nu_{\balpha}^{\mathcal N}}{\sum_{\balpha\in\mathcal A}\nu_{\balpha}(\mathcal N)}:~\text{\eqref{eq-sumnufinitemeasure}-\eqref{eq-sumconstraint} hold}
\right\},
\end{align}
where terms with $\nu_{\balpha}(\mathcal N)=0$ are immaterial and the conditions are
\begin{align}
%\begin{array}{lll}
&\nu_{\balpha}\in \mathcal M_+(\Xi) &&\forall \balpha\in\mathcal A\label{eq-sumnufinitemeasure}\\
&\sum_{\balpha\in\mathcal A,\alpha_k=j}\nu_{\balpha}(\Xi)=w_{k,j} &&\forall j\in[N_k],~\forall k\in [K]\label{eq-sumsumnu}\\
&\sum_{\balpha\in\mathcal A} \nu_{\balpha}(\mathcal N)\in [a,b]\label{eq-sumconstraintprob}\\
&\sum_{k=1}^K \sum_{\balpha\in\mathcal A}\int_{\Xi}\theta_kc(\bxi,\widehat{\bxi}_{k,\alpha_k})\d \nu_{\balpha}(\bxi)\le \epsilon.\label{eq-sumconstraint}
%\end{array}
\end{align}
Below we aim to utilize the representation in \eqref{eq-sumreBZ} to prove that $\mathcal B^{\rm II}_{\mathcal N}=\mathcal R$.

We first show $\mathcal B_{\mathcal N}^{\rm II}\subseteq \mathcal R$. For $R\in\mathcal B_{\mathcal N}^{\rm II}$, assume that $\{\nu_{\balpha}\}_{\balpha\in\mathcal A}$ satisfy \eqref{eq-sumnufinitemeasure}-\eqref{eq-sumconstraint}, and $R$ can be represented by them as in \eqref{eq-sumreBZ}. Similar to the proof of Theorem \ref{th-analysisB1}, we define
\begin{align*}
\begin{array}{ll}
p_{\balpha}:=\nu_{\balpha}(\mathcal N),\quad
q_{\balpha}:=\nu_{\balpha}(\mathcal N^c),\quad
t_{\balpha}:=mp_{\balpha},\quad
r_{\balpha}:=mq_{\balpha},\quad
\mu_{\balpha}:=\nu_{\balpha}^{\mathcal N},
& \balpha\in\mathcal A,\\[2mm]
m:=\displaystyle\frac{1}{\sum_{\balpha\in\mathcal A}p_{\balpha}},\quad
a_{k,j}:=\displaystyle\sum_{\balpha\in\mathcal A:\alpha_k=j}t_{\balpha},
& j\in[N_k],~k\in[K].
\end{array}
\end{align*}
where $\mu_{\balpha}$ may be chosen arbitrarily on $\mathcal N$ when
$p_{\balpha}=0$, since it is then multiplied by $t_{\balpha}=0$.
Then $R=\sum_{\balpha\in\mathcal A}t_{\balpha}\mu_{\balpha}$,
$R(\mathcal N)=1$, $m\in[1/b,1/a]$, and
\begin{align*}
\sum_{j=1}^{N_k}a_{k,j}=1,\quad
a_{k,j}+\sum_{\balpha\in\mathcal A:\alpha_k=j}r_{\balpha}=mw_{k,j},
\quad j\in[N_k],~k\in[K],
\end{align*}
For every $\balpha$ with $r_{\balpha}>0$, define
\begin{align*}
v_{\balpha,k}:=
\int_{\Xi}c(\bxi,\widehat{\bxi}_{k,\alpha_k})\,d\nu_{\balpha}^{\mathcal N^c}(\bxi),
\quad k\in[K].
\end{align*}
Then $\bv_{\balpha}:=(v_{\balpha,1},\ldots,v_{\balpha,K})\in\mathcal V_{\balpha}$ by the definition of $\mathcal V_{\balpha}$. Set
$\bs_{\balpha}:=r_{\balpha}\bv_{\balpha}$ when $r_{\balpha}>0$ and
$\bs_{\balpha}:=0$ when $r_{\balpha}=0$. Hence
$\bs_{\balpha}\in r_{\balpha}\mathcal V_{\balpha}$ for all $\balpha$.
To show $R\in \mathcal R$,
it remains to verify that 
\begin{align}\label{eq-1mainthB2}
\sum_{k=1}^K \theta_k W\left(R,\sum_{j=1}^{N_k}a_{k,j}\delta_{\widehat{\bxi}_{k,j}}\right)\le m\epsilon -\sum_{\balpha\in\mathcal A}\btheta^{\top}\bs_{\balpha}.
\end{align}
Indeed, applying \eqref{eq-sumconstraint} yields
\begin{align*}
\epsilon
&\ge \sum_{k=1}^K \sum_{\balpha\in\mathcal A}\int_{\Xi}\theta_kc(\bxi,\widehat{\bxi}_{k,\alpha_k})\d \nu_{\balpha}(\bxi)\notag\\
% &=\sum_{k=1}^K \sum_{\balpha\in\mathcal A}\int_{\Xi}\theta_kc(\bxi,\widehat{\bxi}_{k,\alpha_k})\d \left(\nu_{\balpha}(\mathcal N)\nu_{\balpha}^{\mathcal N}(\bxi)+\nu_{\balpha}(\mathcal N^c)\nu_{\balpha}^{\mathcal N^c}(\bxi)\right)\notag\\
&=\sum_{k=1}^K \sum_{\balpha\in\mathcal A}\int_{\Xi}\theta_kc(\bxi,\widehat{\bxi}_{k,\alpha_k})\d \left(\frac{1}{m}t_{\balpha}\mu_{\balpha}(\bxi)+\frac{1}{m}r_{\balpha}\nu_{\balpha}^{\mathcal N^c}(\bxi)\right)\notag\\
&=\frac{1}{m}\sum_{k=1}^K \sum_{\balpha\in\mathcal A}t_{\balpha}\int_{\Xi}\theta_kc(\bxi,\widehat{\bxi}_{k,\alpha_k})\d \mu_{\balpha}(\bxi)+ \frac{1}{m}\sum_{\balpha\in\mathcal A} \btheta^{\top} \bs_{\balpha}\notag\\
% &\ge \frac{1}{m}\sum_{k=1}^K \sum_{\balpha\in\mathcal A}t_{\balpha}\int_{\Xi}\theta_kc(\bxi,\widehat{\bxi}_{k,\alpha_k})\d \mu_{\balpha}(\bxi)+\frac{1}{m}\sum_{\balpha\in\mathcal A}r_{\balpha}d_{\balpha}^{\btheta}\\%\label{eq-chIIfirstdirection}\\
&=\frac{1}{m}\sum_{k=1}^K \sum_{j=1}^{N_k}\sum_{\balpha\in\mathcal A,\alpha_k=j}t_{\balpha}\int_{\Xi}\theta_kc(\bxi,\widehat{\bxi}_{k,j})\d \mu_{\balpha}(\bxi)+\frac{1}{m}\sum_{\balpha\in\mathcal A} \btheta^{\top} \bs_{\balpha}\notag\\
&=\frac{1}{m}\sum_{k=1}^K \sum_{j=1}^{N_k}a_{k,j}\int_{\Xi}\theta_kc(\bxi,\widehat{\bxi}_{k,j})\d\left(\frac{\sum_{\balpha\in\mathcal A,\alpha_k=j}t_{\balpha}\mu_{\balpha}(\bxi)}{a_{k,j}}\right)+\frac{1}{m}\sum_{\balpha\in\mathcal A} \btheta^{\top} \bs_{\balpha}\notag\\
&\ge \frac{1}{m}\sum_{k=1}^K\theta_k{W\left(R,\sum_{j=1}^{N_k}a_{k,j}\delta_{\widehat{\bxi}_{k,j}}\right)}+\frac{1}{m}\sum_{\balpha\in\mathcal A} \btheta^{\top} \bs_{\balpha}, \notag
\end{align*}
where the last inequality holds because of the definition of the OT cost $W$ and 
\begin{align*}
\sum_{j=1}^{N_k}a_{k,j}\left(\frac{\sum_{\balpha\in\mathcal A,\alpha_k=j}t_{\balpha}\mu_{\balpha}(\bxi)}{a_{k,j}}\right)
=\sum_{\balpha\in\mathcal A}t_{\balpha}\mu_{\balpha}=R.
\end{align*}
This yields \eqref{eq-1mainthB2}, and we have concluded that $R\in\mathcal R$, which gives $\mathcal B_{\mathcal N}^{\rm II}\subseteq \mathcal R$.

We next prove the reverse inclusion $\mathcal R\subseteq\mathcal B_{\mathcal N}^{\rm II}$. Let $R\in\mathcal R$, and let
$m$, $\{a_{k,j}\}$, $\{r_{\balpha}\}$, and $\{\bs_{\balpha}\}$ satisfy the conditions in \eqref{eq-mainthB2}. Choose transport plans from $R$ to $\{\sum_{j=1}^{N_k} a_{k,j}\delta_{\widehat{\bxi}_{k,j}}\}_{k\in[K]}$ whose weighted total cost is smaller than $m\epsilon-\sum_{\balpha\in\mathcal A} \btheta^{\top}\bs_{\balpha}$. Since these plans share the same first marginal $R$, the gluing lemma and Lemma \ref{lm:DMM} give finite measures $\mu_{\balpha}$ on $\mathcal N$ such that
\begin{align}
&\sum_{\balpha\in\mathcal A}\mu_{\balpha}=R,\notag\\
&\sum_{\balpha\in\mathcal A:\alpha_k=j}\mu_{\balpha}(\mathcal N)=a_{k,j}
&& \forall j\in[N_k],~k\in[K],\notag\\
&\sum_{k=1}^{K}\sum_{\balpha\in\mathcal A}\int_{\Xi}\theta_k
c(\bxi,\widehat{\bxi}_{k,\alpha_k})\,d\mu_{\balpha}(\bxi)
\le m\epsilon-\sum_{\balpha\in\mathcal A}\btheta^{\top} \bs_{\balpha}.\label{eq-conB2converse12c}
\end{align}

Since $s_{\balpha}\in r_{\balpha}\mathcal V_{\balpha}$, for every
$\balpha$ with $r_{\balpha}>0$ there exists
$P_{\balpha}\in\mathcal P(\mathcal N^c)$ such that
\begin{align}
s_{\balpha,k}
=r_{\balpha}\E_{P_{\balpha}}
\left[c(\bxi,\widehat{\bxi}_{k,\alpha_k})\right],
\quad k\in[K].\label{eq-conB2conversesk}
\end{align}
Set $\zeta_{\balpha}:=r_{\balpha}P_{\balpha}$ when $r_{\balpha}>0$ and
$\zeta_{\balpha}:=0$ when $r_{\balpha}=0$.
Define finite measures
\begin{align*}
\nu_{\balpha}:=\frac{1}{m}\mu_{\balpha}
+\frac{1}{m}\zeta_{\balpha}.
\end{align*}
We verify that these measures satisfy \eqref{eq-sumnufinitemeasure}-\eqref{eq-sumconstraint}.
Note that the construction of $\bs_{\balpha}$, $P_{\balpha}$, $\zeta_{\balpha}$ and $\nu_{\balpha}$ are the same to those in the proof of Theorem \ref{th-analysisB1}.
The following results that has been given in the proof of Theorem \ref{th-analysisB1} can be directly applied:
\begin{align*}
\sum_{\balpha\in\mathcal A:\alpha_k=j}\nu_{\balpha}(\Xi)
=\frac{1}{m}\left(
a_{k,j}+\sum_{\balpha\in\mathcal A:\alpha_k=j}r_{\balpha}
\right)
=w_{k,j},
\end{align*}
\begin{align*}
\sum_{\balpha\in\mathcal A}\nu_{\balpha}(\mathcal N)
=\frac{1}{m}\sum_{\balpha\in\mathcal A}\mu_{\balpha}(\mathcal N)
=\frac{1}{m}\in[a,b],
\end{align*}
and the conditional law induced by $\sum_{\balpha}\nu_{\balpha}$ is
\begin{align*}
\frac{\sum_{\balpha\in\mathcal A}\nu_{\balpha}(\mathcal N)\nu_{\balpha}^{\mathcal N}}
{\sum_{\balpha\in\mathcal A}\nu_{\balpha}(\mathcal N)}
=\sum_{\balpha\in\mathcal A}\mu_{\balpha}=R.
\end{align*}
Finally, \eqref{eq-conB2converse12c} and
\eqref{eq-conB2conversesk} give
\begin{align*}
\sum_{k=1}^K \sum_{\balpha\in\mathcal A}\int_{\Xi}\theta_kc(\bxi,\widehat{\bxi}_{k,\alpha_k})\d \nu_{\balpha}(\bxi)=\frac{1}{m} \sum_{k=1}^K \sum_{\balpha\in\mathcal A}\int_{\Xi}\theta_kc(\bxi,\widehat{\bxi}_{k,\alpha_k})\d \mu_{\balpha}(\bxi)+\frac{1}{m} \sum_{\balpha\in\mathcal A} \btheta^{\top} \bs_{\balpha}\le \epsilon.
\end{align*}
Thus \eqref{eq-sumnufinitemeasure}-\eqref{eq-sumconstraint} hold, and the representation \eqref{eq-sumreBZ} implies
$R\in\mathcal B_{\mathcal N}^{\rm II}$. This proves
$\mathcal R\subseteq\mathcal B_{\mathcal N}^{\rm II}$ and completes the proof. 
\end{proof}

~\\
\begin{proof}[Proof of Lemma \ref{lm-BII-eqclosure}]
Let $C$ be the feasible set of allocation variables $(m,\ba,\br,\bs)$ in \eqref{eq-mainthB2}. For each $\balpha\in\mathcal A$, set
$
\overline{\mathcal V}_{\balpha}:=\operatorname{cl}(\mathcal V_{\balpha}),
$
and let $\overline C$ be the allocation set obtained from $C$ by replacing $\bs_{\balpha}\in r_{\balpha}\mathcal V_{\balpha}$ with $\bs_{\balpha}\in r_{\balpha}\overline{\mathcal V}_{\balpha}$. We denote by $\widetilde{\mathcal B}_{\mathcal N}^{\rm II}$ the corresponding induced conditional set, i.e., the set in \eqref{eq-mainthB2} with each $\mathcal V_{\balpha}$ replaced by $\overline{\mathcal V}_{\balpha}$. It is straightforward to see that the infimum of \eqref{eq-BII-infradius} is attained when $\mathcal V_{\balpha}$ is replaced by $\overline{\mathcal V}_{\balpha}$. Hence, we have $\widetilde{\mathcal B}_{\mathcal N}^{\rm II}=\overline{\mathcal B}_{\mathcal N}^{\rm II}$. Below we will show that 
\begin{align}\label{eq-eqclosure}
\sup_{Q\in \mathcal B^{\rm II}_{\mathcal N}
}\E_{Q}[\ell(\bbeta,\bxi)]=\sup_{Q\in \widetilde{\mathcal B}^{\rm II}_{\mathcal N}
 }\E_{Q}[\ell(\bbeta,\bxi)]\quad \forall \bbeta\in\mathcal D.
 \end{align}

The proof of \eqref{eq-eqclosure} and the associated constructions are analogous to those in the step ``Equivalence with the closure of $\mathcal V_{\balpha}$'' in the proof of Theorem \ref{th-analysisB1}. For completeness, we provide the details here.
For $\bar{\bbeta}\in\mathcal D$, we write $\ell(\cdot):=\ell(\bar{\bbeta},\cdot)$, and
let ${V}^{\rm II}(\bar\bbeta)$ and $\widetilde{V}^{\rm II}(\bar\bbeta)$ denote
the values of the left and right-hand sides of \eqref{eq-eqclosure}, respectively. 
For $(m,\ba,\br,\bs)\in \overline{C}$, define
\begin{align*}
\widehat{\p}_k^{\ba}=\sum_{j=1}^{N_k}a_{k,j}\delta_{\widehat\bxi_{k,j}},
\qquad
\rho(m,\ba,\br,\bs)=m\epsilon-\sum_{\balpha\in\mathcal A}\btheta^{\top} \bs_{\balpha}.
\end{align*}
Hence,
\begin{align}
\widetilde{V}^{\rm II}(\bar{\bbeta})
=
\sup_{(m,\ba,\br,\bs)\in \overline{C}}
\sup_{\substack{Q\in \mathcal B^{\rm II}\left(\{\widehat{\p}_k^{\ba}\}_{k\in[K]},\rho(m,\ba,\br,\bs)\right)\\
Q(\mathcal N)=1}}
\E_Q[\ell(\bxi)].
\label{eq-pf-T2II-fixed-outer}
\end{align}
Since $\mathcal V_{\balpha}\subseteq\overline{\mathcal V}_{\balpha}$, we have
$
V^{\rm II}(\bar\bbeta)
\le
\widetilde{V}^{\rm II}(\bar\bbeta).
$
To see the reverse inequality, take any
$(m,\ba,\br,\bs)\in \overline{C}$
and any feasible $Q$ in the inner problem of \eqref{eq-pf-T2II-fixed-outer}, that is,
$$
\sum_{k=1}^K \theta_k W\left(Q,\widehat{\p}_k^{\ba}\right)\le \rho,
\qquad
Q(\mathcal N)=1,
$$
where $\rho:=\rho(m,\ba,\br,\bs)$.
% where
% $$
% \widehat P_k^{\ba}:=\sum_{j=1}^{N_k}a_{k,j}\delta_{\widehat\xi_{k,j}},
% \qquad
% \rho_k(m,\ba,\br,\bs)
% :=
% m\varepsilon_k-\sum_{\balpha\in\mathcal A}s_{\balpha,k}.
% $$
Let $p=1/m$. The identity 
$
1+\sum_{\balpha\in\mathcal A}r_{\balpha}=m,
$
implies
$
p\left(1+\sum_{\balpha\in\mathcal A}r_{\balpha}\right)=1.
$
For each $\balpha$, choose
$$
\bs_{\balpha}^n\in r_{\balpha}\mathcal V_{\balpha}
\quad\text{such that}\quad
\bs_{\balpha}^n\to \bs_{\balpha}.
$$
This is possible because
$
\bs_{\balpha}\in r_{\balpha}\overline{\mathcal V}_{\balpha},
$
and $\overline{\mathcal V}_{\balpha}$ is the closure of $\mathcal V_{\balpha}$
By the definition of $\mathcal V_{\balpha}$ in \eqref{eq-outside-cost-vector}, each $\bs_{\balpha}^n$ can be represented by a probability measure $\mu_{\balpha}^n$ on $\mathcal N^c$ such that
\begin{align}\label{eq-20260710-1}
r_{\balpha}
\int_{\mathcal N^c}
c(\bxi,\widehat\bxi_{k,\alpha_k})\,\mu_{\balpha}^n(\d\bxi)
=
s_{\balpha,k}^n,
\qquad k\in[K].
\end{align}
Define
\begin{align}\label{eq-20260710-3}
P_n
=
pQ+
p\sum_{\balpha\in\mathcal A}r_{\balpha}\mu_{\balpha}^n.
\end{align}
Then
$
P_n(\Xi)=p+p\sum_{\balpha\in\mathcal A} r_{\balpha}=1,
$
which implies that
$P_n$ is a probability measure. Moreover, we have
$$
P_n(\mathcal N)=p,
\qquad
(P_n)^{\mathcal N}=Q.
$$
Note that $\sum_{k=1}^K \theta_k W(Q,\widehat{\p}_k^{\ba})\le \rho$. There 
exist couplings $\pi_k^n$ between $Q$ and $\widehat{\p}_k^{\ba}$  such that
\begin{align}\label{eq-20260710-2}
\sum_{k=1}^K\theta_k\int c(\bxi_1, \bxi_2) \d \pi_k^n( \bxi_1, \bxi_2) \leq \rho+\frac{1}{n}.
\end{align}
% {\color{black}
% By gluing lemma, there exists $\widetilde{\pi}^n\in\Pi(Q,\widehat{\p}_1^{\ba},\dots,\widehat{\p}_K^{\ba})$ whose
% $(\bxi,\bxi_k)$-marginal is $\pi_k^n$ for every $k$, and thus
% \begin{align*}
% \sum_{k=1}^K\theta_k\int c(\bxi_0, \bxi_k) \d \widetilde{\pi}^n(\bxi_0, \bxi_1,\dots,\bxi_K) \leq \rho+\frac{1}{n}.
% \end{align*}
% }
Constructing a coupling between $P_n$ and $\widehat{\p}_k$  as follows:
$$
\Gamma_k^n=p \pi_k^n+p \sum_{\balpha \in \mathcal{A}} r_{\balpha}\left(  \mu_{\balpha}^n  \otimes \delta_{\widehat{\bxi}_{k, \alpha_k}}\right).
$$
Indeed, the conclusion that $P_n$ is the first marginal is straightforward. 
The second marginal of $\Gamma_k^n$ has the same support of $\widehat{\p}_k$, where the probability mass on $\widehat{\bxi}_{k,j}$ is
\begin{align*}
p a_{k, j}+p \sum_{\balpha: \alpha_k=j} r_{\balpha}=p m w_{k, j}=w_{k, j}.
\end{align*}
This yields that $\widehat{\p}_k$ is the second marginal. 
% By gluing lemma, there exists $\widetilde{\Gamma}^n\in\Pi(P_n,\widehat{\p}_1,\dots,\widehat{\p}_K)$ whose
% $(\bxi,\bxi_k)$-marginal is $\Gamma_k^n$ for every $k\in[K]$.
Therefore,
\begin{align}
\sum_{k=1}^K \theta_k W\left(P_n, \widehat{\P}_k\right) &\le \sum_{k=1}^K\theta_k\int c(\bxi_1,\bxi_2) \d {\Gamma}_k^n(\bxi_1,\bxi_2)\notag\\
&=\sum_{k=1}^K \theta_k \left(p\int c(\bxi_1,\bxi_2)\d \pi_k^n(\bxi_1,\bxi_2)+p\sum_{\balpha\in\mathcal A}r_{\balpha}\int c(\bxi,\widehat{\bxi}_{k,\alpha_k})\d \mu_{\balpha}^n(\bxi)\right)\notag\\
&\le p \left(\rho+\frac{1}{n}\right)+p\sum_{\balpha\in\mathcal A}\btheta^{\top}\bs_{\balpha}^n\notag\\
&=p\left(m\epsilon-\sum_{\balpha\in\mathcal A}\btheta^{\top}\bs_{\balpha}+\frac{1}{n}\right)+p\sum_{\balpha\in\mathcal A}\btheta^{\top}\bs_{\balpha}^n\notag\\
&=\epsilon+\frac{p}{n}+p\sum_{\balpha\in\mathcal A}\btheta^{\top}(\bs_{\balpha}^n-\bs_{\balpha})\longrightarrow \epsilon\quad {\rm as}~n\to\infty,\label{eq-20260710-6}
\end{align}
where the second inequality follows from \eqref{eq-20260710-1} and \eqref{eq-20260710-2}, we have used $\rho=m\epsilon-\sum_{\balpha\in\mathcal A}\btheta^{\top}\bs_{\balpha}$ in the second equality, and the convergence is due to $\bs_{\balpha}^n\to \bs_{\balpha}$. 
For $\lambda\in(0,1)$,
define 
\begin{align*}
\widetilde{P}_{n,\lambda}=(1-\lambda) P_n+\lambda P^{\circ},
\end{align*}
where $P^\circ$ is the probability measure satisfying the conditions in Assumption \ref{assump:SlaterII}. Note that
\begin{align*}
\sum_{k=1}^K \theta_k W(P^\circ, \widehat{\p}_k)<\epsilon\quad{\rm and}\quad W(\widetilde{P}_{n,\lambda}, \widehat{\p}_k)
\le (1-\lambda)W({P}_n, \widehat{\p}_k)+\lambda W(P^\circ, \widehat{\p}_k).
\end{align*}
Combining with \eqref{eq-20260710-6},  it holds that, for large enough $n$,
\begin{align*}
\sum_{k=1}^K \theta_k W(\widetilde{P}_{n,\lambda},\widehat{\p}_k)<\epsilon.
\end{align*}
Now we are able to choose a sequence of $\{\lambda_n\}_{n\in\N}\subseteq (0,1)$ such that $\lambda_n\downarrow 0$ and $\widetilde{P}_{n,\lambda_n}\in\mathcal B^{\rm II}(\{\widehat{\p}_k\}_{k\in[K]},\epsilon)$ for all $n$.
Moreover, we have $\widetilde{P}_{n,\lambda_n}(\mathcal N)\in[a,b]$ as ${P}_n(\mathcal N)=p=1/m\in[a,b]$, and ${P}^{\circ}(\mathcal N)\in[a,b]$ from Assumption \ref{assump:SlaterII}. 
% This yields
% \begin{align*}
% \widetilde{P}_{n,\lambda}^{\mathcal N}\in \mathcal B_{\mathcal N}^{\rm I}(\{\widehat{\p}_k\}_{k\in[K]},\bepsilon).
% \end{align*}
In fact, $V^{\rm II}(\bar{\bbeta})$ is the optimal value of the robust optimization problem:
$$V^{\rm II}(\bar{\bbeta})=\sup_{P\in \mathcal B^{\rm II}(\{\widehat{\p}_k\}_{k\in[K]},\epsilon), P(\mathcal N)\in[a,b]}\E_P[\ell(\bxi)\mid\bxi\in\mathcal N].$$
Because $\widetilde{P}_{n,\lambda_n}$ is included in the feasible set of the above problem,
this gives
\begin{align}\label{eq-20260710-4}
V^{\rm II}(\bar{\bbeta})\ge \E_{\widetilde{P}_{n,\lambda_n}}[\ell(\bxi)\mid\bxi\in\mathcal N],\quad \forall n.
\end{align}
Recall from \eqref{eq-20260710-3} that
$
P_n=pQ+p\sum_{\balpha\in\mathcal A}r_{\balpha}\mu_{\balpha}^n.
$
Among the measures appearing in this decomposition, $Q$ is the only probability measure that assigns positive mass to $\mathcal N$.
By some standard calculation, we have
\begin{align*}
\widetilde{P}_{n,\lambda_n}^{\mathcal N}=
\frac{\left(1-\lambda_n\right) p Q+\lambda_n P^{\circ}(\mathcal N) (P^{\circ})^{\mathcal N}}{\left(1-\lambda_n\right) p+\lambda_n P^{\circ}(\mathcal N)}.
\end{align*}
Then
\begin{align}\label{eq-20260710-5}
\E_{\widetilde{P}_{n,\lambda_n}}[\ell(\bxi)\mid\bxi\in\mathcal N] 
&= \E_{\widetilde{P}_{n,\lambda_n}^{\mathcal N}}[\ell(\bxi)]\notag\\
&=\frac{(1-\lambda_n)p}{\left(1-\lambda_n\right) p+\lambda_n P^{\circ}(\mathcal N)}\E_Q[\ell(\bxi)]+\frac{\lambda_n P^{\circ}(\mathcal N)}{\left(1-\lambda_n\right) p+\lambda_n P^{\circ}(\mathcal N)}\E_{P^\circ}[\ell(\bxi)\mid\bxi\in\mathcal N]\notag\\
&\longrightarrow \E_Q[\ell(\bxi)]\quad{\rm as}~n\to\infty,
\end{align}
where we have used $\E_{P^\circ}[\ell(\bxi)\mid\bxi\in\mathcal N]<\infty$ in the convergence step.
Because $Q$ is an arbitrary element of the feasible set of \eqref{eq-pf-T2II-fixed-outer}, combining 
\eqref{eq-20260710-4} and \eqref{eq-20260710-5} implies that
$
V^{\rm II}(\bar{\bbeta})\ge \widetilde{V}^{\rm II}(\bar{\bbeta}).
$
This completes the proof.
\end{proof}

\subsection{Proof of Theorem \ref{th:sumB2Duality}}\label{app:pfTHsum}

Let $\{\p_k\}_{k\in[K]}$ be discrete distributions with the same form as those in Section \ref{sec:proofTHintersectionDual}, and let $\ell:\Xi\to\R$ be a measurable and upper semicontinuous loss function.
To prove Theorem \ref{th:sumB2Duality}, we need to first consider the unconditional DRO problem as follows:
\begin{align}\label{prob:unconB2}
\mathcal L^{\rm II}(\epsilon):=\sup_{Q\in \mathcal B^{\rm{II}}\left(\{\p_k\}_{k\in[K]}, \epsilon\right), Q(\mathcal N)=1} \E_{Q}[\ell(\bxi)].
\end{align}

{\color{black}
\begin{theorem}\label{th-unconB2}
The function $\mathcal L^{\rm II}$ is concave on its effective domain. For every $\epsilon$ in the relative interior of its effective domain, $\mathcal L^{\rm II}(\epsilon)$ equals the optimal value of the following dual problem:
\begin{align}\label{prob0:thunconditionB2}
\begin{array}{lll}
\min & \lambda\epsilon+\sum_{k=1}^K \sum_{j=1}^{N_k} p_{k, j} \gamma_{k, j} & \\
{\rm s.t. } & \lambda \ge 0, \bm{\gamma}_k \in \mathbb{R}^{N_k} & \forall k \in[K] \\
& \sup _{\bxi \in \mathcal N} \left\{\ell(\bxi)-\lambda\sum_{k=1}^K \theta_k c\left(\bxi, \widehat{\bxi}_{k, \alpha_k}\right)\right\} \le \sum_{k=1}^K \gamma_{k, \alpha_k} ~~~& \forall \balpha \in \mathcal{A}.
\end{array}
\end{align}
Moreover, suppose that Assumptions~\ref{assump:growthI} holds for the loss function $\ell$ given in \eqref{prob:unconB2}, that is, $\ell(\bxi)-\ell(\widehat{\bxi})\le L c(\bxi,\widehat{\bxi})$ for some $L>0$, and all $\bxi\in\mathcal N$ and all source observations $\widehat{\bxi}$. Suppose also that Assumption \ref{assump:costI} holds.
Then, $\mathcal L^{\rm II}$ is proper and upper semicontinuous. As a result, the above strong duality holds for all $\epsilon$ in the effective domain of $\mathcal L^{\rm II}$, including boundary radii.
\end{theorem}

\begin{proof}
By the convexity of the mapping $Q\mapsto W(Q,P)$, the proof of the concavity of $\mathcal L^{\rm II}$ is analogous to the corresponding argument in the proof of Lemma~\ref{lm:boundaryI}. We therefore omit the details.

Next, we apply a similar proof of Theorem 1 in \cite{ZYG24} to verify the dual formulation 
 \eqref{prob0:thunconditionB2}.
Define
\begin{align*}
\epsilon^*=\min_{Q: Q(\mathcal N)=1} \sum_{k=1}^K \theta_k W(Q,{\p}_k).
\end{align*}
It is straightforward to see that 
the weighted-distance set $\mathcal B^{\rm II}(\{\p_k\}_{k\in[K]},\epsilon)\cap \{Q:Q(\mathcal N)=1\}$ is empty if $\epsilon<\epsilon^*$.
Consider the Legendre transform of $\mathcal L^{\rm II}$:
\begin{align*}
\left(\mathcal L^{\rm II}\right)^*(\lambda)&:=\sup _{\varepsilon \geq \epsilon^*}\left\{\mathcal L^{\rm II}(\varepsilon)-{\lambda}{\varepsilon}\right\} \\
& =\sup _{\varepsilon \geq \epsilon^*} \sup_{Q\in\mathcal P(\Xi),Q(\mathcal N)=1}\left\{\mathbb{E}_{Q}[\ell(\bxi)]-\lambda{\varepsilon}: \sum_{k=1}^K\theta_kW\left(Q, \p_k\right) \leq \varepsilon\right\} \\
& =\sup_{Q\in\mathcal P(\Xi), Q(\mathcal N)=1} \sup _{\varepsilon \geq\epsilon^*}\left\{\mathbb{E}_{Q}[\ell(\bxi)]-\lambda{\varepsilon}: \sum_{k=1}^K\theta_kW\left({Q}, \p_k\right) \leq \varepsilon\right\} \\
& =\sup_{Q\in\mathcal P(\Xi), Q(\mathcal N)=1}\left\{\mathbb{E}_{Q}[\ell(\bxi)]-\lambda\sum_{k=1}^K \theta_k W\left(Q, \p_k\right)\right\} \\
& =\sup_{Q\in\mathcal P(\Xi), Q(\mathcal N)=1}\left\{\mathbb{E}_{Q}[\ell(\bxi)]-\lambda\sum_{k=1}^K \theta_k \inf _{\pi_k \in \Pi(Q, \p_k)} \mathbb{E}_{\left(\bxi,\widehat{\bxi}_k\right) \sim \pi_k}\left[c\left(\bxi,\widehat{\bxi}_k\right)\right]\right\}, \\
& =\sup_{Q\in\mathcal P(\Xi), Q(\mathcal N)=1, \pi_k\in\Pi(Q,\p_k)~\forall k\in[K]} \mathbb{E}_{\bxi \sim Q,\left(\bxi,\widehat{\bxi}_k\right) \sim \pi_k, \forall k \in[K]}\left[\ell(\bxi)-\lambda\sum_{k \in[K]} \theta_k c\left(\bxi,\widehat{\bxi}_k\right)\right]\\
&=\sup_{Q\in\mathcal P(\mathcal N), \pi_k\in\Pi(Q,\p_k)~\forall k\in[K]} \mathbb{E}_{\bxi \sim Q,\left(\bxi,\widehat{\bxi}_k\right) \sim \pi_k, \forall k \in[K]}\left[\ell(\bxi)-\lambda\sum_{k \in[K]} \theta_k c\left(\bxi,\widehat{\bxi}_k\right)\right]
\end{align*}
where we have used the definition of optimal transport cost of $W$ in the fifth equality.  
Therefore,
\begin{align}
\left(\mathcal L^{\rm II}\right)^*(\lambda)&=\sup_{Q\in\mathcal P(\mathcal N), \pi_k\in\Pi(Q,\p_k)~\forall k\in[K]} \mathbb{E}_{\bxi \sim Q,\left(\bxi,\widehat{\bxi}_k\right) \sim \pi_k, \forall k \in[K]}\left[\ell(\bxi)-\lambda\sum_{k \in[K]} \theta_k c\left(\bxi,\widehat{\bxi}_k\right)\right]\notag\\
& \ge\sup_{\left(\widehat{\bxi}_1, \dots, \widehat{\bxi}_K\right) \sim P\in\Pi(\p_1,\dots,\p_K), \pi \in \Pi_{P}, \pi\in\mathcal P(\mathcal N\times \Xi^K)} \mathbb{E}_{\left(\bxi,\widehat{\bxi}_1, \dots, \widehat{\bxi}_K\right) \sim \pi}\left[\ell(\bxi)-\lambda\sum_{k=1}^{K} \theta_k c\left(\bxi,\widehat{\bxi}_k\right)\right]\label{eq1-unconB2}\\
&=\sup_{\left(\widehat{\bxi}_1, \dots, \widehat{\bxi}_K\right) \sim P\in\Pi(\p_1,\dots,\p_K)} \sup _{\pi\in \Pi_{P},\ \pi\in\mathcal P(\mathcal N\times \Xi^K)} \mathbb{E}_{\left(\bxi,\widehat{\bxi}_1, \dots, \widehat{\bxi}_K\right) \sim \pi}\left[\ell(\bxi)-\lambda\sum_{k=1}^{K} \theta_k c\left(\bxi,\widehat{\bxi}_k\right)\right]\notag\\
&=\sup_{\pi \in\Pi(\p_1,\dots,\p_K)}  \mathbb{E}_{\left(\widehat{\bxi}_1, \dots, \widehat{\bxi}_K\right) \sim \pi}\left[\sup_{\bxi\in\mathcal N}\left\{\ell(\bxi)-\lambda\sum_{k=1}^{K} \theta_k c\left(\bxi,\widehat{\bxi}_k\right)\right\}\right],
%,\label{eq1-unconB2}
\end{align}
where the last equality follows from the discreteness of the reference distributions $\{\p_k\}_{k\in[K]}$, which allows the interchangeability principle to be applied; see \citet[Proposition~1 and Example~1]{ZYG24}.
On the other hand, for any $Q\in\mathcal P(\mathcal N)$ and $\pi_k\in\Pi(Q,\p_k)$ for $k\in[K]$,  an iterative application of the Gluing Lemma \citep[Chapter 1]{V08} shows that there exists a multi-margin transportation plan ${\pi} \in \Pi\left(Q, {\mathbb{P}}_1, \dots, {\mathbb{P}}_K\right)$ with $P_{0, k \#} {\pi}=\pi_k$ for all $k \in[K]$, where $P_{0, k}$ is defined as usual through $P_{0, k}\left(\bxi, \bxi_1, \ldots, \bxi_K\right)=\left(\bxi, \bxi_k\right)$. Therefore, the converse direction of \eqref{eq1-unconB2} also holds.
% we have
% \begin{align}\label{eq2-unconB2}
% \left(\mathcal L^{\rm II}\right)^*(\lambda)&=\sup_{Q\in\mathcal P(\mathcal N), \pi_k\in\Pi(Q,\p_k)~\forall k\in[K]} \mathbb{E}_{\bxi \sim Q,\left(\bxi,\widehat{\bxi}_k\right) \sim \pi_k, \forall k \in[K]}\left[\ell(\bxi)-\lambda\sum_{k \in[K]} \theta_k c\left(\bxi,\widehat{\bxi}_k\right)\right]\notag\\
% & \le\sup_{\left(\widehat{\bxi}_1, \dots, \widehat{\bxi}_K\right) \sim P\in\Pi(\p_1,\dots,\p_K)} \sup _{\pi\sim \Pi_{P}} \mathbb{E}_{\left(\bZ,\widehat{\bZ}_1, \dots, \widehat{\bZ}_K\right) \sim \pi}\left[\ell(\bZ)-\lambda\sum_{k=1}^{K} \theta_k c_{\mathcal Z}\left(\bZ,\widehat{\bZ}_k\right)\right].
% \end{align}
This yields
\begin{align}\label{eq-LII*}
\left(\mathcal L^{\rm II}\right)^*(\lambda)=\sup_{\pi \in\Pi(\p_1,\dots,\p_K)}  \mathbb{E}_{\left(\widehat{\bxi}_1, \dots, \widehat{\bxi}_K\right) \sim \pi}\left[\sup_{\bxi\in\mathcal N}\left\{\ell(\bxi)-\lambda\sum_{k=1}^{K} \theta_k c\left(\bxi,\widehat{\bxi}_k\right)\right\}\right].
\end{align}
% It then follows from the interchangeability principle in Definition \ref{def:IP} that 
% \begin{align*}
% \left(\mathcal L^{\rm II}\right)^*(\lambda)=\sup_{\pi\in\Pi(\p_1^{\mathcal Z},\dots,\p_K^{\mathcal Z})} \mathbb{E}_{\left(\widehat{\bZ}_1, \dots, \widehat{\bZ}_K\right) \sim \pi}\left[\sup_{\bz\in\mathcal Z}\left\{\ell(\bz)-\lambda\sum_{k=1}^{K} \theta_k c_{\mathcal Z}\left(\bz,\widehat{\bZ}_k\right)\right\}\right].
% \end{align*}
Note that $\left(\mathcal L^{\rm II}\right)^*(\lambda)=\infty$ if $\lambda<0$.
The Legendre transform of $\left(\mathcal L^{\rm II}\right)^*$ has the form: 
\begin{align}\label{eq3-unconB2}
\left(\mathcal L^{\rm II}\right)^{**}(\epsilon)&=\inf_{\lambda\ge 0}\left\{\lambda\epsilon+\left(\mathcal L^{\rm II}\right)^{*}(\lambda)\right\}\notag\\
&=\inf_{\lambda\ge 0}\left\{\lambda\epsilon+\sup_{\pi\in\Pi(\p_1,\dots,\p_K)} \mathbb{E}_{\left(\widehat{\bxi}_1, \dots, \widehat{\bxi}_K\right) \sim \pi}\left[\sup_{\bxi\in\mathcal N}\left\{\ell(\bxi)-\lambda\sum_{k=1}^{K} \theta_k c\left(\bxi,\widehat{\bxi}_k\right)\right\}\right]\right\}.
\end{align}
Note that $\mathcal L^{\rm II}(\cdot)$ is concave, and $(\epsilon^*,\infty)$ is the interior of the effective domain. The concave biconjugacy theorem \citep[Theorem 12.2]{R70} gives, for any $\epsilon> \epsilon^*$
\begin{align}\label{prob:thunconditionB2}
\mathcal L^{\rm II}(\epsilon)
=
\left(\mathcal L^{\rm II}\right)^{**}(\epsilon)
&=
\inf_{\lambda\ge0}\left\{\lambda\epsilon+\left(\mathcal L^{\rm II}\right)^*(\lambda)\right\}\notag\\
&=\min_{\lambda \geq 0}\left\{\lambda\epsilon+\sup_{\pi\in\Pi({\p}_1,\dots,{\p}_K)}\E_{(\widehat{\bxi}_1,\dots,\widehat{\bxi}_K)\sim \pi}\left[\sup _{\bxi \in \mathcal{N}}\left\{\ell(\bxi)-\lambda\sum_{k=1}^K \theta_k c\left( \bxi,\widehat{\bxi}_k\right)\right\}\right]\right\}.
\end{align}
Since $\{\p_k\}_{k\in [K]}$ are discrete, any joint distribution in $\Pi(\p_1,\dots,\p_K)$ has the form
\begin{align*}
\pi=\sum_{\balpha\in\mathcal A}q_{\balpha} \delta_{\left(\widehat{\bxi}_{1,\alpha_1},\dots,\widehat{\bxi}_{K,\alpha_K}\right)};\quad q_{\balpha}\ge 0\quad\forall \balpha\in\mathcal A;\quad \sum_{\balpha\in\mathcal A, \alpha_k=j}q_{\balpha}=p_{k,j}\quad\forall j\in[N_k],k\in[K].
\end{align*}
For fixed $\lambda\ge0$, define
\begin{align*}
g_{\balpha}(\lambda):=
\sup _{\bxi \in \mathcal N}
\left\{\ell(\bxi)-\lambda\sum_{k=1}^K \theta_k c\left(\bxi, \widehat{\bxi}_{k, \alpha_k}\right)\right\},
\quad \balpha\in\mathcal A.
\end{align*}
The supremum over $\pi\in\Pi(\p_1,\dots,\p_K)$ in \eqref{prob:thunconditionB2} is therefore the optimal value of the finite linear program
\begin{align*}
\begin{array}{ll}
\max_{\bq} & \displaystyle\sum_{\balpha\in\mathcal A}q_{\balpha}g_{\balpha}(\lambda)\\
{\rm s.t.} & q_{\balpha}\ge0 \quad \forall \balpha\in\mathcal A,\\
&\displaystyle\sum_{\balpha\in\mathcal A:\alpha_k=j}q_{\balpha}=p_{k,j}
\quad \forall j\in[N_k],~k\in[K].
\end{array}
\end{align*}
By finite-dimensional linear programming duality, this maximum is equal to
\begin{align*}
\begin{array}{ll}
\min_{\bm\gamma_1,\ldots,\bm\gamma_K} & \displaystyle\sum_{k=1}^K\sum_{j=1}^{N_k}p_{k,j}\gamma_{k,j}\\
{\rm s.t.} & g_{\balpha}(\lambda)\le \displaystyle\sum_{k=1}^K\gamma_{k,\alpha_k}
\quad \forall \balpha\in\mathcal A,\\
& \bm\gamma_k\in\mathbb R^{N_k}\quad \forall k\in[K].
\end{array}
\end{align*}
Substituting this dual representation into \eqref{prob:thunconditionB2} yields exactly \eqref{prob0:thunconditionB2}. 

We now turn to the ``Moreover'' part.
The properness of $\mathcal L^{\rm II}$ can be directly derived from Proposition \ref{prop:weightset} and
the properness of $\mathcal L^{\rm I}$ shown in Theorem \ref{lm:boundaryI}.
The proof of upper semicontinuity and the associated constructions are analogous to those used in the proof of Theorem~\ref{lm:boundaryI}. For completeness, we provide a streamlined argument, omitting details that have already been established there.

Following a similar decomposition method, we obtain that $\mathcal L^{\rm II}(\epsilon)$ is equivalent to 
\begin{align}\label{eq-equnconditionalDROII}
\begin{array}{lll}\text { sup } & \sum_{\balpha \in \mathcal{A}} \int_{\mathcal N} \ell(\bxi) \mathrm{d} \nu_{\balpha}(\bxi) & \\ \text { s.t. } & \nu_{\balpha} \in \mathcal{M}_{+}(\mathcal N) & \forall \alpha \in \mathcal{A} \\ & \sum_{\substack{\balpha \in \mathcal{A}, \alpha_k=j}}\nu_{\balpha}(\mathcal N) =p_{k, j} & \forall j \in\left[N_k\right], \forall k \in[K] \\ & \sum_{k=1}^K\sum_{\balpha \in \mathcal{A}} \int_{\mathcal N} \theta_kc\left(\bxi, \widehat{\bxi}_{k, \alpha_k}\right) \mathrm{d} \nu_{\balpha}(\bxi) \leq \varepsilon.
\end{array}
\end{align}
Let $\epsilon_n\to\epsilon$ with $\epsilon_n,\epsilon\in [\epsilon^*,\infty)$. Denote by 
\begin{align*}
r:=\limsup_{n\to\infty} \mathcal L^{\rm II}(\epsilon_n).
\end{align*}
% Passing
% to a subsequence if necessary, we may assume that
% $$
% \lim_{n\to\infty}\mathcal L^{\rm I}(\bepsilon^n)
% =
% \limsup_{n\to\infty}\mathcal L^{\rm I}(\bepsilon^n)
% =:r .
% $$
We aim to prove that $r\le \mathcal L^{\rm II}(\epsilon)$.
For each $n$, choose a feasible 
$\{\nu^n_{\balpha}\}_{\balpha\in\mathcal A}$ of
\eqref{eq-equnconditionalDROII} with $\epsilon$ replaced by $\epsilon_n$
such that
$$
R_n:=
\sum_{\balpha\in\mathcal A}
\int_{\mathcal N}\ell(\bxi)\d\nu^n_{\balpha}(\bxi)
\ge
\mathcal L^{\rm II}(\epsilon_n)-\frac1n .
$$
It holds that
$$
\limsup_{n\to\infty} R_n=r.
$$
Moreover, we have
$$
\sup_n
\sum_{k=1}^K\sum_{\balpha\in\mathcal A}
\int_{\mathcal N}\theta_k
c(\bxi,\widehat{\bxi}_{k,\alpha_k})\d\nu^n_{\balpha}(\bxi)
<\infty .
$$
Assumption~\ref{assump:costI} yields the tightness of
$\{\nu^n_{\balpha}\}_{n\in\mathbb N}$ for every
$\balpha\in\mathcal A$. Prokhorov's theorem yields a subsequence, still denoted by $n$, and
finite measures $\{\nu_{\balpha}\}_{\balpha\in\mathcal A}\subseteq \mathcal M_+(\Xi)$ such that
$$
\nu^n_{\balpha}\Rightarrow\nu_{\balpha},
\qquad \balpha\in\mathcal A.
$$
This implies that $\{\nu_{\balpha}\}_{\balpha\in\mathcal A}$ satisfies the marginal constraints of \eqref{eq-equnconditionalDROII}.
Define
\begin{align*}
C=
\sum_{k=1}^K\sum_{\balpha\in\mathcal A}
\int_{\mathcal N}\theta_k
c(\bxi,\widehat{\bxi}_{k,\alpha_k})\d\nu_{\balpha}(\bxi)\quad{\rm and}\quad C_n=
\sum_{k=1}^K\sum_{\balpha\in\mathcal A}
\int_{\mathcal N}\theta_k
c(\bxi,\widehat{\bxi}_{k,\alpha_k})\d\nu^n_{\balpha}(\bxi),\quad n\in\N.
\end{align*}
By the Portmanteau theorem, we have
\begin{align*}
C\le
\liminf_{n\to\infty}
\sum_{k=1}^K\sum_{\balpha\in\mathcal A}\int_{\mathcal N}\theta_k
c(\bxi,\widehat{\bxi}_{k,\alpha_k})\d\nu^n_{\balpha}(\bxi)
= 
\liminf_{n\to\infty}C_n
\le\liminf_{n\to\infty}\epsilon_n=\epsilon.
\end{align*}
Hence, $\{\nu_{\balpha}\}_{\balpha\in\mathcal A}$ is feasible for
\eqref{eq-equnconditionalDROII} with radius $\epsilon$. Denote its
objective value by
$$
R:=
\sum_{\balpha\in\mathcal A}
\int_{\mathcal N}\ell(\bxi)\d\nu_{\balpha}(\bxi).
$$
Then $R\le \mathcal L^{\rm II}(\epsilon)$.
Further, we assume without loss of generality that $\min_{k\in[K]}\theta_k>0$, and 
for each $k\in[K]$ and
$\balpha\in\mathcal A$, define
$$
h_{k,\balpha}(\bxi):=
\ell(\widehat{\bxi}_{k,\alpha_k})
+
\widetilde{L}\theta_k c(\bxi,\widehat{\bxi}_{k,\alpha_k})
-
\ell(\bxi),
\qquad \bxi\in\mathcal N,
$$
where 
\begin{align*}
\widetilde{L}=\frac{L}{\min_{k\in[K]}\theta_k},
\end{align*}
and
$L>0$ is the Lipschitz constant in Assumption \ref{assump:growthI}.  Hence, $h_{k,\balpha}(\bxi)\ge 0$ for all $\bxi\in\mathcal N$.
Define
\begin{align*}
H=
\sum_{k=1}^K\sum_{\balpha\in\mathcal A}\int_{\mathcal N}
h_{k,\balpha}(\bxi)\d\nu_{\balpha}(\bxi)\quad{\rm and}\quad 
H_n=\sum_{k=1}^K\sum_{\balpha\in\mathcal A}\int_{\mathcal N}
h_{k,\balpha}(\bxi)\d\nu^n_{\balpha}(\bxi),\quad n\in\N.
\end{align*}
The Portmanteau theorem gives
\begin{align*}%\label{eq-lowersemiH_kII}
H\le \liminf_{n\to\infty} H_n.
\end{align*}
One can check that
\begin{align*}
&K R_n=\sum_{\balpha\in\mathcal A}
\int_{\mathcal N}\ell(\bxi)\d\nu^n_{\balpha}(\bxi)
=
\sum_{k=1}^{K}\sum_{j=1}^{N_k}p_{k,j}\ell(\widehat{\bxi}_{k,j})
+
\widetilde{L} C_n-H_n;\\
&K R=\sum_{\balpha\in\mathcal A}
\int_{\mathcal N}\ell(\bxi)\d\nu_{\balpha}(\bxi)=\sum_{k=1}^K
\sum_{j=1}^{N_k}p_{k,j}\ell(\widehat{\bxi}_{k,j})
+
\widetilde{L} C-H.
\end{align*}
It follows
that
\begin{align}\label{eq-rRinequalityII}
K(r-R)&=\limsup_{n\to\infty} (KR_n-KR)\notag\\
&=\limsup_{n\to\infty} \left\{\sum_{k=1}^K\sum_{j=1}^{N_k}p_{k,j}\ell(\widehat{\bxi}_{k,j})
+ \widetilde{L} C_n-H_n\right\}-\sum_{k=1}^K\sum_{j=1}^{N_k}p_{k,j}\ell(\widehat{\bxi}_{k,j})
-
\widetilde{L} C+H \notag\\
&= \widetilde{L}(\epsilon-C)-\liminf_{n\to\infty}H_n+H\notag\\
&\le L(\epsilon-C),
\qquad k\in[K].
\end{align}
If $r\le R$, then $r\le R\le\mathcal L^{\rm II}(\bepsilon)$, and the desired inequality follows. Suppose
instead that $r>R$. Then \eqref{eq-rRinequalityII} implies
$$
\epsilon-C\ge \frac{K(r-R)}{L}>0,
\qquad k\in[K].
$$
For sufficiently large $n$, define
\begin{align*}%\label{eq-deftn}
t_n=
\max_{k\in[K]}
\frac{(\epsilon_n-\epsilon)_+}
{(\epsilon_n-\epsilon)_+ + \epsilon-C}.
\end{align*}
Then, $0\le t_n\le1$ and $t_n\to0$. Set
$$
\widetilde\nu^n_{\balpha}:=
(1-t_n)\nu^n_{\balpha}+t_n\nu_{\balpha},
\qquad \balpha\in\mathcal A .
$$
The marginal constraints of \eqref{eq-equnconditionalDROII} are preserved.
% Below we show that the cost constraints of \eqref{eq-equnconditionalDROII} with radius $\epsilon$ also holds for $\{\widetilde\nu_{\balpha}\}_{\balpha\in\mathcal A}$. 
% By the definition of $t_n$ in\eqref{eq-deftn}, we have
% \begin{align}\label{eq-tninequality}
% (\epsilon_k-C_k)t_n\ge (1-t_n)(\epsilon_k^n-\epsilon_k)_+\quad \forall k\in[K].
% \end{align}
Moreover, one can check that
\begin{align*}
\sum_{k=1}^K\sum_{\balpha\in\mathcal A}
\int_{\mathcal N}\theta_k
c(\bxi,\widehat{\bxi}_{k,\alpha_k})\d\widetilde\nu^n_{\balpha}(\bxi)-\epsilon
&=(1-t_n)C_n+t_nC-\epsilon\le 0.
\end{align*}
Hence, $\{\widetilde\nu^n_{\balpha}\}_{\balpha\in\mathcal A}$ is feasible for
\eqref{eq-equnconditionalDROII} with radii $\epsilon$. Its
objective value is $(1-t_n)R_n+t_nR\le \mathcal L^{\rm II}(\epsilon)$. On the other hand, 
\begin{align*}
\limsup_{n\to\infty}\{(1-t_n)R_n+t_nR\}=\limsup_{n\to\infty}R_n=r.
\end{align*}
This implies $\mathcal L^{\rm II}(\epsilon)\ge r$. 
Hence $\mathcal L^{\rm II}$ is upper semicontinuous on $D$. This competes the proof.
\end{proof}

}

~\\
\begin{proof}[Proof of Theorem \ref{th:sumB2Duality}]
Fix any $\bar{\bbeta}\in\mathcal D$ and write
$\ell_{\bar{\bbeta}}(\bxi):=\ell(\bar{\bbeta},\bxi)$. To simplify notation, we write $\ell(\bxi)$ for this fixed loss throughout the derivation and restore the minimization over $\bbeta$ at the end. By Theorem \ref{th-analysisB2} and Lemma \ref{lm-BII-eqclosure}, the robust evaluation value for the fixed decision $\bar{\bbeta}$ satisfies
% $\sup_{Q\in\mathcal B^{\rm I}(\bepsilon),Q(\mathcal N)\in[a,b]} \E_Q[\ell(\bxi)|\bxi\in \mathcal N]$ is equivalent to
\begin{align}
&\sup_{Q\in\mathcal B^{\rm II}\left(\{\widehat{\p}_k\}_{k\in[K]},\epsilon\right),Q(\mathcal N)\in[a,b]} \E_Q[\ell(\bxi)\mid\bxi\in \mathcal N]\notag\\
&=\sup_{Q\in\mathcal B_{\mathcal N}^{\rm II}\left(\{\widehat{\p}_k\}_{k\in[K]},\epsilon\right)} \E_Q[\ell(\bxi)]\notag\\
&=\sup_{m, \{a_{k,j}\}_{j\in[N_k],k\in[K]}, \{r_{\balpha}\}_{\balpha\in\mathcal A}}\sup_{Q\in\mathcal B^{\rm II}\left(\left\{\sum_{j=1}^{N_k}a_{k,j}\delta_{\widehat{\bxi}_{k,j}}\right\}_{k\in[K]},\epsilon m-\sum_{\balpha\in\mathcal A}r_{\balpha}d_{\balpha}^{\btheta}\right), Q(\mathcal N)=1}\E_{Q}[\ell(\bxi)]\label{eq-B2SD1}\\
&\begin{array}{lll}
{\rm s.t.}~~&\sum_{j=1}^{N_k}a_{k,j}=1 &\forall k\in[K]\\
&a_{k,j}+\sum_{\balpha\in\mathcal A,\alpha_k=j} r_{\balpha}=m w_{k,j},~a_{k,j}\ge 0~~&\forall j\in[N_k], \forall k\in[K]\\
&m\in\left[\frac{1}{b},\frac{1}{a}\right],~\mathbf r\ge 0.
\end{array}\label{eq-B2conSD1}
\end{align}
By Theorem \ref{th-unconB2},
the inner supremum problem of \eqref{eq-B2SD1} is equivalent to
\begin{align}\label{eq-B2conSD2}
\begin{array}{lll}\inf & 
\lambda\left(\epsilon m-\sum_{\balpha\in\mathcal A}r_{\balpha}d_{\balpha}^{\btheta}\right)+\sum_{k=1}^K\sum_{j=1}^{N_k}a_{k,j}\gamma_{k,j}\\
\text {s.t.} & \lambda\ge 0,~\bm\gamma_k \in \mathbb{R}^{N_k} & \forall k\in[K] \\ 
& \sup _{\bxi \in \mathcal N} \left\{\ell(\bxi)-\lambda\sum_{k=1}^K \theta_k c(\bxi,\widehat{\bxi}_{k, \alpha_k})\right\} \leq \sum_{k=1}^K \gamma_{k, \alpha_k} & \forall \balpha \in \mathcal{A}.
\end{array}
\end{align}
Denote by $C$ and $D$ the sets of all variables satisfying all conditions in \eqref{eq-B2conSD1} and \eqref{eq-B2conSD2}, respectively.
For the fixed decision $\bar{\bbeta}$, the robust evaluation problem is equivalent to
\begin{align*}
\sup_{\left\{m,\{a_{k,j}\}_{j\in[N_k],k\in[K]}, \{r_{\balpha}\}_{\balpha\in\mathcal A}\right\}\in C}\inf_{\left\{\lambda,\{\bm\gamma_k\}_{k\in[K]}\right\}\in D} 
\left\{\lambda\left(\epsilon m-\sum_{\balpha\in\mathcal A}r_{\balpha}d_{\balpha}^{\btheta}\right)+\sum_{k=1}^K\sum_{j=1}^{N_k}a_{k,j}\gamma_{k,j}\right\}.
\end{align*}
Since $C$ is a compact convex polytope ($m$ lies in a box, each $a_k$ in a simplex, and $\sum_{\balpha}r_{\balpha}=m-1$ with $\mathbf r\ge0$), $D$ is convex, and the objective is bilinear, Sion's minimax theorem \citep{S58} justifies the interchange. Thus the value is
\begin{align}\label{B1conSD3}
\inf_{\left\{\lambda,\{\bm\gamma_k\}_{k\in[K]}\right\}\in D} \sup_{\left\{m,\{a_{k,j}\}_{j\in[N_k],k\in[K]}, \{r_{\balpha}\}_{\balpha\in\mathcal A}\right\}\in C}
\left\{\lambda\left(\epsilon m-\sum_{\balpha\in\mathcal A}r_{\balpha}d_{\balpha}^{\btheta}\right)+\sum_{k=1}^K\sum_{j=1}^{N_k}a_{k,j}\gamma_{k,j}\right\}.
\end{align}
Note that the inner supremum problem of \eqref{B1conSD3} is a linear program. By standard calculation, we have the following duality for it:
\begin{align*}
\begin{array}{lll}\inf & 
\sum_{k=1}^K\varrho_k-\frac{1}{b}\tau_1+\frac{1}{a}\tau_2\\
\text {s.t. } & \varrho_k\in\R,~\bm\varphi_{k}\in\R^{N_k},~\bm\psi_k\in\R_+^{N_k} &\forall k\in[K]\\
&\tau_1,\tau_2\in\R_+\\
&\eta_{\balpha}\in\R_+ &\forall \balpha\in\mathcal A\\ 
& \lambda \epsilon-\sum_{k=1}^K\sum_{j=1}^{N_k}{\varphi_{k,j}}{w_{k,j}}+\tau_1-\tau_2=0\\
&\gamma_{k,j}-\varrho_k+\varphi_{k,j}+\psi_{k,j}=0 &\forall j\in[N_k],~k\in[K]\\
&-\lambda d_{\balpha}^{\btheta}+\sum_{k=1}^K\varphi_{k,\alpha_k}+\eta_{\balpha}=0 &\forall \balpha\in\mathcal A.
\end{array}
\end{align*}
Combining with \eqref{B1conSD3} gives \eqref{prob-sum} for the fixed decision $\bar{\bbeta}$.
%, with $\ell(\cdot)$ interpreted as $\ell(\bar{\bbeta},\cdot)$. 
Since $\bar{\bbeta}$ was arbitrary, minimizing the resulting representation over $\bbeta\in\mathcal D$ gives the stated formulation.
\end{proof}

\section{Proofs of results in Section \ref{sec:tractabilityB3}}\label{app:proofBarycenter}

\begin{proof}[Proof of Corollary \ref{co:barycenterB3Duality}]
Equation~\eqref{eq-B3eqclosure} follows directly from Lemma~\ref{lm-BII-eqclosure} specialized to the single-source case.

Applying Theorem~\ref{th:sumB2Duality} to the single-source setting, with reference distribution $\widehat{\p}_{\rm bar}$, we obtain that problem~\eqref{prob:B3inner} is equivalent to
\begin{align*}%\tag{${\rm P}^{\rm II}$}\label{prob-sum}
\begin{array}{lll}\inf & 
\varrho-\frac{1}{b}\tau_1+\frac{1}{a}\tau_2\\
\text{\rm s.t.} & \bbeta\in\mathcal D,\ \lambda, \tau_1, \tau_2\in\R_+,~\varrho\in\R\\
&\varphi_{j}\in\R,~\psi_j\in\R_+,~\gamma_j\in\R,~\eta_j\in\R_+ &\forall j\in[N]\\
& \lambda \epsilon-\sum_{j=1}^{N}w_j\varphi_{j}+\tau_1-\tau_2=0\\
&\gamma_{j}-\varrho+\varphi_{j}+\psi_{j}=0 &\forall j\in[N]\\
&-\lambda d_{j}+\varphi_{j}+\eta_{j}=0 &\forall j\in[N]\\
& \sup _{\bxi \in \mathcal N} \left\{\ell(\bbeta,\bxi)-\lambda c(\bxi,\widehat{\bxi}_{j})\right\} \leq  \gamma_{j} & \forall j \in [N].
\end{array}
\end{align*}
Since $\eta_j\ge 0$ and $\varphi_j\in\R$,
the constraint $-\lambda d_{j}+\varphi_{j}+\eta_{j}=0$ can be replaced by the constraints: $\varphi_j=\lambda d_j-\varphi_j'$ and $\varphi_j'\ge 0$. Substituting $\varphi_j=\lambda d_j-\varphi_j'$ into the other constraints yields the desired result.\end{proof}

\section{Proofs of results in Section \ref{sec:feasible}}

\begin{proof}[Proof of Proposition \ref{prop:radiiregion1}]
Define
\begin{align*}
&m=\inf\left\{W\left(Q,\widehat{\p}_1\right): Q(\mathcal N)\in[a,b]\right\};\notag\\
&f(\epsilon)=\inf\left\{W\left(Q,\widehat{\p}_2\right): W\left(Q,\widehat{\p}_1\right)\le \epsilon,~Q(\mathcal N)\in[a,b]\right\},~~\epsilon\ge m.%\label{eq-regionradiif}
\end{align*}
By the definition of $\mathcal E^{\rm I}$ in \eqref{eq-radiiregion1}, we have
\begin{align*}
\mathcal E^{\rm I}=\left\{\bepsilon\in\R^2_+: \mbox{There exists $Q$ such that $W(Q,\widehat{\p}_1)\le \epsilon_1$, $W(Q,\widehat{\p}_2)\le \epsilon_2$ and $Q(\mathcal N)\in[a,b]$}\right\}.
\end{align*}
From the above expression, $\mathcal E^{\rm I}=\{\bepsilon\in\R_+^2:\epsilon_1\ge m,\ \epsilon_2\ge f(\epsilon_1)\}$ and $\mathcal E^{\rm I}$ is upward closed. Moreover, the set $\{(Q,\epsilon_1):W(Q,\widehat\p_1)\le\epsilon_1,\ Q(\mathcal N)\in[a,b]\}$ is convex and the objective $W(Q,\widehat\p_2)$ is convex in $Q$, so $f$ is convex and nonincreasing on $[m,\infty)$ and hence continuous on $(m,\infty)$. Combining upward closedness with this continuity yields
\begin{align*}
(\mathcal E^{\rm I})^{\circ}=\{\bepsilon\in\R^2_+: \epsilon_1> m,~\epsilon_2>f(\epsilon_1)\}.
\end{align*}
Below we aim to prove that $m$ and $f(\epsilon_1)$ are the optimal values of the optimization problems in \eqref{eq-regionradii1} and \eqref{eq-regionradii2}, respectively. 

We first consider $m$. By Lemma \ref{lm:DMM}, $m$ equals the optimal value of the problem:
\begin{align}\label{eq-regionradiipf1}
\begin{array}{lll}
\inf &\sum_{j=1}^{N_1} \int_{\Xi} c(\bxi,\widehat{\bxi}_{1,j})\d\mu_j(\bxi)\\
{\rm s.t.} &\mu_j\in\mathcal M_+(\Xi),~\mu_j(\Xi)=w_{1,j}~~&\forall j\in[N_1]\\
& \sum_{j=1}^{N_1}\mu_j(\mathcal N)\in[a,b].
\end{array}
\end{align}
Let $p_j=\mu_j(\mathcal N)$ and $q_j=\mu_j(\mathcal N^c)$, and it holds that $\mu_j=p_j\mu_j^{\mathcal N}+q_j\mu_j^{\mathcal N^c}$. Substituting them into \eqref{eq-regionradiipf1} yields
\begin{align*}
\begin{array}{lll}
\inf &\sum_{j=1}^{N_1} \int_{\Xi} p_j c(\bxi,\widehat{\bxi}_{1,j})\d\mu_j^{\mathcal N}(\bxi)+\sum_{j=1}^{N_1} \int_{\Xi} q_j c(\bxi,\widehat{\bxi}_{1,j})\d\mu_j^{\mathcal N^c}(\bxi)\\
{\rm s.t.} &\sum_{j=1}^{N_{1}}p_j\in[a,b]\\
&p_j+q_j=w_{1,j},~p_j\ge 0,~q_j\ge 0 &\forall j\in[N_{1}].
\end{array}
\end{align*}
Since $\mu_j^{\mathcal N}\in\mathcal P(\mathcal N)$ and $\mu_j^{\mathcal N^c}\in\mathcal P(\mathcal N^c)$, it is straightforward to check that the above problem is equivalent to \eqref{eq-regionradii1}. Hence, we have verified that $m$ is equal to the optimal value of \eqref{eq-regionradii1}.

Let us now focus on $f(\epsilon_1)$ with $\epsilon_1> m$.
Define
\begin{align*}
F\left(\epsilon,Q\right)= W\left(Q,\widehat{\p}_2\right)+\delta_{\mathcal F(\epsilon)}(Q),
\end{align*}
where 
$\mathcal F(\epsilon):=\{Q: W(Q,\widehat{\p}_1)\le \epsilon,~Q(\mathcal N)\in[a,b]\}$,
and $\delta_{\mathcal F(\epsilon)}(Q)$ is the extended-valued indicator, equal to $0$ if $Q\in\mathcal F(\epsilon)$ and $+\infty$ otherwise. By the convexity of the mapping $Q\mapsto W(Q,\widehat{\p}_2)$ and the set $\mathcal F(\epsilon)$, we have that $F(\epsilon,Q)$ is jointly convex in $(\epsilon,Q)$
Note that $f(\epsilon)=\inf_{Q}F(\epsilon,Q)$, and thus, $f(\epsilon)$ is convex. Consider its conjugate function
\begin{align*}
f^*(\lambda)&=\sup_{\epsilon\ge m}\{\lambda\epsilon-f(\epsilon)\}\\
&=\sup_{\epsilon\ge m}\sup_{Q:Q(\mathcal N)\in[a,b], W(Q,\widehat{\p}_1)\le \epsilon} \left\{\lambda\epsilon- W(Q,\widehat{\p}_2)\right\}\\
&=\sup_{Q:Q(\mathcal N)\in[a,b], W(Q,\widehat{\p}_1)\le \epsilon} \left\{\lambda\epsilon- W(Q,\widehat{\p}_2)\right\},
% &=\sup_{Q:Q(\mathcal N)\in[a,b]}\left\{\lambda W(Q,\widehat{\p}_1)- W(Q,\widehat{\p}_2)\right\}.
\end{align*}
where the last equality holds because $W(Q,\widehat{\p}_1)\ge m$ whenever $Q(\mathcal N)\in[a,b]$. It is straightforward to see that $f^*(\lambda)=\infty$ if $\lambda>0$. For $\lambda\le 0$, we have
\begin{align*}
f^*(\lambda)=\sup_{Q:Q(\mathcal N)\in[a,b]} \left\{\lambda W(Q,\widehat{\p}_1)- W(Q,\widehat{\p}_2)\right\}.
\end{align*}
Therefore, the conjugate function of $f^*$ can be represented as
\begin{align*}
f^{**}(\epsilon)&=\sup_{\lambda\le 0} \{\lambda\epsilon-f^*(\lambda)\}\\
&=\sup_{\lambda\le 0}\inf_{Q: Q(\mathcal N)\in[a,b]} \{\lambda\epsilon-\lambda W(Q,\widehat{\p}_1)+W(Q,\widehat{\p}_2)\}\\
&=\sup_{\lambda\ge 0}\inf_{Q: Q(\mathcal N)\in[a,b]} \{W(Q,\widehat{\p}_2)+\lambda W(Q,\widehat{\p}_1)-\lambda\epsilon\}.
\end{align*}
Note that the condition $\epsilon_1>m$ implies that $\epsilon_1$ lies in the interior of the effective domain of $f$. The biconjugacy theorem \citep[Theorem 12.2]{R70} gives
\begin{align*}
f(\epsilon_1)=f^{**}(\epsilon_1)
=\sup_{\lambda\ge 0}\inf_{Q: Q(\mathcal N)\in[a,b]} \{W(Q,\widehat{\p}_2)+\lambda W(Q,\widehat{\p}_1)-\lambda\epsilon_1\}
\end{align*}
By Lemma \ref{lm:DMM}, one can check with standard calculation that the above problem is equivalent to 
\begin{align}\label{eq-regionradiipf2}
\begin{array}{llll}
\sup_{\lambda\ge 0}&\inf &\sum_{\balpha\in\mathcal A} \int_{\Xi}\left( c(\bxi,\widehat{\bxi}_{2,\alpha_2})+\lambda c(\bxi,\widehat{\bxi}_{1,\balpha_1})\right)\d\mu_{\balpha}(\bxi)-\lambda\epsilon_1\\
&{\rm s.t.} &\mu_{\balpha}\in\mathcal M_+(\Xi) &\forall \balpha\in\mathcal A\\
&&\sum_{\balpha\in\mathcal A,\alpha_k=j}\mu_{\balpha}(\Xi)=w_{k,j}~~&\forall j\in[N_k],~k=1,2\\
&& \sum_{\balpha\in\mathcal A}\mu_{\balpha}(\mathcal N)\in[a,b].
\end{array}
\end{align}
Similar to the analysis for $m$, we can set $p_{\balpha}=\mu_{\balpha}(\mathcal N)$ and $q_{\balpha}=\mu_{\balpha}(\mathcal N^c)$, and thus, $\mu_{\balpha}=p_{\balpha}\mu_{\balpha}^{\mathcal N}+q_{\balpha}\mu_{\balpha}^{\mathcal N^c}$. Substituting them into \eqref{eq-regionradiipf2} with some standard calculations yields that the problem \eqref{eq-regionradiipf2} is equivalent to the problem in \eqref{eq-regionradii2}. This completes the proof.
\end{proof}

~\\
\begin{proof}[Proof of Corollary \ref{prop:special}]
Note that under this setting $\inf_{\bxi\in\mathcal
N}c(\bxi,\widehat{\bxi}_{1,j})=\|\widehat{\bx}_{1,j}-\bx_0\|_{\mathcal X}$ and $\inf_{\bxi\in\mathcal N^c}c(\bxi,\widehat{\bxi}_{1,j})=0$. Thus, the equivalent reformulation of the problem in \eqref{eq-regionradii1} holds. To see the equivalent reformulation of the problem in \eqref{eq-regionradii2}, it is straightforward to verify that 
\begin{align*}
\inf_{\bxi_{\balpha}\in\mathcal N}\left\{c(\bxi_{\balpha},\widehat{\bxi}_{2,\alpha_2})+\lambda c(\bxi_{\balpha},\widehat{\bxi}_{1,\balpha_1})\right\}=\|\widehat{\bx}_{2,\alpha_2}-\bx_0\|_{\mathcal X}+\lambda \|\widehat{\bx}_{1,\alpha_1}-\bx_0\|_{\mathcal X}+(\lambda\wedge 1)d_{\balpha}^{\mathcal Y}
\end{align*}
and 
\begin{align*}
\inf_{\bxi_{\balpha}'\in\mathcal N^c}\left\{c(\bxi_{\balpha}',\widehat{\bxi}_{2,\alpha_2})+\lambda c(\bxi_{\balpha}',\widehat{\bxi}_{1,\balpha_1})\right\}=(\lambda \wedge 1)d_{\balpha},
\end{align*}
where $\lambda\wedge 1=\min\{\lambda,1\}$.
Therefore, the problem in \eqref{eq-regionradii2} is equivalent to 
\begin{align}\label{eq-eqregionradii2}
\begin{array}{llll}
\sup_{\lambda\ge 0}&\inf_{\mathbf p,\bq} 
&\sum_{\balpha\in\mathcal A}p_{\balpha}\left(\|\widehat{\bx}_{2,\alpha_2}-\bx_0\|_{\mathcal X}+\lambda\|\widehat{\bx}_{1,\alpha_1}-\bx_0\|_{\mathcal X}\right)\\
&&+(\lambda \wedge1)\sum_{\balpha\in\mathcal A}\left(p_{\balpha}d_{\balpha}^{\mathcal Y}+q_{\balpha}d_{\balpha}\right)-\lambda\epsilon_1\\
&{\rm s.t.} &\sum_{\balpha\in\mathcal A}p_{\balpha}\in[a,b]\\
&&p_{\balpha},~q_{\balpha}\ge 0 &\forall \balpha\in\mathcal A\\
&&\sum_{\balpha\in\mathcal A,\alpha_{k}=j} (p_{\balpha}+q_{\balpha})= w_{k,j} &\forall j\in[N_{k}], k=1,2.
\end{array}
\end{align}
Note that the objective function in \eqref{eq-eqregionradii2} is  concave in $\lambda$ and linear in $(\bp,\bq)$, where the feasible set of $(\bp,\bq)$ is compact. It follows from a minimax theorem (see e.g., \cite{S58}) that \eqref{eq-eqregionradii2} is equivalent to
\begin{align*}%\label{eq-eqregionradii2}
\begin{array}{llll}
&\inf_{\mathbf p,\bq}\sup_{\lambda\ge 0}
&\sum_{\balpha\in\mathcal A}p_{\balpha}\left(\|\widehat{\bx}_{2,\alpha_2}-\bx_0\|_{\mathcal X}+\lambda\|\widehat{\bx}_{1,\alpha_1}-\bx_0\|_{\mathcal X}\right)\\
&&+(\lambda \wedge1)\sum_{\balpha\in\mathcal A}\left(p_{\balpha}d_{\balpha}^{\mathcal Y}+q_{\balpha}d_{\balpha}\right)-\lambda\epsilon_1\\
&{\rm s.t.} &\sum_{\balpha\in\mathcal A}p_{\balpha}\in[a,b]\\
&&p_{\balpha},~q_{\balpha}\ge 0 &\forall \balpha\in\mathcal A\\
&&\sum_{\balpha\in\mathcal A,\alpha_{k}=j} (p_{\balpha}+q_{\balpha})= w_{k,j} &\forall j\in[N_{k}], k=1,2.
\end{array}
\end{align*}
The inner supremum problem can be reformulated as
\begin{align*}
\begin{array}{lll}
&\sup_{\lambda,t}
&\sum_{\balpha\in\mathcal A}p_{\balpha}\left(\|\widehat{\bx}_{2,\alpha_2}-\bx_0\|_{\mathcal X}+\lambda\|\widehat{\bx}_{1,\alpha_1}-\bx_0\|_{\mathcal X}\right)+t\sum_{\balpha\in\mathcal A}\left(p_{\balpha}d_{\balpha}^{\mathcal Y}+q_{\balpha}d_{\balpha}\right)-\lambda\epsilon_1\\
&{\rm s.t.} &\lambda\ge 0,~t\le \lambda,~t\le 1.
\end{array}
\end{align*}
Finally, taking the dual of the above supremum problem leads to the desired equivalent representation.
\end{proof}

~~\\
\begin{proof}[Proof of  Proposition \ref{prop:feasibleradiiII}]
Let
\begin{align*}
\epsilon^*_{\rm II}:=
\inf\left\{
\sum_{k=1}^K\theta_k W(Q,\widehat{\p}_k):
Q\in\mathcal P(\Xi),\ Q(\mathcal N)\in[a,b]
\right\}.
\end{align*}
Because $\mathcal E^{\rm II}$ is upward closed, its left endpoint is
$\epsilon^*_{\rm II}$. We show that the displayed linear program has this
value.

By Lemma~\ref{lm:DMM}, the objective defining $\epsilon^*_{\rm II}$ can be
written as the value of
\begin{align*}
\begin{array}{llll}
\inf_{\{\mu_{\balpha}\}_{\balpha\in\mathcal A}}
&\displaystyle
\sum_{\balpha\in\mathcal A}\int_{\Xi}
\sum_{k=1}^K\theta_k c(\bxi,\widehat{\bxi}_{k,\alpha_k})
\d\mu_{\balpha}(\bxi)\\
{\rm s.t.}
&\mu_{\balpha}\in\mathcal M_+(\Xi) &\forall \balpha\in\mathcal A\\
&\displaystyle
\sum_{\balpha\in\mathcal A,\alpha_k=j}\mu_{\balpha}(\Xi)=w_{k,j}
&\forall j\in[N_k],~k\in[K]\\
&\displaystyle
\sum_{\balpha\in\mathcal A}\mu_{\balpha}(\mathcal N)\in[a,b].
\end{array}
\end{align*}
For each $\balpha$, set
$p_{\balpha}:=\mu_{\balpha}(\mathcal N)$ and
$q_{\balpha}:=\mu_{\balpha}(\mathcal N^c)$. Splitting
$\mu_{\balpha}$ into its restrictions to $\mathcal N$ and
$\mathcal N^c$ gives
\begin{align*}
\int_{\Xi}\sum_{k=1}^K\theta_k c(\bxi,\widehat{\bxi}_{k,\alpha_k})
\d\mu_{\balpha}(\bxi)
\ge
p_{\balpha}h_{\balpha}^{\btheta}
{}+
q_{\balpha}d_{\balpha}^{\btheta}.
\end{align*}
Therefore $\epsilon^*_{\rm II}$ is bounded below by the linear program in
the proposition. Conversely, given any feasible $(\bp,\bq)$ for that
linear program, choose points in $\mathcal N$ and $\mathcal N^c$ whose
costs are within an arbitrary $\delta>0$ of the infima defining
$h_{\balpha}^{\btheta}$ and $d_{\balpha}^{\btheta}$, and place masses
$p_{\balpha}$ and $q_{\balpha}$ on those points. The resulting measures
$\mu_{\balpha}$ satisfy the marginal and conditioning constraints above
and have objective value within $\delta\sum_{\balpha}(p_{\balpha}+q_{\balpha})$
of the linear-program value. Letting $\delta\downarrow0$ proves equality.
\end{proof}

\section{Proofs of results in Section \ref{sec:extension}}

\begin{proof}[Proof of Theorem \ref{th-minimax}]
Define 
\begin{align*}
\mathcal P_{\mathcal N}=\left\{Q^{\mathcal N}: Q\in\mathcal P,~Q(\mathcal N)\in A\right\},
\end{align*}
where we recall that $Q^{\mathcal N}$ is defined as $Q^{\mathcal N}(B)=Q(B\cap \mathcal N)/Q(\mathcal N)$ for all Borel $B\subseteq\Xi$.
Thus, we have $\mathcal P_{\mathcal N}^{\ell}=\left\{Q\circ\ell^{-1}: Q\in\mathcal P_{\mathcal N}\right\}$.
% Define
% \begin{align*}
% R(Q,\bt)=\E_Q[\phi(Z,\bt)],~~Q\in\mathcal P(\R),~\bt\in\R^n.
% \end{align*}
It holds that 
\begin{align*}
\sup_{Q\in\mathcal P,Q(\mathcal N)\in A} \mathcal R_Q(\ell(\bxi)\mid\bxi\in\mathcal N)
&=\sup_{Q\in\mathcal P, Q(\mathcal N)\in A}\inf_{\mathbf t\in\R^n} \E_Q[\phi(\ell(\bxi),\mathbf t)\mid\bxi\in\mathcal N]\\
&=\sup_{Q\in\mathcal P_{\mathcal N}}\inf_{\mathbf t\in\R^n} \E_Q[\phi(\ell(\bxi),\mathbf t)]\\
&=\sup_{Q\in\mathcal P_{\mathcal N}^{\ell}}\inf_{\mathbf t\in\R^n} \E_Q[\phi(Z,\mathbf t)]=\sup_{Q\in\mathcal P_{\mathcal N}^{\ell}}\inf_{\mathbf t\in\R^n} R(Q,\bt)
\end{align*}
and
\begin{align*}
\inf_{\mathbf t\in\R^n}\sup_{Q\in\mathcal P, Q(\mathcal N)\in A}\E_Q[\phi(\ell(\bxi),\mathbf t)\mid\bxi\in\mathcal N]
&=\inf_{\mathbf t\in\R^n}\sup_{Q\in\mathcal P_{\mathcal N}}\E_Q[\phi(\ell(\bxi),\mathbf t)]\\
&=\inf_{\mathbf t\in\R^n}\sup_{Q\in\mathcal P_{\mathcal N}^{\ell}}\E_Q[\phi(Z,\mathbf t)]
=\inf_{\mathbf t\in\R^n}\sup_{Q\in\mathcal P_{\mathcal N}^{\ell}} R(Q,\bt).
\end{align*}
Below we aim to prove the minimax result that 
\begin{align*}%\label{eq-minmax1}
\sup_{Q\in\mathcal P_{\mathcal N}^{\ell}}\inf_{\mathbf t\in\R^n} R(Q,\bt)=\inf_{\mathbf t\in\R^n} \sup_{Q\in\mathcal P_{\mathcal N}^{\ell}}R(Q,\bt),
\end{align*}
and this will complete the proof.
By the minimax inequality, it holds that 
\begin{align*}
\sup_{Q\in\mathcal P_{\mathcal N}^{\ell}}\inf_{\mathbf t\in\R^n} R(Q,\mathbf t)
\le\inf_{\mathbf t\in\R^n} \sup_{Q\in\mathcal P_{\mathcal N}^{\ell}}R(Q,\mathbf t).
\end{align*}
It remains to consider the other direction that 
\begin{align}\label{eq-minimaxsuff}
\sup_{Q\in\mathcal P_{\mathcal N}^{\ell}}\inf_{\mathbf t\in\R^n} R(Q,\mathbf t)
\ge\inf_{\mathbf t\in\R^n} \sup_{Q\in\mathcal P_{\mathcal N}^{\ell}}R(Q,\mathbf t).
\end{align}

To verify \eqref{eq-minimaxsuff},
we first claim 
that $\mathcal P_{\mathcal N}$ is a convex set, and as a result, $\mathcal P_{\mathcal N}^{\ell}$ is also a convex set. Indeed, for $R_1,R_2\in\mathcal P_{\mathcal N}$, there exist $Q_1,Q_2\in\mathcal P$ such that $R_i=Q_i^{\mathcal N}$ and $q_i:=Q_i(\mathcal N)\in A$ for $i=1,2$. Fix $\theta\in[0,1]$. If $\theta\in\{0,1\}$ the claim is immediate. Otherwise, define
\begin{align*}
\rho:=\frac{\theta q_2}{\theta q_2+(1-\theta)q_1}\in(0,1).
\end{align*}
Since $\mathcal P$ is convex, $\widetilde Q:=\rho Q_1+(1-\rho)Q_2$ belongs to $\mathcal P$. Moreover,
\begin{align*}
\widetilde Q(\mathcal N)=\rho q_1+(1-\rho)q_2
=\frac{q_1q_2}{\theta q_2+(1-\theta)q_1},
\end{align*}
which lies between $q_1$ and $q_2$ and therefore belongs to $A$ because $A$ is an interval. Finally, for every measurable $B\subseteq\Xi$,
\begin{align*}
\widetilde Q^{\mathcal N}(B)
&=\frac{\rho Q_1(B\cap\mathcal N)+(1-\rho)Q_2(B\cap\mathcal N)}
{\rho q_1+(1-\rho)q_2}\\
&=\theta Q_1^{\mathcal N}(B)+(1-\theta)Q_2^{\mathcal N}(B)
=\theta R_1(B)+(1-\theta)R_2(B).
\end{align*}
Thus $\theta R_1+(1-\theta)R_2\in\mathcal P_{\mathcal N}$, proving that $\mathcal P_{\mathcal N}$ is convex. The pushforward map $Q\mapsto Q\circ\ell^{-1}$ is linear, so $\mathcal P_{\mathcal N}^{\ell}$ is convex as well.

Next, we consider the case that $\mathcal P_{\mathcal N}^{\ell}$ is a convex polytope, which has the form $\mathcal P_{\mathcal N}^{\ell}={\rm conv}\{P_1,\dots,P_S\}$ where ``$\rm conv$" means the convex hull. We claim that $\bigcup_{Q\in\mathcal P_{\mathcal N}^{\ell}}\arg\min_{\mathbf t}R(Q,\mathbf t)$ is a subset of a compact set in $\R^n$. Indeed, denote by 
\begin{align}\label{eq-minmaxa}
a=\sup_{Q\in\mathcal P_{\mathcal N}^{\ell}}\inf_{\mathbf t\in\R^n}R(Q,\mathbf t),
\end{align}
and let $B_r=\{\mathbf t\in\R^n:\|\mathbf t\|\le r\}$ be a ball with radius $r\ge 0$. By Assumption \ref{ass:phi2}, there exists $r_0>0$ such that 
\begin{align}\label{eq-minmax2}
\min_{Q\in\{P_1,\dots,P_S\}}R(Q,\mathbf t)>a~~\forall \mathbf t\in (B_{r_0})^c.
\end{align}
% which implies 
% \begin{align*}
% \bigcup_{P\in\{P_1,\dots,P_k\}}\arg\min_{\mathbf t}R(P,\mathbf t)\subseteq B_{a_0}.
% \end{align*}
Let $Q\in\mathcal P_{\mathcal N}^{\ell}$ with the form $Q=\sum_{i=1}^S \lambda_i P_i$, where $\lambda_i\ge 0$ and $\sum_{i=1}^S\lambda_i=1$. It holds that
\begin{align*}
R(Q,\mathbf t)=R\left(\sum_{i=1}^S\lambda_iP_i,\mathbf t\right)
=\sum_{i=1}^S \lambda_i R(P_i,\mathbf t)\ge \min_{Q\in\{P_1,\dots,P_S\}}R(Q,\mathbf t)>a\ge \inf_{\mathbf t\in\R^n}R(Q,\mathbf t)~~\forall \mathbf t\in (B_{r_0})^c,
\end{align*}
where the second equality holds because $R(Q,\bt)$ is linear in $Q$, the second inequality is due to \eqref{eq-minmax2} and the last inequality follows from the definition of $a$ in \eqref{eq-minmaxa}.
This implies that $\arg\min_{\mathbf t} R(Q,\mathbf t)\subseteq B_{r_0}$ for all $Q\in\mathcal P_{\mathcal N}^{\ell}$, and hence, we have 
\begin{align}\label{eq-minimaxcompact}
\bigcup_{Q\in\mathcal P_{\mathcal N}^{\ell}}\arg\min_{\mathbf t} R(Q,\mathbf t)\subseteq B_{r_0},
\end{align}
where $B_{r_0}$ is a compact set.
% We claim that $\arg\min_{\mathbf t}R(P,\mathbf t)$ is a bounded set for all $P\in\mathcal P$. This is because Assumption \ref{ass:phi2} implies that $R(P,\mathbf t_n)\to\infty$ for any sequence $\{\mathbf t_n\}_{n\in\R}$ with $\|\mathbf t_n\|\to\infty$. 
% Hence, there exists a compact set $A$ such that 
Further, 
% we conclude that
% \begin{align}\label{eq-minimaxcompact}
% \sup_{P\in\mathcal P}\inf_{\mathbf t\in\R^m}R(P,\mathbf t)
% =\sup_{P\in\mathcal P}\inf_{\mathbf t\in B_{a_0}}R(P,\mathbf t).
% \end{align}
note that $R(Q,\mathbf t)$ is linear in $Q$. Assumption \ref{ass:phi1} implies that $R(Q,\mathbf t)$ is convex and lower-semicontinuous in $\mathbf t$, and $B_{r_0}$ is a compact set. We have
\begin{align*}
%\sup_{P\in\mathcal P} \rho(P)
\sup_{Q\in\mathcal P_{\mathcal N}^{\ell}}\inf_{\mathbf t\in\R^n}R(Q,\mathbf t)
=\sup_{Q\in\mathcal P_{\mathcal N}^{\ell}}\inf_{\mathbf t\in B_{r_0}}R(Q,\mathbf t)
=\inf_{\mathbf t\in B_{r_0}}\sup_{Q\in\mathcal P_{\mathcal N}^{\ell}}R(Q,\mathbf t)
\ge \inf_{\mathbf t\in \R^n}\sup_{Q\in\mathcal P_{\mathcal N}^{\ell}}R(Q,\mathbf t),
\end{align*}
where the first equality follows from \eqref{eq-minimaxcompact} and the second equality is due to a minimax theorem (e.g., \cite{S58}).
%The other direction that $\sup_{P\in\mathcal P}\inf_{\mathbf t\in\R^m}R(P,\mathbf t)\le \inf_{\mathbf t\in\R^m}\sup_{P\in\mathcal P}R(P,\mathbf t)$ is trivial. 
This yields \eqref{eq-minimaxsuff} when $\mathcal P_{\mathcal N}^{\ell}$ is a convex polytope.

Consider now that $\mathcal P_{\mathcal N}^{\ell}$ is a general convex set. For a set of distributions $\mathcal Q\subseteq \mathcal P(\R)$, denote by 
\begin{align*}
\mbox{
$\mathcal R_{\mathcal Q}(\mathbf t):=\sup_{Q\in\mathcal Q}R(Q,\mathbf t)=\sup_{Q\in\mathcal Q}\E_Q[\phi(Z,\mathbf t)]$,~~ $\mathbf t\in\R^n$,
}
\end{align*}
and $b=\inf_{\mathbf t\in\R^n}\mathcal R_{\mathcal P_{\mathcal N}^{\ell}}(\mathbf t)$.
We aim to verify that $\sup_{Q\in\mathcal P_{\mathcal N}^{\ell}}\inf_{\mathbf t\in\R^n}R(Q,\mathbf t)\ge b$.
From the definition of $b$, we have
\begin{align}\label{eq1-minimax}
\mathcal R_{\mathcal P_{\mathcal N}^{\ell}}(\mathbf t)\ge b~~\forall \mathbf t\in\R^n.
\end{align}
Let $Q_0\in\mathcal P_{\mathcal N}^{\ell}$. By Assumption \ref{ass:phi2}, there exists $r>0$ such that 
\begin{align}\label{eq2-minimax}
R(Q_0,\mathbf t)\ge b~~\forall \mathbf t\in (B_r)^c.
\end{align}
Let $\mathbb{D}=\{\mathbf t_i\}_{i\in\N}$ be a countable dense of $\R^n$, and let $\{Q_{i,j}\}_{i,j\in\N}\subseteq \mathcal P_{\mathcal N}^{\ell}$ be a sequence such that 
$$\lim_{j\to\infty}R(Q_{i,j},\mathbf t_i)= \mathcal R_{\mathcal P_{\mathcal N}^{\ell}}(\mathbf t_i)~~\forall i\in\N.$$ 
Define $\mathcal Q_S=\{Q_{1,S},Q_{2,S},\dots,Q_{S,S}\}\subseteq \mathcal P_{\mathcal N}^{\ell}$ for $S\in\N$. It holds that $\mathcal R_{\mathcal Q_S}(\mathbf t)\to \mathcal R_{\mathcal P_{\mathcal N}^{\ell}}(\mathbf t)$ as $S\to\infty$ for all $\bt\in\mathbb D$. Note that $\mathcal R_{\mathcal Q_S}:\R^n\to\R$ is convex as it is the supremum of a set of convex functions.
By Theorem 10.8 of \cite{R70} and noting that $B_r$ is a compact set, we have that $\{\mathcal R_{\mathcal Q_S}\}_{S\in\N}$ uniformly converges to $\mathcal R_{\mathcal P_{\mathcal N}^{\ell}}$ on $B_r$, i.e., 
\begin{align*}%\label{eq-minimax3}
\lim_{S\to\infty}\sup_{\mathbf t\in B_r}\left|\mathcal R_{\mathcal Q_S}(\mathbf t)-\mathcal R_{\mathcal P_{\mathcal N}^{\ell}}(\mathbf t)\right|=0.
\end{align*}
Combining with \eqref{eq1-minimax} implies that for any $\gamma>0$, we have 
\begin{align}\label{eq3-minimax}
\mathcal R_{\mathcal Q_S}(\mathbf t)\ge b-\gamma ~~\forall \mathbf t\in B_r
\end{align}
for large enough $S$.
Further define $\mathcal P_S={\rm conv}(\mathcal Q_S\cup \{Q_0\})\subseteq \mathcal P_{\mathcal N}^{\ell}$, and for $\gamma>0$ and large enough $S$,
\begin{align*}
\sup_{Q\in\mathcal P_{\mathcal N}^{\ell}}\inf_{\mathbf t\in\R^n} R(Q,\mathbf t)
&\ge\sup_{Q\in\mathcal P_S}\inf_{\mathbf t\in\R^n} R(Q,\mathbf t)
\ge \inf_{\mathbf t\in\R^n} \sup_{Q\in\mathcal P_S}R(Q,\mathbf t)
=\inf_{\mathbf t\in\R^n} \mathcal R_{\mathcal P_S}(\mathbf t)\\
&=\min\left\{ \inf_{\mathbf t\in B_r}\mathcal R_{\mathcal P_S}(\mathbf t), \inf_{\mathbf t\in (B_r)^c}\mathcal R_{\mathcal P_S}(\mathbf t)\right\}\\
&\ge \min\left\{ \inf_{\mathbf t\in B_r}\mathcal R_{\mathcal Q_S}(\mathbf t), \inf_{\mathbf t\in (B_r)^c}R(Q_0,\mathbf t)\right\}\\
&\ge \min\left\{ \inf_{\mathbf t\in B_r}\mathcal R_{\mathcal Q_S}(\mathbf t), b\right\}\ge \min\{b-\gamma,b\}=b-\gamma,
\end{align*}
where we have used the previous result for convex polytope in the second inequality by noting that $\mathcal P_S$ is a convex polytope, and the fourth and the fifth inequality follows from \eqref{eq2-minimax} and \eqref{eq3-minimax}, respectively.
Since $\gamma>0$ can be arbitrarily small, we conclude that $\sup_{Q\in\mathcal P_{\mathcal N}^{\ell}}\inf_{\mathbf t\in\R^n} R(Q,\mathbf t)\ge b$, and this verifies \eqref{eq-minimaxsuff} for a general convex set $\mathcal P_{\mathcal N}^{\ell}$. Hence, we complete the proof.
\end{proof}

\section{More refined formulations for the three models}\label{app:morerefined}

% In this section, we present more refined formulations of problems~\eqref{prob-capK=2}, \eqref{prob-sum}, and \eqref{prob-barycenterB3}, which are used in the the numerical experiments of Section \ref{sec:num}, as stated in Theorem~\ref{th:SDK=2}, Theorem~\ref{th:sumB2Duality}, and Corollary~\ref{co:barycenterB3Duality}, respectively, under additional structural assumptions. 

In this section, we present more refined formulations of the problems~\eqref{prob-capK=2}, \eqref{prob-sum}, and \eqref{prob-barycenterB3}, as stated in Corollary~\ref{th:SDK=2}, Theorem~\ref{th:sumB2Duality}, and Corollary~\ref{co:barycenterB3Duality}, respectively. 
%These formulations, developed under additional structural assumptions, are employed in the numerical experiments discussed in Section~\ref{sec:num}.
As they follow in a similar manner from the proof of Theorem~4.2 in \cite{EK18}, we omit the detailed proofs.
Throughout this appendix, we write the loss as $\ell(\bxi)$ for a fixed decision; when $\ell$ depends on $\bbeta$, the reformulations apply pointwise to $\ell(\bbeta,\cdot)$ and the outer minimization retains the constraint $\bbeta\in\mathcal D$.

Before presenting the results, we introduce some necessary preliminaries. In what follows, we assume that the cost function is a norm, i.e, $c(\bxi,\bxi')=\|\bxi-\bxi'\|$.
Denote by $\|\cdot\|_*$ the dual norm of $\|\cdot\|$, and $\sigma_{\mathcal N}$ is the support function on $\mathcal N$ defined as $\sigma_{\mathcal N}(\bm\eta)=\sup_{\bxi\in\mathcal N}\bm\eta^{\top}\bxi$. The conjugate function of $\ell$ on $\R^d$ is defined as $\ell^*(\bm\eta)=\sup_{\bxi\in\R^d}\{\bm\eta^{\top}\bxi-\ell(\bxi)\}$.

\begin{assumption}[Max-concave loss functions]\label{ass:lossfun1}
The conditional set $\mathcal N \subseteq \Xi$ is convex and closed.
The loss function has the form $\ell(\bxi)=\max_{l\in[L]}\ell_l(\bxi)$, where the negative constituent functions $-\ell_l$ are proper, convex, and lower semicontinuous for all $l \in [L]$. Moreover, we assume that $\ell_l$ is not identically $-\infty$ on $\mathcal N$ for all $l\in[L]$.
\end{assumption}

\begin{assumption}[Max-affine loss functions]\label{ass:lossfun2}
The conditional set has the form $\mathcal N=\{\bx_0\}\times\{\by:C\by\le \bg\}$, where $\bx_0\in\mathcal X$, $C\in \R^{m\times d_{\mathcal Y}}$ and $\bg\in \R^m$, and the polyhedron $\{\by:C\by\le\bg\}$ is bounded, so that $\mathcal N$ is compact.
The loss function has the form $\ell(\bxi)=\max_{l\in[L]}\{\bs_{l}^{\top}\by+t_{l}\}$, where $\bs_{l}\in\R^{d_{\mathcal Y}}$ and $t_l\in\R$ for all $l\in[L]$. 
\end{assumption}

\begin{proposition}\label{prop:refinedB1}
Suppose that Assumption \ref{ass:lossfun1} holds. Then, the problem \eqref{prob-capK=2} in Corollary \ref{th:SDK=2} is equivalent to
\begin{align*}%\tag{${\rm P}^{\rm I}$}\label{prob-capK=2}
\begin{array}{lll}\inf & 
\varrho_1+\varrho_2-\frac{1}{b}\tau_1+\frac{1}{a}\tau_2\\
 {\rm s.t. } & \varrho_k\in\R, \bm\varphi_{k}\in\R^{N_k}, \bm\psi_k\in\R_+^{N_k}, \zeta_k\in\R_+, \tau_k\in\R_+, \lambda_k \in \mathbb{R}_{+}, \bm\gamma_k \in \mathbb{R}^{N_k} &\forall k=1,2\\
&\eta_{\balpha}\in\R_+ &\forall \balpha\in\mathcal A\\ 
&\bz_{l,k,\alpha_k}\in\R^d,~\bv_{l,\balpha}\in\R^d &\forall \balpha\in\mathcal A, k=1,2,l\in[L]  \\
& \sum_{k=1}^2 \lambda_k\epsilon_k-\sum_{k=1}^2\sum_{j=1}^{N_k}{\varphi_{k,j}}{w_{k,j}}+\tau_1-\tau_2=0\\
&-\lambda_1+\lambda_2+\zeta_1-\zeta_2=0\\
&\gamma_{k,j}-\varrho_k+\varphi_{k,j}+\psi_{k,j}=0 &\forall j\in[N_k],~k=1,2\\
&(-\lambda_1+\zeta_1)d_{\balpha}+\sum_{k=1}^2\varphi_{k,\alpha_k}+\eta_{\balpha}=0 &\forall \balpha\in\mathcal A\\
% & \sup _{\bxi \in \mathcal N} \left\{\ell(\bxi)-\sum_{k=1}^2 \lambda_k \|\bxi-\widehat{\bxi}_{k, \alpha_k}\|\right\} \leq \sum_{k=1}^2 \gamma_{k, \alpha_k} & \forall \balpha \in \mathcal{A}\\
&(-\ell_l)^*\left(\sum_{k=1}^2\bz_{l,k,\alpha_k}-\bv_{l,\balpha}\right)+\sigma_{\mathcal N}(\bv_{l,\balpha})-\sum_{k=1}^2 \bz_{l,k,\alpha_k}^{\top}\widehat{\bxi}_{k, \alpha_k}\le \sum_{k=1}^2 \gamma_{k, \alpha_k} & \forall \balpha \in \mathcal{A}, l\in[L]\\
& \|\bz_{l,k,\alpha_k}\|_*\le \lambda_k &\forall \balpha\in\mathcal A, k=1,2, l\in[L].
\end{array}
\end{align*}
% where $\|\cdot\|_*$ is the dual norm of $\|\cdot\|$, and $\sigma_{\mathcal N}$ is the support function on $\mathcal N$ defined as $\sigma_{\mathcal N}(\bxi)=\sup_{\bm\eta\in\mathcal N}\bxi^{\top}\bm\eta$.
Moreover, if Assumption \ref{ass:lossfun2} holds, then the above problem reduces to
\begin{align*}%\tag{${\rm P}^{\rm I}$}\label{prob-capK=2}
\begin{array}{lll}\inf & 
\varrho_1+\varrho_2-\frac{1}{b}\tau_1+\frac{1}{a}\tau_2\\
 {\rm s.t. } & \varrho_k\in\R, \bm\varphi_{k}\in\R^{N_k}, \bm\psi_k\in\R_+^{N_k}, \zeta_k\in\R_+, \tau_k\in\R_+, \lambda_k \in \mathbb{R}_{+}, \bm\gamma_k \in \mathbb{R}^{N_k} &\forall k=1,2\\
&\eta_{\balpha}\in\R_+ &\forall \balpha\in\mathcal A\\ 
&\bz_{l,k,\alpha_k}\in\R^d,~\bomega_{l,\balpha}\in\R_+^m &\forall \balpha\in\mathcal A, k=1,2,l\in[L]  \\
& \sum_{k=1}^2 \lambda_k\epsilon_k-\sum_{k=1}^2\sum_{j=1}^{N_k}{\varphi_{k,j}}{w_{k,j}}+\tau_1-\tau_2=0\\
&-\lambda_1+\lambda_2+\zeta_1-\zeta_2=0\\
&\gamma_{k,j}-\varrho_k+\varphi_{k,j}+\psi_{k,j}=0 &\forall j\in[N_k],~k=1,2\\
&(-\lambda_1+\zeta_1)d_{\balpha}+\sum_{k=1}^2\varphi_{k,\alpha_k}+\eta_{\balpha}=0 &\forall \balpha\in\mathcal A\\
% & \sup _{\bxi \in \mathcal N} \left\{\ell(\bxi)-\sum_{k=1}^2 \lambda_k \|\bxi-\widehat{\bxi}_{k, \alpha_k}\|\right\} \leq \sum_{k=1}^2 \gamma_{k, \alpha_k} & \forall \balpha \in \mathcal{A}\\
&t_l+\bomega_{l,\balpha}^{\top}\bg-\sum_{k=1}^2\bz_{l,k,\alpha_k}^{\top}\left(\widehat{\bxi}_{k,\alpha_k}-(\bx_0,0,\dots,0)\right)\le \sum_{k=1}^2 \gamma_{k,\alpha_k} &\forall \balpha\in\mathcal A, l\in[L]\\
&-\sum_{k=1}^2\bz^{\mathcal Y}_{l,k,\alpha_k}+\bomega_{l,\balpha}^{\top}C=\bs_l &\forall \balpha\in\mathcal A, l\in[L]\\
& \|\bz_{l,k,\alpha_k}\|_*\le \lambda_k &\forall \balpha\in\mathcal A, k=1,2, l\in[L],
\end{array}
\end{align*}
where $\bz^{\mathcal Y}_{l,k,\alpha_k}$ denotes the vector composed of the components of $\bz_{l,k,\alpha_k}$ from the $(d_{\mathcal X}+1)$th to the $d$th dimensions.
\end{proposition}

\begin{proposition}\label{prop:refinedB2}
Suppose that Assumption \ref{ass:lossfun1} holds.
%and $c(\bxi,\bxi')=\|\bxi-\bxi'\|$. 
Then, the problem \eqref{prob-sum} in Theorem \ref{th:sumB2Duality} is equivalent to
\begin{align*}%\tag{${\rm P}^{\rm II}$}\label{prob-sum}
\begin{array}{lll}\inf & 
\sum_{k=1}^K\varrho_k-\frac{1}{b}\tau_1+\frac{1}{a}\tau_2\\
\text{\rm s.t.} & \lambda, \tau_1, \tau_2\in\R_+\\
&\varrho_k\in\R,~\bm\varphi_{k}\in\R^{N_k},~\bm\psi_k\in\R_+^{N_k},~\bm\gamma_k\in\R^{N_k} &\forall k\in[K]\\
&\eta_{\balpha}\in\R_+ &\forall \balpha\in\mathcal A\\ 
&\bz_{l,k,\alpha_k}\in\R^d,~\bv_{l,\balpha}\in\R^d &\forall \balpha\in\mathcal A, k\in[K],l\in[L]  \\
& \lambda \epsilon-\sum_{k=1}^K\sum_{j=1}^{N_k}{\varphi_{k,j}}{w_{k,j}}+\tau_1-\tau_2=0\\
&\gamma_{k,j}-\varrho_k+\varphi_{k,j}+\psi_{k,j}=0 &\forall j\in[N_k],~k\in[K]\\
&-\lambda d_{\balpha}^{\btheta}+\sum_{k=1}^K\varphi_{k,\alpha_k}+\eta_{\balpha}=0 &\forall \balpha\in\mathcal A\\
&(-\ell_l)^*\left(\sum_{k=1}^K\bz_{l,k,\alpha_k}-\bv_{l,\balpha}\right)+\sigma_{\mathcal N}(\bv_{l,\balpha})-\sum_{k=1}^K \bz_{l,k,\alpha_k}^{\top}\widehat{\bxi}_{k, \alpha_k}\le \sum_{k=1}^K \gamma_{k, \alpha_k} & \forall \balpha \in \mathcal{A}, l\in[L]\\
& \|\bz_{l,k,\alpha_k}\|_*\le \lambda\theta_k &\forall \balpha\in\mathcal A, k\in[K], l\in[L].
% & \sup _{\bxi \in \mathcal N} \left\{\ell(\bxi)-\lambda\sum_{k=1}^K \theta_k c(\bxi,\widehat{\bxi}_{k, \alpha_k})\right\} \leq \sum_{k=1}^K \gamma_{k, \alpha_k} & \forall \balpha \in \mathcal{A}.
\end{array}
\end{align*}
Moreover, if Assumption \ref{ass:lossfun2} holds, then the above problem reduces to
\begin{align*}%\tag{${\rm P}^{\rm II}$}\label{prob-sum}
\begin{array}{lll}\inf & 
\sum_{k=1}^K\varrho_k-\frac{1}{b}\tau_1+\frac{1}{a}\tau_2\\
\text{\rm s.t.} & \lambda, \tau_1, \tau_2\in\R_+\\
&\varrho_k\in\R,~\bm\varphi_{k}\in\R^{N_k},~\bm\psi_k\in\R_+^{N_k},~\bm\gamma_k\in\R^{N_k} &\forall k\in[K]\\
&\eta_{\balpha}\in\R_+ &\forall \balpha\in\mathcal A\\ 
&\bz_{l,k,\alpha_k}\in\R^d,~\bomega_{l,\balpha}\in\R_+^m &\forall \balpha\in\mathcal A, k\in[K],l\in[L]  \\
& \lambda \epsilon-\sum_{k=1}^K\sum_{j=1}^{N_k}{\varphi_{k,j}}{w_{k,j}}+\tau_1-\tau_2=0\\
&\gamma_{k,j}-\varrho_k+\varphi_{k,j}+\psi_{k,j}=0 &\forall j\in[N_k],~k\in[K]\\
&-\lambda d_{\balpha}^{\btheta}+\sum_{k=1}^K\varphi_{k,\alpha_k}+\eta_{\balpha}=0 &\forall \balpha\in\mathcal A\\
&t_l+\bomega_{l,\balpha}^{\top}\bg-\sum_{k=1}^K\bz_{l,k,\alpha_k}^{\top}\left(\widehat{\bxi}_{k,\alpha_k}-(\bx_0,0,\dots,0)\right)\le \sum_{k=1}^K \gamma_{k,\alpha_k} &\forall \balpha\in\mathcal A, l\in[L]\\
&-\sum_{k=1}^K\bz^{\mathcal Y}_{l,k,\alpha_k}+\bomega_{l,\balpha}^{\top}C=\bs_l &\forall \balpha\in\mathcal A, l\in[L]\\
& \|\bz_{l,k,\alpha_k}\|_*\le \lambda \theta_k &\forall \balpha\in\mathcal A, k\in[K], l\in[L].
\end{array}
\end{align*}
\end{proposition}

\begin{proposition}\label{prop:refinedB3}
Suppose that Assumption \ref{ass:lossfun1} holds.
%and $c(\bxi,\bxi')=\|\bxi-\bxi'\|$. 
Then, the problem \eqref{prob-barycenterB3} in Corollary \ref{co:barycenterB3Duality} is equivalent to
\begin{align*}%\tag{${\rm P}^{\rm III}$}\label{prob-barycenterB3}
\begin{array}{lll}\inf & 
\varrho-\frac{1}{b}\tau_1+\frac{1}{a}\tau_2\\
\text{\rm s.t.} & \lambda,\tau_1,\tau_2\in\R_+,~\varrho\in\R\\
& \varphi_j,\psi_j\in\R_+,~\gamma_j\in\R &\forall j\in[N]\\
&\bz_{l,j}\in\R^d,~\bv_{l,j}\in\R^d &\forall j\in[N], l\in[L]\\
& \lambda \epsilon-\lambda\sum_{j=1}^Nw_jd_j+\tau_1-\tau_2+\sum_{j=1}^N \varphi_j w_j=0\\
&\gamma_{j}+\lambda d_j-\varrho-\varphi_{j}+\psi_{j}=0 &\forall j\in[N]\\
&(-\ell_l)^*(\bz_{l,j}-\bv_{l,j})+\sigma_{\mathcal N}(\bv_{l,j})-\bz_{l,j}^\top \widehat{\bxi}_j\le \gamma_j &\forall j\in[N],l\in[L]\\
&\|\bz_{l,j}\|_*\le \lambda &\forall j\in[N], l\in[L].
% & \sup _{\bxi \in \mathcal N} \left\{\ell(\bxi)-\lambda  c(\bxi,\widehat{\bxi}_{j})\right\} \leq \gamma_j & \forall j \in [N].
\end{array}
\end{align*}
Moreover, if Assumption \ref{ass:lossfun2} holds and $\operatorname{cl}(\mathcal N^c)=\Xi$ (which holds whenever $\mathcal N$ has empty interior in $\Xi$, as in singleton conditioning on a continuous support), then $d_j=0$ for all $j\in[N]$, and the above problem reduces to
\begin{align*}%\tag{${\rm P}^{\rm III}$}\label{prob-barycenterB3}
\begin{array}{lll}\inf & 
\varrho-\frac{1}{b}\tau_1+\frac{1}{a}\tau_2\\
\text{\rm s.t.} & \lambda,\tau_1,\tau_2\in\R_+,~\varrho\in\R\\
& \varphi_j,\psi_j\in\R_+,~\gamma_j\in\R &\forall j\in[N]\\
&\bz_{l,j}\in\R^d,~\bomega_{l,j}\in\R_+^m &\forall j\in[N], l\in[L]\\
& \lambda \epsilon+\tau_1-\tau_2+\sum_{j=1}^N \varphi_j w_j=0\\
&\gamma_{j}-\varrho-\varphi_{j}+\psi_{j}=0 &\forall j\in[N]\\
&t_l+\bomega_{l,j}^{\top}\bg-\bz_{l,j}^{\top}\left(\widehat{\bxi}_{j}-(\bx_0,0,\dots,0)\right)\le \gamma_j &\forall j\in[N],l\in[L]\\
&-\bz_{l,j}^{\mathcal Y}+\bomega_{l,j}^{\top}C=\bs_l  &\forall j\in[N],l\in[L]\\
&\|\bz_{l,j}\|_*\le \lambda &\forall j\in[N],l\in[L]. 
\end{array}
\end{align*}
\end{proposition}

% {\color{black}
% \section{DRO for unconditional problems}

% Formulation I (intersection-based):

% \begin{align*}
% \begin{array}{lll}\inf & \sum_{k=1}^K \varepsilon_k \lambda_k+\sum_{k=1}^K \sum_{j=1}^{N_k} w_{k, j} \gamma_{k, j} & \\ 
% \text { s.t. } & \lambda_k \in \mathbb{R}_{+}, \bgamma_k \in \mathbb{R}^{N_k} & \forall k \in[K] \\ 
% &t_l+\bomega_{l,\balpha}^{\top}\bg-\sum_{k=1}^K\bz_{l,k,\balpha_k}^{\top}\widehat{\bxi}_{k,\alpha_k}\le \sum_{k=1}^K \gamma_{k,\alpha_k} &\forall \balpha\in\mathcal A, l\in[L]\\
% &-\sum_{k=1}^K\bz_{l,k,\balpha_k}+\bomega_{l,\balpha}^{\top}C=\bs_l &\forall \balpha\in\mathcal A, l\in[L]\\
% & \|\bz_{l,k,\alpha_k}\|_*\le \lambda_k &\forall \balpha\in\mathcal A, k\in[K], l\in[L].
% \end{array}
% \end{align*}

% \noindent Formulation II (weighted-based):

% \begin{align*}
% \begin{array}{lll}\inf & \varepsilon \lambda+\sum_{k=1}^K \sum_{j=1}^{N_k} w_{k, j} \gamma_{k, j} & \\ 
% \text { s.t. } & \lambda \in \mathbb{R}_{+}, \bgamma_k \in \mathbb{R}^{N_k} & \forall k \in[K] \\ 
% &t_l+\bomega_{l,\balpha}^{\top}\bg-\sum_{k=1}^K\bz_{l,k,\balpha_k}^{\top}\widehat{\bxi}_{k,\alpha_k}\le \sum_{k=1}^K \gamma_{k,\alpha_k} &\forall \balpha\in\mathcal A, l\in[L]\\
% &-\sum_{k=1}^K\bz_{l,k,\balpha_k}+\bomega_{l,\balpha}^{\top}C=\bs_l &\forall \balpha\in\mathcal A, l\in[L]\\
% & \|\bz_{l,k,\alpha_k}\|_*\le \lambda \theta_k &\forall \balpha\in\mathcal A, k\in[K], l\in[L].
% \end{array}
% \end{align*}
% }

\end{document}